\pdfoutput=1
\documentclass{amsart}

      \title[On the Adams isomorphism   for equivariant orthogonal spectra]
            {On the Adams isomorphism\\ for equivariant orthogonal spectra}

     \author{Holger Reich}
    \address{Institut f\"ur Mathematik, Freie Universit\"at Berlin, Germany}
      \email{\hemail{holger.reich@fu-berlin.de}}
    \urladdr{\hurl{mi.fu-berlin.de/math/groups/top/members/Professoren/reich.html}}

     \author{Marco Varisco}
    \address{Department of Mathematics and Statistics, University at Albany, SUNY, USA}
      \email{\hemail{mvarisco@albany.edu}}
    \urladdr{\hurl{albany.edu/~mv312143/}}

       \date{September 22, 2015}

\usepackage{amssymb}
\usepackage{mathtools} \mathtoolsset{showonlyrefs}
\usepackage{enumitem}
\usepackage{etoolbox}
\usepackage{adjustbox}
\usepackage{pifont}
\usepackage{mathrsfs}
\usepackage{stmaryrd}
\usepackage{pdflscape}
\usepackage{tikz}      \usetikzlibrary{cd}
\usepackage[shortcuts]{extdash}
\usepackage[pagebackref, pdfinfo={
      Title={On the Adams isomorphism for equivariant orthogonal spectra},
     Author={Holger Reich and Marco Varisco},
    Subject={MSC2010: 55P42, 55P91},
   Keywords={Adams isomorphism, equivariant stable homotopy theory}
}]{hyperref}
\usepackage[alphabetic, backrefs, msc-links, nobysame]{amsrefs}

\makeatletter
\renewcommand{\tocsection}[3]{%
  \indentlabel{\@ifnotempty{#2}{\makebox[2.25em][r]{\ignorespaces#1 #2.\quad}}}#3}
\makeatother

\newcommand*{\DOI}[1]{\href{http://dx.doi.org/#1}{#1}}
\newcommand*{\MSC}[1]{\href{http://www.ams.org/msc/msc2010.html?t=#1}{#1}}
\newcommand*{\hurl}[2][www.]{\href{http://#1#2}{\nolinkurl{#2}}}
\newcommand*{\hemail}[1]{\href{mailto:#1}{\nolinkurl{#1}}}

\setlist{leftmargin=*}
\setlist[enumerate]{label=(\roman*)}

\numberwithin{equation}{section}

\theoremstyle{plain}
  \newtheorem{addendum}    [equation]{Addendum}
  \newtheorem{corollary}   [equation]{Corollary}
  \newtheorem{keylemma}    [equation]{Key Lemma}
  \newtheorem{lemma}       [equation]{Lemma}
  \newtheorem{proposition} [equation]{Proposition}
  \newtheorem{main}        [equation]{Main Theorem}
  \newtheorem{theorem}     [equation]{Theorem}

\theoremstyle{definition}
  \newtheorem{construction}[equation]{Construction}
  \newtheorem{definition}  [equation]{Definition}
  \newtheorem{example}     [equation]{Example}
  \newtheorem{fact}        [equation]{Fact}
  \newtheorem{facts}       [equation]{Facts}
  \newtheorem{remark}      [equation]{Remark}

\newcommand*{\one}  {\text{\ding{192}}} \newcommand*{\bone}  {\text{\ding{202}}}
\newcommand*{\two}  {\text{\ding{193}}} \newcommand*{\btwo}  {\text{\ding{203}}}
\newcommand*{\three}{\text{\ding{194}}} 
\newcommand*{\four} {\text{\ding{195}}} 
\newcommand*{\five} {\text{\ding{196}}} 
\newcommand*{\six}  {\text{\ding{197}}}

\newcommand*\framed[1]%
{\tikz[baseline=(char.base)]%
  {\node[shape=rectangle, draw, inner sep=2pt] (char)
   {\textsf{\scriptsize\makebox[\widthof{H}][c]{#1}}};
  }
}

\DeclareMathAlphabet{\matheurm}      {U}{eur}{m}{n}
    \SetMathAlphabet{\matheurm}{bold}{U}{eur}{b}{n}

\newcommand*{\define}[5]{%
  \ifstrequal{#2}{*}{\expandafter#1\expandafter*}{\expandafter#1}%
  \csname#4#5\endcsname{#3{#5}}
}
\forcsvlist{\define{\newcommand}{*}{\mathcal}{C}}{%
  A,B,C,D,E,F,G,H,I,J,K,L,M,N,O,P,Q,R,S,T,U,V,W,X,Y,Z,%
}
\forcsvlist{\define{\newcommand}{*}{\mathbb}{I}}{%
  A,B,C,D,E,F,G,H,I,  K,L,M,N,O,P,Q,R,S,T,U,V,W,X,Y,Z%
}                
\renewcommand*{\IJ}{\mathbb{J}}

\forcsvlist{\define{\DeclareMathOperator}{}{}{}}{act, asbl, coasbl, ASBL, COASBL, coll, cw, diag, emb, ev, hofib, id, im, incl, ind, coind, Lan, map, nat, obj, pr, pt, res, swap, untw}

\forcsvlist{\define{\DeclareMathOperator}{*}{}{}}{colim, holim}

\forcsvlist{\define{\newcommand}{*}{\matheurm}{}}{Ab, Cat, Gr, Ho, NSp, Sp, Th}

\newcommand*{\TO}[1][]{\stackrel{#1}{\longrightarrow}}
\newcommand*{\MOR}[4][]{#2\colon#3\TO[#1]#4}
\newcommand*{\AND}{\qquad\text{and}\qquad}
\newcommand*{\aND} {\quad\text{and}\quad}

\DeclarePairedDelimiterX\SET[2]{\{}{\}}{\,#1\;\delimsize\vert\;#2\,}

\newcommand*{\ds}{\displaystyle}
\newcommand*{\ts}{\textstyle}

\newcommand*{\op}{{\operatorname{op}}}
\newcommand*{\tr}{{\operatorname{tr}}}

\DeclareMathOperator*{\bigLan}{Lan}
\DeclareMathOperator*{\tensor}{\otimes}
\DeclareMathOperator*{\sma}   {\wedge}
\newcommand*{\ssma}{\!\sma\!} 
\newcommand*{\indco}{\operatorname{(co)ind}} 

\newcommand*{\adams}{A}
\newcommand*{\preadams}{\widetilde{\adams}}

\newcommand*{\spec}[1]{\mathbf{#1}} 

\newcommand*{\LL}[2][]{\CL_{#2}^{\operatorname{#1}}}

\newcommand*{\shift}{\smash{\operatorname{sh}}\vphantom{\Omega}}

\newcommand*{\fr}{r}
\newcommand*{\Q}[2][]{Q^{#2}\!}

\newcommand*{\fd}{fin.\ dim.\@}
\newcommand*{\piiso}{$\underline{\pi}_*$\=/isomorphism}
\newcommand*{\iso}{\matheurm{iso}}
\newcommand*{\semi}{\rtimes}

\newcommand*{\modu}[3]{\vphantom{#2}_{#1}{#2}_{#3}}

\begin{document}

\begin{abstract}
We give a natural construction and a direct proof of the Adams isomorphism for equivariant orthogonal spectra.
More precisely, for any finite group~$G$, any normal subgroup $N$ of~$G$, and any orthogonal $G$-spectrum~$\spec{X}$, we construct a natural map $\adams$ of orthogonal $G/N$-spectra from the homotopy $N$-orbits of~$\spec{X}$ to the derived $N$-fixed points of~$\spec{X}$, and we show that $\adams$ is a stable weak equivalence if $\spec{X}$ is cofibrant and $N$-free.
This recovers a theorem of Lewis, May, and Steinberger in the equivariant stable homotopy category, which in the case of suspension spectra was originally proved by Adams.
We emphasize that our Adams map $\adams$ is natural even before passing to the homotopy category.
One of the tools we develop is a replacement-by-$\Omega$-spectra construction with good functorial properties, which we believe is of independent interest.
\end{abstract}

\maketitle

\tableofcontents

\thispagestyle{empty}

\pagebreak


\section{Introduction}

In his seminal paper~\cite{Adams} Frank Adams proved the following surprising result.\linebreak
Suppose that~$X$ is a finite pointed CW-complex on which a finite group~$G$ acts cellularly and freely away from the base point.
Think of~$X$ as an object in the equivariant stable homotopy category, i.e., consider the equivariant suspension spectrum~$\Sigma^\infty X$.
Then in the stable homotopy category there is an isomorphism between the derived $G$-orbits and the derived $G$-fixed points of~$\Sigma^\infty X$, even though unstably $X^G$ is just a point.
The derived $G$-orbit spectrum of~$\Sigma^\infty X$ is just $\Sigma^\infty(X/G)$; the definition of the derived $G$-fixed point spectrum is more delicate, but in the free case it is identified by Adams's result.
As a consequence there is an isomorphism between the $G$-equivariant stable homotopy groups of~$X$ and the non-equivariant stable homotopy groups of the orbit space~$X/G$.

Adams's theorem was generalized by Lewis, May, and Steinberger~\cite{LMS} from suspension spectra to arbitrary ones, and became known as the \emph{Adams isomorphism}.
A precise statement is formulated in Corollary~\ref{cor:LMS}.

The purpose of this paper is to lift the Adams isomorphism from the stable homotopy category of spectra to a \emph{natural} weak equivalence in the category of orthogonal spectra.
This is the content of our Main Theorem~\ref{thm:main}.
The need for a construction of the Adams isomorphism that is natural before passing to the homotopy category arises, for example, in connection with Farrell-Jones assembly maps for topological cyclic homology of group algebras, because assembly maps cannot be defined in the homotopy category.
This application is explained at the end of the introduction.

Equivariant orthogonal spectra are a modern model for equivariant stable homotopy theory.
They were introduced and studied extensively by Mandell and May in~\cite{MM} and play an important role, for example, in the work of Hill, Hopkins, and Ravenel~\cite{HHR} on the Kervaire invariant problem.
The paper~\cite{HHR} also contains extensive foundational material on equivariant orthogonal spectra, and Schwede's lecture notes~\cite{Schwede} are another excellent resource.
We review the foundations of the theory in Sections~\ref{sec:conventions} through~\ref{sec:change-of-groups}, emphasizing the point of view of orthogonal spectra as continuous functors.
We exploit this functorial point of view not only to clarify some fundamental constructions, but also to develop new tools.

The main new tool is a replacement-by-$\Omega$-spectra construction with good functorial properties.
In order to explain this we introduce some notation and terminology.
Here and throughout the paper we let $G$ denote a finite group.
We refer to real orthogonal $G$-representations of finite or countably infinite dimension simply as $G$-representations.
$G$-representations form a topological $G$-category~$\LL{G}$ with morphism spaces~$\CL(\CU,\CV)$ given by not necessarily equivariant linear isometries, equipped with the conjugation $G$-action; see Sections \ref{sec:conventions} and~\ref{sec:Thom}.
We denote by~$\Sp_G$ the topological $G$-category of orthogonal $G$-spectra, and define $G$-universes and \piiso s as usual; see Sections \ref{sec:orth} and~\ref{sec:homotopy-theory}.
An orthogonal $G$-spectrum is called good if the structure maps are closed embeddings (see Definition~\ref{def:good}), and it is a $G$-$\Omega$-spectrum if the adjoints of the structure maps are $G$-weak equivalences (see Definition~\ref{def:Omega}).

We emphasize that $\Sp_G$ is defined without reference to the choice of a universe.
However, we use universes to construct replacements by $G$-$\Omega$-spectra,
and this replacement construction depends functorially on the choice of a universe, before passing to the homotopy category.
This is the content of the following theorem, which is proved in Section~\ref{sec:Q}.

\begin{theorem}
\label{thm:Q}
For any orthogonal $G$-spectrum~$\spec{X}$ and any $G$-representation~$\CU$, there is an orthogonal $G$-spectrum~$\Q[h]{\CU}(\spec{X})$ together with a natural transformation
\(
\MOR{\fr}{\spec{X}}{\Q[h]{\CU}(\spec{X})}
\)
such that:
\begin{enumerate}

\item\label{i:fun}
The assignment
\[
(\CU,\spec{X})\longmapsto\Q[h]{\CU}(\spec{X})
\]
extends to a $G$-functor $\LL{G}\times\Sp_G\TO\Sp_G$ which is continuous in each variable;

\item\label{i:repl}
If $\spec{X}$ is good then
\(
\MOR{\fr}{\spec{X}}{\Q[h]{\CU}(\spec{X})}
\)
is a $\underline{\pi}_*$-equivalence;

\item\label{i:Omega}
If $\spec{X}$ is good and $\CU$ is a complete $G$-universe then $\Q[h]{\CU}(\spec{X})$ is a $G$-$\Omega$-spectrum;

\item\label{i:res-Q}
Given another finite group~$\Gamma$ and a group homomorphism $\MOR{\alpha}{\Gamma}{G}$ there is a natural $\Gamma$-isomorphism
\[
\res_\alpha\Q{\CU}(\spec{X})\TO[\cong]\Q{\res_\alpha\CU}(\res_\alpha\spec{X})
\,.
\]
\end{enumerate}
\end{theorem}

We refer to~$\Q[h]{}$ as a \emph{bifunctorial} replacement construction.
The notation is inspired by the traditional use of the letter~$Q$ for the $\Omega^\infty \Sigma^\infty$ construction for spaces; see e.g.~\cite{Carlsson}*{Section~2}.
Notice that part~\ref{i:fun} of Theorem~\ref{thm:Q} implies the following corollary; we think of the maps~\eqref{eq:fun-in-U} as \emph{change of universe} maps.

\begin{corollary}
\label{cor:fun-in-U}
For any orthogonal $G$-spectrum~$\spec{X}$ and for any pair of $G$\=/representations $\CU$ and~$\CV$ there is a $G$-map of orthogonal $G$-spectra
\begin{equation}
\label{eq:fun-in-U}
\MOR{\Phi_{\CU,\CV}}{\CL(\CU,\CV)_+\sma\Q[h]{\CU}(\spec{X})}{\Q[h]{\CV}(\spec{X})}
\end{equation}
such that, for all $G$-representations $\CU$, $\CV$, and~$\CW$, the diagram
\begin{equation}
\label{eq:fun-in-U-assoc}
\begin{tikzcd}[column sep=large]
\CL(\CV,\CW)_+\sma\CL(\CU,\CV)_+\sma\Q[h]{\CU}(\spec{X})
\arrow{r}{\id\sma\Phi_{\CU,\CV}}
\arrow{d}[swap]{\circ\sma\id}
&
\CL(\CV,\CW)_+\sma\Q[h]{\CV}(\spec{X})
\arrow{d}{\Phi_{\CV,\CW}}
\\
\CL(\CU,\CW)_+\sma\Q[h]{\CU}(\spec{X})
\arrow{r}[swap]{\Phi_{\CU,\CW}}
&
\Q[h]{\CW}(\spec{X})
\end{tikzcd}
\end{equation}
commutes, and the composition
\begin{equation}
\label{eq:fun-in-U-unit}
S^0\sma\Q[h]{\CU}(\spec{X})
\xrightarrow{i\sma\id}
\CL(\CU,\CU)_+\sma\Q[h]{\CU}(\spec{X})
\xrightarrow{\Phi_{\CU,\CU}}
\Q[h]{\CU}(\spec{X})
\end{equation}
is the canonical isomorphism, where $i$ is the inclusion of the identity.
\end{corollary}

The bifunctorial replacement is a key step in our construction of transfer maps (Section~\ref{sec:transfer}) and of the Adams map (Section~\ref{sec:construction}).
To formulate our main theorem we need some more terminology.
Let $N$ be a normal subgroup of the finite group~$G$.
We denote by~$\CF(N)$ the family of subgroups $H$ of~$G$ such that $H\cap N=1$, and let $E\CF(N)$ be the universal space for~$\CF(N)$.
We say that an orthogonal $G$-spectrum~$\spec{X}$ is $N$-free if the projection
\begin{equation}
\label{eq:N-free}
E\CF(N)_+\sma\spec{X}\TO\spec{X}
\end{equation}
is a \piiso; compare~\cite{MM}*{Definitions IV.6.1(ii) on page~69 and~VI.4.2 on page~100}.
For an orthogonal $G$-spectrum~$\spec{X}$ the symbols $\spec{X}/N$ and~$\spec{X}^N$ denote the naive $N$-orbits and the naive $N$-fixed points, defined by applying levelwise the corresponding constructions for spaces; see Section~\ref{sec:change-of-groups}.
No fibrant nor cofibrant replacement is hidden in the notation and in our setup there are no change-of-universe functors.

\begin{main}
\label{thm:main}
Let $G$ be a finite group and $N$ be a normal subgroup of~$G$.
Let $\CU$ be a complete $G$-universe.
Then for any orthogonal $G$-spectrum~$\spec{X}$ there is an equivariant map of orthogonal $G/N$-spectra
\[
\MOR{\adams}{E\CF(N)_+\sma_N\spec{X}}{\Q[h]{\CU}(\spec{X})^N}
\]
such that:
\begin{enumerate}
\item\label{i:adams-natural}
$\adams$ is natural in~$\spec{X}$;
\item\label{i:adams-iso}
if $\spec{X}$ is good and $N$-free then $\adams$ is a \piiso.
\end{enumerate}
\end{main}

We call $\adams$ the \emph{Adams map}.
We emphasize that $\adams$ is natural even before passing to the stable homotopy category.
Moreover, $\adams$ is defined for all orthogonal $G$-spectra, even without the $N$-free and good assumptions, and without any fibrancy or cofibrancy assumptions.
Main Theorem~\ref{thm:main} is proved in Section~\ref{sec:proof}.

In the special case $N=G$, the universal space $E\CF(G)$ is just~$EG$, the universal cover of the classifying space of~$G$, and the Adams map is a natural map of non-equivariant orthogonal spectra
\begin{equation*}
\label{eq:adams-special}
\MOR{\adams}{EG_+\sma_G\spec{X}}{\Q[h]{\CU}(\spec{X})^G}
.
\end{equation*}
In the general case, after forgetting the $G/N$-action the source of the Adams map is the $N$-homotopy orbits of~$\spec{X}$ considered as an orthogonal $N$-spectrum; see also Lemma~\ref{lem:well-known}\ref{i:sma-F(N)}.
In fact, $\res_{N\leq G}E\CF(N) = EN$ and
\[
\res_{1\leq G/N}\Bigl(E\CF(N)_+\sma_N\spec{X}\Bigr)
\cong
\Bigl(\res_{N\leq G}\bigl(E\CF(N)_+\sma\spec{X}\bigr)\Bigr)/N
\cong
EN_+\sma_N\res_{N\leq G}\spec{X}
\,.
\]
If $\spec{X}$ is good then $\Q[h]{\CU}(\spec{X})$ is a $G$-$\Omega$-spectrum $\underline{\pi}_*$-isomorphic to~$\spec{X}$ by Theorem~\ref{thm:Q}\ref{i:repl} and~\ref{i:Omega}, and therefore the target of the Adams map are the derived (also known as categorical, as opposed to the naive or geometrical) fixed points of~$\spec{X}$.

From Main Theorem~\ref{thm:main} one recovers the classical Adams isomorphism as follows.
Mandell and May construct in~\cite{MM}*{Section~III.4 on pages~47--51} the stable model structure on the topological category of orthogonal $G$-spectra~$\Sp^G$.
The weak equivalences are the \piiso s,
the fibrant objects are the orthogonal $G$-$\Omega$-spectra,
and cofibrant implies good by Lemma~\ref{lem:good}\ref{i:cofib=>good}.
We denote by $\Ho\Sp^G$ the corresponding homotopy category and by $[-,-]^G$ the morphism sets in~it.
By~\cite{MM}*{Theorems IV.1.1 on page~59 and IV.2.9(iv) on page~63} $\Ho\Sp^G$ is equivalent to the stable homotopy category of $G$-spectra in the sense of~\cite{LMS}.

Now let $N$ be a normal subgroup of~$G$ and let $\CU$ be a complete $G$-universe.
Then there is also a different stable model structure on~$\Sp^G$, that we call the $N$-trivial stable model structure, where the weak equivalences are the $\CU^N$-\piiso s, i.e., the maps inducing isomorphisms of the homotopy groups defined as in~\eqref{eq:pi} but with $\CU$ replaced by~$\CU^N$.
Notice that we follow the point of view explained in~\cite{MM}*{paragraph before Theorem~V.1.7 on page~76}, and we consider two different model structures on the same category, instead of changing also the underlying category.

\begin{corollary}
\label{cor:LMS}
Let $G$ be a finite group and $N$ be a normal subgroup of~$G$.
Let $\CU$ be a complete $G$-universe.
For any orthogonal $G$-spectrum~$\spec{X}$,
if $\spec{X}$ is cofibrant in the $N$-trivial stable model structure and the projection $\MOR{\pr}{E\CF(N)_+\sma\spec{X}}{\spec{X}}$ is a weak equivalence in the $N$-trivial stable model structure,
then in the stable homotopy category of $G/N$-spectra there is an isomorphism
\begin{equation}
\label{eq:LMS}
\spec{X}/N
\TO[\simeq]
\Q[h]{\CU}(\spec{X})^N
.
\end{equation}
Moreover, for every cofibrant $G/N$-spectrum~$\spec{Y}$ there is an isomorphism
\begin{equation}
\label{eq:LMS-adj}
\bigl[\spec{Y},\spec{X}/N\bigr]^{G/N}
\cong
\bigl[p^*\spec{Y},\spec{X}\bigr]^G
,
\end{equation}
where $\MOR{p}{G}{G/N}$ denotes the projection.
\end{corollary}

Corollary~\ref{cor:LMS} was originally proved by Lewis, May, and Steinberger~\cite{LMS}*{Theorem~II.7.1 on page~97}, and not just for finite but also for compact Lie groups.
In the special case when $G$ is finite and $\spec{X}=\Sigma^\infty X$ is the suspension spectrum of a finite $N$-free $G$-CW-complex~$X$, it was proved by Adams~\cite{Adams}*{Theorem~5.4 on page~500}.
Mandell and May~\cite{MM}*{Chapter~VI on pages 88--100} prove that the equivalence recalled above between the homotopy categories of orthogonal $G$-spectra and $G$-spectra in the sense of~\cite{LMS} respects orbits and fixed points, and then use this to \emph{``transport the Adams isomorphism, which is perhaps the deepest foundational result in equivariant stable homotopy theory, from $G$-spectra to [\ldots]~orthogonal $G$-spectra''}~\cite{MM}*{first paragraph of Section~VI.4 on page~99}.
We emphasize though that they transport the Adams isomorphism only into the stable homotopy category of orthogonal $G$-spectra, whereas we realize it as a  map of orthogonal $G$-spectra that is natural even before passing to the homotopy category, and that is defined for all spectra.

\begin{proof}[Proof of Corollary~\ref{cor:LMS} from Main Theorem~\ref{thm:main}]
The proof relies on three Quillen adjunctions proved in~\cite{MM}.
The assumptions imply that the projection $\pr$ is a weak equivalence between cofibrant objects in the $N$-trivial stable model structure.
It follows that applying any left Quillen adjoint to~$\pr$ produces a weak equivalence between cofibrant objects.
We use this observation twice.

First, by~\cite{MM}*{Proposition~V.3.12 on page~81}, taking $N$-orbits gives a left Quillen adjoint (to~$p^*$) from the $N$-trivial stable model structure on~$\Sp^G$ to the usual stable model structure on~$\Sp^{G/N}$.
Therefore $\pr$ induces a \piiso\ 
\begin{equation}
\label{eq:zig}
E\CF(N)_+\sma_N\spec{X}\TO\spec{X}/N
\end{equation}
of orthogonal $G/N$-spectra.

Second, by~\cite{MM}*{Corollary~V.1.8 on page~77}, the identity is a left Quillen adjoint (to itself) from the $N$-trivial to the usual stable model structure on~$\Sp^G$.
Therefore $\spec{X}$ is $N$-free according to the definition in~\eqref{eq:N-free}, and $\spec{X}$ is cofibrant also in the usual stable model structure, and so in particular $\spec{X}$ is good by Lemma~\ref{lem:good}\ref{i:cofib=>good}.
Then Main Theorem~\ref{thm:main} applies and the Adams map
\begin{equation}
\label{eq:zag}
\MOR{\adams}{E\CF(N)_+\sma_N\spec{X}}{\Q[h]{\CU}(\spec{X})^N}
\end{equation}
is a \piiso\ of orthogonal $G/N$-spectra.
The isomorphism~\eqref{eq:LMS} is then given in the stable homotopy category of $G/N$-spectra by 
\eqref{eq:zig} and~\eqref{eq:zag}.

Finally, by~\cite{MM}*{Proposition~V.3.10 on page~81}, taking $N$-fixed points gives a right Quillen adjoint (to~$p^*$) from the usual stable model structure on~$\Sp^G$ to the usual stable model structure on~$\Sp^{G/N}$.
This adjunction, together with \eqref{eq:LMS} and the \piiso\ $\spec{X}\TO\Q[h]{\CU}(\spec{X})$ from~Theorem~\ref{thm:Q}\ref{i:repl}, implies~\eqref{eq:LMS-adj}.
\end{proof}

We end this introduction by explaining our main motivation for this work.
The Adams isomorphism plays an important role in the theory of topological cyclic homology, as we proceed to explain.
Suppose that $\spec{A}$ is a ring or a ring spectrum, and fix a prime number~$p$.
Topological cyclic homology at~$p$ is defined as
\[
TC(\spec{A};p) = \holim_{R,F} THH(\spec{A})^{C_{p^n}}
,
\]
where $THH(\spec{A})$, topological Hochschild homology, is an orthogonal $S^1$-spectrum, and for any integer $n\geq1$ there are maps
\[
\MOR{F}{THH(\spec{A})^{C_{p^{n}}}}{THH(\spec{A})^{C_{p^{n-1}}}}
\aND
\MOR{R}{THH(\spec{A})^{C_{p^{n}}}}{THH(\spec{A})^{C_{p^{n-1}}}}
\]
between the underlying fixed point spectra, with $C_{p^n}$ denoting the finite cyclic group of order~$p^n$ seen as a subgroup of~$S^1$.
Moreover, there is the cyclotomic trace map $\MOR{trc}{K(\spec{A})}{TC(\spec{A};p)}$ from the algebraic $K$-theory spectrum of~$\spec{A}$ to topological cyclic homology.
The cyclotomic trace has been used very successfully for computations of algebraic $K$-theory; see for example~\citelist{\cite{BHM} \cite{HM-annals}}.

An essential input for $TC$ computations is the fundamental fibration sequence, that identifies the homotopy fiber of~$R$ with
\[
{EC_{p^n}}_+\sma_{C_{p^n}} THH(\spec{A})
\,;
\]
see for example~\cite{HM-top}*{Proposition~2.1 on page~33}.
The key step in this identification is the Adams isomorphism.
Using the results of this paper we can lift the fundamental fibration sequence from the stable homotopy category to the category of orthogonal spectra.

In joint work with L\"uck and Rognes~\cite{LRRV} we compare via the cyclotomic trace the Farrell-Jones assembly map in algebraic $K$-theory to the assembly map in~$TC$, and prove new rational injectivity results for both assembly maps.
We refer to~\cite{LR} for a comprehensive survey of assembly maps and isomorphism conjectures.
Since the assembly maps are constructed by taking homotopy colimits over the orbit category, it is essential that the fundamental fibration sequence, and therefore the Adams isomorphism, is natural before passing to the homotopy category, because homotopy colimits and hence assembly maps cannot be defined in the homotopy category.


\subsection*{Structure of the article}
The main new results are contained in Sections~\ref{sec:Q-fin} through~\ref{sec:N=G}.
The bifunctorial replacement Theorem~\ref{thm:Q} is proved in Section~\ref{sec:Q} after some preparations in Section~\ref{sec:Q-fin}, and some important consequences of it are established in Sections~\ref{sec:asbl-coasbl} and~\ref{sec:free}.
Sections~\ref{sec:transfer} and~\ref{sec:wirth} discuss the construction and properties of transfer maps, using the bifunctorial replacements of the previous sections.
All these tools come together in the final three sections, which are devoted to the Adams isomorphism.
The Adams map is constructed in Section~\ref{sec:construction}, where also part~\ref{i:adams-natural} of Main Theorem~\ref{thm:main} is established.
Part~\ref{i:adams-iso} is then proved in Section~\ref{sec:proof}, and a shortcut of the proof in the special case when $N=G$ is explained in Section~\ref{sec:N=G}.

In Sections~\ref{sec:conventions} through~\ref{sec:change-of-groups} we review the foundations of equivariant orthogonal spectra.
Experienced readers might skip these sections and refer back to them only when needed; we draw their attention in particular to Remark~\ref{rem:notation}, where our departures from the notations and conventions of~\cites{MM, HHR} are explained, and to Definition~\ref{def:good} and Lemma~\ref{lem:good} about good and almost good orthogonal spectra. 
For less experienced readers, on the other hand, in these initial sections we make explicit and collect in one place several facts that we rely upon and that are scattered (sometimes less explicitly) in the existing literature.


\subsection*{Acknowledgments}
We would like to thank Stefan Schwede and Irakli Patchkoria for insightful conversations and comments, the referee for helping us improve the exposition, and CRC~647 in Berlin for financial support.


\section{Conventions}
\label{sec:conventions}

We begin by fixing some notation and terminology.
We denote by $\CT$ the category of pointed, compactly generated, and weak Hausdorff spaces (see e.g.~\cite{tD-top}*{Section~7.9 on pages~186--195} and~\cite{Strickland}), which from now on we simply refer to as \emph{pointed spaces}.
Given pointed spaces $X$ and~$Y$ we write $\CT(X,Y)$ or $\map(X,Y)$ for the pointed space of pointed maps from~$X$ to~$Y$.
With respect to the smash product, $\CT$ is a closed symmetric monoidal category, with unit object~$S^0$.
We refer to categories and functors enriched over~$\CT$ as \emph{topological categories} and \emph{continuous functors}.
For the basic definitions about enriched categories see for example~\cite{MacLane}*{Section VII.7 on pages~180--181}.

Everywhere in this paper $G$ denotes a finite group.
A \emph{pointed $G$-space} is a pointed 
space equipped with a continuous left action of~$G$ fixing the base point.
We denote by~$\CT^G$ the topological category whose objects are all pointed $G$-spaces and with morphism spaces $\CT^G(X,Y)$ the subspaces of~$\CT(X,Y)$ consisting of the $G$-equivariant maps.
Endowing the smash product with the diagonal $G$-action, $S^0$ with the trivial action, and $\CT(X,Y)$ with the conjugation action, the category $\CT^G$ is closed symmetric monoidal.
A \emph{topological $G$-category} is for us a category enriched over~$\CT^G$, and a \emph{continuous $G$-functor} is a $\CT^G$-enriched functor.
Pointed $G$-spaces also form a topological $G$-category, that we denote~$\CT_G$, with morphisms the pointed $G$-spaces
\(
\CT_G(X,Y)=\CT(X,Y)
\)
of not necessarily equivariant maps.
Notice that $\obj\CT^G=\obj\CT_G$ and $\CT^G(X,Y)=\CT_G(X,Y)^G$.


\section{Thom spaces of complement bundles}
\label{sec:Thom}

Throughout this paper the expression \emph{$G$-representation} refers to a real, orthogonal $G$-representation of finite or countably infinite dimension, always equipped with the colimit topology from the finite dimensional subrepresentations.
Given $G$-representations $U$ and~$V$, we denote by $\CL(U,V)$ the space of all linear isometric maps, endowed with the conjugation action of~$G$ and with the subspace topology $\CL(U,V)\subset\CL(U,V)_+\subset\CT(U,V)$.
We denote by $\LL{G}$ the topological $G$-category with objects the $G$\=/representations and with morphisms the pointed $G$-spaces $\CL(U,V)_+$.
We let $\LL[fin]{G}$ be the full subcategory of the finite dimensional representations, and $\iso\LL[fin]{G}$ the subcategory of all isomorphisms in~$\LL[fin]{G}$.
We often use the symbols $U$, $V$, \ldots\ for \fd\ representations, and $\CU$, $\CV$, \ldots\ for infinite dimensional ones.
Given a $G$-representation~$\CU$ we denote by~$\CP(\CU)$ the poset of all finite dimensional $G$-subrepresentations of~$\CU$ ordered by inclusion, and we consider it as a discrete subcategory of~$\LL[fin]{G}$.

The \fd\ $G$-representations also form another topological $G$-category with respect to bigger morphism $G$-spaces, the Thom spaces of certain bundles that we now proceed to explain.
Let $U$ and~$V$ be \fd\ $G$-representations and set $k=\dim V - \dim U$.
If $k\geq0$ let~$\Gr_k(V)$ be the Grassmannian of $k$-dimensional subspaces of~$V$, equipped with the obvious $G$-action.
Consider the map
\[
\MOR{c}{\CL(U,V)}{\Gr_k(V)}
\,,\quad
f\mapsto V-f(U)
\,,
\]
where $V-f(U)$ denotes the orthogonal complement of~$f(U)$ in~$V$.
Clearly, $c$ is $G$-equivariant.
Now define~$\xi(U,V)$ to be the pullback along~$c$ of the tautological $k$-dimensional $G$-vector bundle over~$\Gr_k(V)$.
More explicitly,
\[
\xi(U,V) =
\SET{(f,v)\in\CL(U,V)\times V}{f(U)\perp v}
\,,
\]
where the projection $\xi(U,V)\TO\CL(U,V)$ is just the projection onto the first factor, and the $G$-action sends $(f,v)$ to~$(gf(g^{-1}-),gv)$.

\begin{definition}
Given \fd\ $G$-representations $U$ and~$V$, define
\[
\Th(U,V)=\Th\bigl(\xi(U,V)\TO\CL(U,V)\bigr)
\]
to be the Thom space of~$\xi(U,V)$, i.e, the quotient of the fiberwise one-point compactification of~$\xi(U,V)$ where all the points at infinity are identified and taken as the base point.
\end{definition}

Notice that $\Th(U,V)$ is a pointed $G$-space.
The zero section of~$\xi(U,V)$ gives a pointed $G$-map
\begin{equation}
\label{eq:L-to-Th}
\MOR{z=z_{U,V}}{\CL(U,V)_+}{\Th(U,V)}
\,,
\end{equation}
which is a homeomorphism if and only if~$k=0$.
In~the case where $k<0$, then $\CL(U,V)=\emptyset$ and $\Th(U,V)=\pt$.

In the special case where $U\subseteq V$ is a $G$-subrepresentation of~$V$ the fiber of~$\xi(U,V)$ over the inclusion~$\MOR{i}{U}{V}$ is~$V-U$, and therefore we obtain a pointed $G$-map
\begin{equation}
\label{eq:S->Th-sub}
\MOR{a=a_{U\subseteq V}}{S^{V-U}}{\Th(U,V)}
\end{equation}
with $a(0)=z(i)$, where $S^{V-U}$ denotes the one-point compactification of~$V-U$ pointed at infinity.
More generally, given \fd\ $G$-representations $U$ and~$V$, the fiber of~$\xi(V,U\oplus V)$ over the $G$-isometry~$\MOR{i}{V}{U\oplus V}$, $v\mapsto(0,v)$, is naturally isomorphic to~$U$ via the map $u\mapsto(i,(u,0))$, and therefore we obtain a pointed $G$-map
\begin{equation}
\label{eq:S->Th-plus}
\MOR{a=a_{U,V}}{S^U}{\Th(V,U\oplus V)}
\,.
\end{equation}

Given \fd\ $G$-representations $U$, $V$, and $W$, composition of isometries induces a $G$-map of bundles
\[
\MOR{\circ}{\xi(V,W)\times\xi(U,V)}{\xi(U,W)}
\,,\quad
\bigl((g,w),(f,v)\bigr)\mapsto(gf,w+g(v))
\]
and therefore a $G$-map
\begin{equation}
\label{eq:Th-comp}
\MOR{\circ}{\Th(V,W)\sma\Th(U,V)}{\Th(U,W)}
\,.
\end{equation}
The composition maps in~\eqref{eq:Th-comp} are associative and unital, where the identity in $\Th(U,U)$ is given by~$(\id_U,0)=z(\id_U)$.

Thus we get a topological $G$-category $\Th_G$ whose objects are the \fd\ $G$\=/representations, i.e., $\obj\Th_G=\obj\LL[fin]{G}$, and morphisms pointed $G$-spaces $\Th(U,V)$.
The identity on objects together with the zero section maps~$z$ in~\eqref{eq:L-to-Th} defines a continuous $G$-functor $\MOR{z}{\LL[fin]{G}}{\Th_G}$.
We write~$\Th^G$ for the topological category with $\obj\Th^G=\obj\Th_G$ and morphism spaces $\Th^G(U,V)=\Th(U,V)^G$.

\begin{remark}
\label{rem:notation}
A warning about notation is in order here.
Our category $\iso\LL[fin]{G}$ is denoted~$\mathscr{I}_G$ in~\cite{MM}*{Definition~II.2.1 on page~32},
whereas $\Th_G$ is denoted~$\mathscr{J}_G$ in~\cite{MM}*{Definition~II.4.1 on page~35} and in~\cite{HHR}*{Definition~A.10 on page~123}.
Notice also that in the notation $\mathscr{I}_G$ or~$\mathscr{J}_G$ of~\cite{MM} there is an implicit choice of indexing $G$-representations; see~\cite{MM}*{Definition~II.1.1 on page~30 and Variant~II.2.2 on page~32}.
For us, $\iso\LL[fin]{G}=\mathscr{I}_G^{\mathscr{A}\ell\ell}$ and $\Th_G=\mathscr{J}_G^{\mathscr{A}\ell\ell}$.
\end{remark}

In the special case where $G$ is the trivial group, we omit~$G$ from the notation, and write $\LL{}$ and $\Th$ for the corresponding topological categories.
By considering every finite dimensional inner product space as a trivial $G$-representation, we obtain a continuous functor
\begin{equation}
\label{eq:Th->Th^G}
\MOR{\iota}{\Th}{\Th^G}
,
\end{equation}
which can be thought of as the inclusion of the full subcategory of the trivial $G$\=/representations.
We introduce the notation $\Th^\tr_G$ for the category $\Th$ when viewed as a topological $G$-category with trivial $G$-action on the morphism spaces.
Then the functor in~\eqref{eq:Th->Th^G} becomes a continuous $G$-functor
\begin{equation}
\label{eq:Th->Th_G}
\MOR{\iota}{\Th^\tr_G}{\Th_G}
\,.
\end{equation}
This notation may seem overkill, but it is useful later when discussing naive and non-naive orthogonal $G$-spectra, see Theorem~\ref{thm:naive}, and change of group constructions, see Section~\ref{sec:change-of-groups}.

We conclude this section with some important properties of the category~$\Th_G$ that are often used later.

\begin{lemma}[Untwisting and composition]
\label{lem:untw-comp}
For all \fd\ $G$-representations $U$ and~$V$ there are \emph{untwisting} $G$\=/homeomorphisms
\[
\untw\colon
\Th(U,V)\sma S^U
\TO[\cong]
S^V\sma\CL(U,V)_+
\,,
\]
which are compatible with composition, i.e., for all \fd\ $G$-representations $U$, $V$, and~$W$, the diagram
\begin{equation}
\label{eq:untw-comp}
\begin{tikzcd}[column sep=large]
\Th(V,W)\sma\Th(U,V)\sma S^U
\arrow{d}[swap]{\id\sma\untw}
\arrow{r}{\circ\sma\id}
&
\Th(U,W)\sma S^U
\arrow{dd}{\untw}
\\
\Th(V,W)\sma S^V\sma\CL(U,V)_+
\arrow{d}[swap]{\untw\sma\id}
\\
S^W\sma\CL(V,W)_+\sma\CL(U,V)_+
\arrow{r}[swap]{\id\sma\circ}
&
S^W\sma\CL(U,W)_+
\end{tikzcd}
\end{equation}
commutes.
\end{lemma}

\begin{proof}
The untwisting homeomorphism comes from well-known properties of Thom spaces, because
\[
\Th(U,V)\sma S^U
\cong
\Th\bigl(\xi(U,V)\oplus\underline{U}\bigr)
\cong
\Th\bigl(\underline{V})
\cong
S^V\sma\CL(U,V)_+
\,.
\]
Here $\underline{U}$ and~$\underline{V}$ denote the trivial bundles with fibers $U$ and~$V$ over~$\CL(U,V)$.
The second isomorphism is induced from the vector bundle isomorphism
\[
\xi(U,V)\oplus\underline{U} \TO[\cong] \underline{V}
\,,\quad
\bigl((f,v),u\bigr)\mapsto v+f(u)
\,.
\] 
Thus the untwisting homeomorphism is given by
\[
\MOR{\untw}{\Th(U,V)\sma S^U}{S^V\sma\CL(U,V)_+}
\,,\quad
(f,v)\sma u
\mapsto
\bigl(v+f(u)\bigr)\sma f
\,.
\]
Using this explicit formula it is easy to check that diagram~\eqref{eq:untw-comp} commutes.
\end{proof}

\begin{lemma}
\label{lem:plus}
The assignment
\[
\MOR{-\oplus-}{\Th_G\times\Th_G}{\Th_G}
\,,\quad
(U,W)\longmapsto U\oplus W
\]
extends to a continuous $G$-functor given on morphisms by
\[
\Th(U,V)\sma\Th(W,Z)\TO\Th(U\oplus W,V\oplus Z)
\,,\quad
(f,v)\sma(g,z)\longmapsto(f\oplus g, (v,z))
\,.
\]
Moreover, when $U=0$ and $W=Z$, the diagram
\begin{equation}
\label{eq:plus-a}
\begin{tikzcd}[row sep=tiny, column sep=large]
\Th(0,V)
\arrow[equal]{d}{\wr}
\arrow{r}{-\oplus\id_W}
&
\Th(0\oplus W, V\oplus W)
\arrow[equal]{d}{\wr}
\\
S^V
\arrow{r}[swap]{a_{V,W}}
&
\Th(W, V\oplus W)
\end{tikzcd}
\end{equation}
commutes, where $a_{V,W}$ is defined in~\eqref{eq:S->Th-plus}.
If $U$ is a $G$-subrepresentation of~$V$, then the diagram
\begin{equation}
\label{eq:plus-a-sub}
\begin{tikzcd}[column sep=large]
\Th(U,V)
\arrow{r}{-\oplus\id_W}
&
\Th(U\oplus W, V\oplus W)
\\
S^{V-U}
\arrow{u}{a_{U\subseteq V}}
\arrow{r}[description]{\cong}
&
S^{V\oplus W-U\oplus W}
\arrow{u}[swap]{a_{U\oplus W\subseteq V\oplus W}}
\end{tikzcd}
\end{equation}
commutes, where the vertical maps are defined in~\eqref{eq:S->Th-sub} and the bottom horizontal isomorphism is the evident one.
\end{lemma}

\begin{proof}
This is immediate after unraveling the definitions.
\end{proof}


\section{Orthogonal \texorpdfstring{$G$}{G}-spectra}
\label{sec:orth}

\begin{definition}[Orthogonal spectra]
\label{def:orth-G-spectrum}
An \emph{orthogonal $G$-spectrum~$\spec{X}$} is a continuous $G$-functor
\[
\MOR{\spec{X}}{\Th_G}{\CT_G}
\,.
\]
\end{definition}

We emphasize that a continuous $G$-functor $\MOR{\spec{X}}{\Th_G}{\CT_G}$ consists of a pointed $G$-space $\spec{X}(U)$ for any \fd\ $G$-representation~$U$, together with pointed $G$-maps
\begin{equation}
\label{eq:CT^G-functor}
\Th(U,V)\TO\CT\bigl(\spec{X}(U),\spec{X}(V)\bigr),
\quad
(f,v)\mapsto\spec{X}(f,v)
\end{equation}
for every pair of \fd\ $G$-representations, satisfying associativity and unitality axioms.
Notice that by adjunction the maps~\eqref{eq:CT^G-functor} can equivalently be described by pointed $G$-maps
\begin{equation}
\label{eq:CT^G-functor-adj}
\MOR{\chi_{U,V}}{\Th(U,V)\sma\spec{X}(U)}{\spec{X}(V)},
\quad
(f,v)\sma x\mapsto\bigl(\spec{X}(f,v)\bigr)(x)
\end{equation}
and in this adjoint form the associativity and unitality axioms state that, for every \fd\ $G$-representations $U$, $V$, and~$W$, the diagram
\begin{equation}
\label{eq:spec-assoc}
\begin{tikzcd}[column sep=large]
\Th(V,W)\sma\Th(U,V)\sma\spec{X}(U)
\arrow{r}{\id\sma\chi_{U,V}}
\arrow{d}[swap]{\circ\sma\id}
&
\Th(V,W)\sma\spec{X}(V)
\arrow{d}{\chi_{V,W}}
\\
\Th(U,W)\sma\spec{X}(U)
\arrow{r}[swap]{\chi_{U,W}}
&
\spec{X}(W)
\end{tikzcd}
\end{equation}
commutes, and the composition
\[
S^0\sma\spec{X}(U)
\xrightarrow{a\sma\id}
\Th(U,U)\sma\spec{X}(U)
\xrightarrow{\chi_{U,U}}
\spec{X}(U)
\]
is the canonical isomorphism, where $a=a_{U\subseteq U}$ is the inclusion of the identity defined in~\eqref{eq:S->Th-sub}.

\begin{example}[Sphere and suspension spectra]
The \emph{sphere $G$-spectrum} is the orthogonal $G$-spectrum $\MOR{\spec{S}}{\Th_G}{\CT_G}$ that sends a \fd\ $G$-representation~$U$ to its one-point compactification~$S^U$, and with the maps~\eqref{eq:CT^G-functor-adj} defined as the composition
\[
\Th(U,V)\sma S^U
\xrightarrow{\untw}
S^V\sma\CL(U,V)_+
\TO[\pr]
S^V
.
\]
The commutativity of diagram~\eqref{eq:spec-assoc} follows from the compatibility of untwisting and composition; see diagram~\eqref{eq:untw-comp} in Lemma~\ref{lem:untw-comp}.
More generally, for any pointed $G$-space~$X$ the \emph{suspension $G$-spectrum} $\Sigma^\infty X$ is defined analogously, sending $U$ to~$S^U\sma X$.
\end{example}

\begin{definition}[Structure maps]
Let $\spec{X}$ be an orthogonal $G$-spectrum.
Given \fd\ $G$\=/representations $U$ and~$V$ define the \emph{structure map}~$\sigma_{U,V}$ to be the composition
\begin{equation}
\label{eq:structure-map}
\sigma_{U,V}\colon
S^U\sma\spec{X}(V)
\xrightarrow{a_{U,V}\sma\id}
\Th(V,U\oplus V)\sma\spec{X}(V)
\xrightarrow{\chi_{V,U\oplus V}}
\spec{X}(U\oplus V)
\,,
\end{equation}
where $a_{U,V}$ is the map defined in~\eqref{eq:S->Th-plus}.
Moreover, if $U$ is a $G$-subrepresentation of~$V$ define the \emph{internal} structure map~$\sigma_{U\subseteq V}$ to be the composition
\begin{equation}
\label{eq:structure-map-sub}
\sigma_{U\subseteq V}\colon
S^{V-U}\sma\spec{X}(U)
\xrightarrow{a_{U\subseteq V}\sma\id}
\Th(U,V)\sma\spec{X}(U)
\xrightarrow{\chi_{U,V}}
\spec{X}(V)
\,,
\end{equation}
where $a_{U\subseteq V}$ is the map defined in~\eqref{eq:S->Th-sub}.
\end{definition}

\begin{remark}[Alternative definition of orthogonal $G$-spectra]
\label{rem:alternative}
Given an orthogonal $G$-spectrum~$\spec{X}$, the internal structure maps~\eqref{eq:structure-map-sub} satisfy the following conditions:
\begin{enumerate}

\item\label{i:ax-assoc}
for all \fd\ $G$-representations $U\subseteq V\subseteq W$, the diagram
\begin{equation*}
\begin{tikzcd}[column sep=large]
S^{W-V}\sma S^{V-U}\sma\spec{X}(U)
\arrow{r}{\id\sma\sigma_{U\subseteq V}}
\arrow{d}[swap]{\cong\sma\id}
&
S^{W-V}\sma\spec{X}(V)
\arrow{d}{\sigma_{V\subseteq W}}
\\
S^{W-U}\sma\spec{X}(U)
\arrow{r}[swap]{\sigma_{U\subseteq W}}
&
\spec{X}(W)
\end{tikzcd}
\end{equation*}
commutes;

\item\label{i:ax-unit}
for all \fd\ $G$-representations $U$ the map
\begin{equation*}
\MOR{\sigma_{U\subseteq U}}{S^0\sma\spec{X}(U)}{\spec{X}(U)}
\end{equation*}
is the canonical isomorphism.

\end{enumerate}
Notice also that, by restriction along
\[
\iso\LL[fin]{G}\subseteq\LL[fin]{G}\TO[z]\Th_G
,
\]
$\spec{X}$ yields a continuous $G$-functor
\[
\MOR{\spec{X}}{\iso\LL[fin]{G}}{\CT_G}
,
\]
and
\begin{enumerate}[resume*]
\item\label{i:ax-compat}
for every commutative diagram of \fd\ $G$-representations
\begin{equation*}
\begin{tikzcd}
U
\arrow[hookrightarrow]{r}
\arrow{d}[swap]{f}
&
V
\arrow{d}{g}
\\
U'
\arrow[hookrightarrow]{r}
&
V'
\end{tikzcd}
\end{equation*}
where the horizontal maps are $G$-equivariant inclusions of subrepresentations and the vertical ones are not necessarily equivariant isomorphisms, the diagram
\begin{equation*}
\begin{tikzcd}[column sep=large]
S^{V-U}\sma\spec{X}(U)
\arrow{r}{\sigma_{U\subseteq V}}
\arrow{d}[swap]{S^{g_{|V-U}}\sma\spec{X}(f)}
&
\spec{X}(V)
\arrow{d}{\spec{X}(g)}
\\
S^{V'-U'}\sma\spec{X}(U')
\arrow{r}[swap]{\sigma_{U'\subseteq V'}}
&
\spec{X}(V')
\end{tikzcd}
\end{equation*}
commutes.
\end{enumerate}
On the other hand, suppose that we are given a continuous $G$-functor
\begin{equation*}
\MOR{\spec{X}}{\iso\LL[fin]{G}}{\CT_G}
\end{equation*}
together with pointed $G$-maps
\begin{equation*}
\MOR{\sigma_{U\subseteq V}}{S^{V-U}\sma\spec{X}(U)}{\spec{X}(V)}
\end{equation*}
for any \fd\ $G$-representation~$V$ and any $G$-subrepresentation $U\subseteq V$, satisfying conditions \ref{i:ax-assoc}, \ref{i:ax-unit}, and~\ref{i:ax-compat} above.
This data is taken as the definition of an orthogonal $G$-spectrum in~\cite{MM}*{Definition~II.2.6 on page~32}, and it can be extended to a continuous $G$-functor
\begin{equation*}
\MOR{\spec{X}}{\Th_G}{\CT_G}
\,,
\end{equation*}
i.e., an orthogonal $G$-spectrum according to our definition.
The equivalence of these two definitions is stated in~\cite{MM}*{Theorem~II.4.3 on page~36}, where it is based on more general results from~\cite{MMSS}*{Section~2 on pages~449--450}.
\end{remark}

\begin{definition}[Maps]
A \emph{$G$-map} $\MOR{f}{\spec{X}}{\spec{Y}}$ of orthogonal $G$-spectra is a $G$-equivariant and continuous natural transformation.
We denote by $\Sp^G$ the topological category of orthogonal $G$-spectra and $G$-maps.
Since the category $\Th_G$ has a small skeleton, there are no set-theoretic issues here.
Similarly we can define not necessarily $G$-equivariant maps of orthogonal $G$-spectra as continuous natural transformations, and we denote by $\Sp_G$ the corresponding topological $G$-category.
Notice that $\Sp^G(\spec{X},\spec{Y})=\Sp_G(\spec{X},\spec{Y})^G$.
\end{definition}

\begin{definition}[Naive orthogonal spectra]
A \emph{naive} orthogonal $G$-spectrum is a continuous functor
\[
\MOR{\spec{X}}{\Th}{\CT^G}
.
\]
Equivalently, using the notation from~\eqref{eq:Th->Th_G}, a naive orthogonal $G$-spectrum can be thought of as a continuous $G$-functor
\(
{\Th^\tr_G}\TO{\CT_G}
\).
We denote by~$\NSp_G$ the topological $G$-category of naive orthogonal $G$-spectra.
\end{definition}

Any orthogonal $G$-spectrum $\MOR{\spec{X}}{\Th_G}{\CT_G}$ has an underlying continuous functor $\MOR{\spec{X}}{\Th^G}{\CT^G}$, which we can compose with $\MOR{\iota}{\Th}{\Th^G}$ from~\eqref{eq:Th->Th^G} to obtain a continuous functor
\begin{equation*}
\MOR{\iota^*\spec{X}}{\Th}{\CT^G}
,
\end{equation*}
the underlying naive orthogonal $G$-spectrum of~$\spec{X}$.
The following important result states that the converse is also true.

\begin{theorem}[Naive is not naive]
\label{thm:naive}
The functor
\[
\MOR{\iota^*}{\Sp_G}{\NSp_G}
\]
is an equivalence of topological $G$-categories, with inverse
\[
\MOR{\Lan_\iota}{\NSp_G}{\Sp_G}
\,,
\]
the left Kan extension along~$\MOR{\iota}{\Th^\tr_G}{\Th_G}$.
In particular, if two orthogonal $G$-spectra have isomorphic underlying naive orthogonal $G$-spectra, then they are isomorphic.
\end{theorem}

\begin{proof}
See \cite{HHR}*{Proposition~A.19 on page~126} and also~\cite{MM}*{Theorem~V.1.5 on pages~75--76}.
Left Kan extensions are recalled in the proof of Proposition~\ref{prop:colim->Lan}.
\end{proof}

\begin{remark}
\label{rem:naive}
If $\spec{X}$ is a naive orthogonal $G$-spectrum and $U$ is an $n$-dimensional $G$-representation, then there is a $G$-homeomorphism
\[
\bigl(\Lan_\iota\spec{X}\bigr)(U)
\cong
\CL(\IR^n,U)_+\sma_{O(n)}\spec{X}(\IR^n)
\,.
\]
The choice of an isomorphism~$f\in\CL(\IR^n,U)$ gives a non-equivariant homeomorphism $\spec{X}(\IR^n)\cong(\Lan_\iota\spec{X})(U)$. 
Therefore the underlying non-equivariant space of $(\Lan_\iota\spec{X})(U)$ depends up to homeomorphism only on the dimension of~$U$.
\end{remark}

We now introduce some basic constructions.

\begin{definition}[Smash products with spaces and mapping spectra]
\label{def:sma-map}
Given an orthogonal $G$-spectrum~$\spec{X}$ and a pointed $G$-space~$Z$ we define new orthogonal $G$-spectra
\[
Z\sma\spec{X}
\AND
\map(Z,\spec{X})
\]
as the compositions of~$\MOR{\spec{X}}{\Th_G}{\CT_G}$ and the continuous $G$-functors
\begin{equation}
\label{eq:sma-map}
\MOR{Z\sma-}{\CT_G}{\CT_G}
\AND
\MOR{\map(Z,-)}{\CT_G}{\CT_G}
\,.
\end{equation}
The adjunction between the functors in~\eqref{eq:sma-map} gives an adjunction between the continuous $G$-functors
\[
\MOR{Z\sma-}{\Sp_G}{\Sp_G}
\AND
\MOR{\map(Z,-)}{\Sp_G}{\Sp_G}
\,.
\]
\end{definition}

\begin{definition}[Homotopies]
\label{def:homotopy}
A $G$-equivariant homotopy of orthogonal $G$\=/spectra is a $G$-map
\(
I_+\sma\spec{X}\TO\spec{Y}
\),
where $I=[0,1]$ has the trivial $G$-action.
We then define as usual $G$-homotopic maps and $G$-homotopy equivalences. 
\end{definition}

\begin{definition}[Suspensions, loops, shifts, and~$Q$]
\label{def:slsQ}
Given a \fd\ $G$\=/representation~$U$ we define
\[
\Sigma^U\spec{X}=S^U\sma\spec{X}
\AND
\Omega^U\spec{X}=\map(S^U,\spec{X})
\]
and thus obtain adjoint continuous $G$-functors $\MOR{\Sigma^U,\Omega^U}{\Sp_G}{\Sp_G}$.
We also define the \emph{shift}
\[
\shift^U\spec{X}
\]
as the composition of~$\spec{X}$ and the continuous $G$-functor
\[
\MOR{U\oplus-}{\Th_G}{\Th_G}
\,,
\quad
V\mapsto U\oplus V
\,,
\]
defined on morphisms as
\[
\Th(V,W)\TO\Th(U\oplus V,U\oplus W)
\,,
\quad
\bigl(f,w\bigr)\mapsto\bigl(\id_U\oplus f,(0,w)\bigr)
\,;
\]
see Lemma~\ref{lem:plus}.
Finally we define
\[
\Q{U}(\spec{X})=\Omega^U\shift^U\spec{X}
\,,
\]
a construction that plays a major role in this paper.
\end{definition}

\begin{definition}[Spectrum-level structure maps]
\label{def:spec-structure-maps}
The structure maps~\eqref{eq:structure-map} give a $G$-map of orthogonal $G$-spectra
\begin{equation}
\label{eq:sigma}
\MOR{\sigma_U}{\Sigma^U\spec{X}}{\shift^U\spec{X}}
\,,
\end{equation}
that we call \emph{spectrum-level structure map} of~$\spec{X}$ at the $G$-representation~$U$.
To verify that $\sigma_U$ is a map of spectra we need to show that for all \fd\ $G$-representations $V$ and~$W$ and all $(f,w)\in\Th(V,W)$ the diagram
\[
\begin{tikzcd}[column sep=large]
S^U\sma\spec{X}(V)
\arrow{r}{\sigma_{U,V}}
\arrow{d}[swap]{\id\sma\spec{X}(f,w)}
&
\spec{X}(U\oplus V)
\arrow{d}{\spec{X}((\id\oplus f),(0,w))}
\\
S^U\sma\spec{X}(W)
\arrow{r}[swap]{\sigma_{U,W}}
&
\spec{X}(U\oplus W)
\end{tikzcd}
\]
commutes, and this follows from the definitions.
By adjunction, from \eqref{eq:sigma} we also get a $G$-map of orthogonal $G$-spectra
\begin{equation}
\label{eq:sigma-adj}
\MOR{\widetilde{\sigma}_U}{\spec{X}}{\Q{U}(\spec{X})}
\,.
\end{equation}
Similarly, given a \fd\ $G$-representation~$V$ and a $G$-subrepresentation~$U$, the internal structure maps~\eqref{eq:structure-map-sub} give $G$-maps of orthogonal $G$-spectra
\begin{equation}
\label{eq:sigma-sub}
\MOR{\sigma_{U\subseteq V}}{\Sigma^{V-U}\shift^U\spec{X}}{\shift^V\spec{X}}
\AND
\MOR{\widetilde{\sigma}_{U\subseteq V}}{\shift^U\spec{X}}{\Omega^{V-U}\shift^V\spec{X}}
\,.
\end{equation}
Finally we define
\begin{equation}
\label{eq:tau}
\MOR{\tau_{U\subseteq V}}{\Q{U}(\spec{X})}{\Q{V}(\spec{X})}
\end{equation}
as the composition
\[
\tau_{U\subseteq V}\colon
\Q{U}(\spec{X})
=
\Omega^U\shift^U\spec{X}
\xrightarrow{\Omega^U\widetilde{\sigma}_{U\subseteq V}}
\Omega^U\Omega^{V-U}\shift^V\spec{X}
\cong
\Omega^V\shift^V\spec{X}
=
\Q{V}(\spec{X})
\,
\]
where the isomorphism in the target is induced by the natural $G$-homeomorphism $\Omega^U\Omega^{V-U}\cong\Omega^V$ of functors $\CT_G\TO\CT_G$.
\end{definition}

\begin{remark}
\label{rem:tau}
Let $\CU$ be a $G$-representation, not necessarily finite dimensional.
Recall from the beginning of Section~\ref{sec:conventions} the notation~$\CP(\CU)$ for the poset of \fd\ $G$-subrepresentations of~$\CU$ ordered by inclusion.
Proposition~\ref{prop:fun-in-U-fin} implies that
\[
\MOR{\Q{(-)}(\spec{X})}{\CP(U)}{\Sp^G}
,\quad
U\longmapsto\Q{U}(\spec{X})
\,,\quad
U\subseteq V\longmapsto\tau_{U\subseteq V}
\]
is a functor.
The composition
\begin{equation}
\label{eq:Q0}
\spec{X}\cong\Q{0}(\spec{X})\xrightarrow{\tau_{0\subseteq V}}\Q{V}(\spec{X})
\end{equation}
is equal to~$\widetilde{\sigma}_V$.
Moreover, the diagram
\begin{equation}
\label{eq:triangle-Q}
\begin{tikzcd}[column sep=tiny]
&\spec{X}
\arrow{ld}[swap]{\widetilde{\sigma}_U}
\arrow{rd}{\widetilde{\sigma}_V}
\\
\Q{U}(\spec{X})
\arrow{rr}[swap]{\tau_{U\subseteq V}}
&&
\Q{V}(\spec{X})
\end{tikzcd}
\end{equation}
commutes, by the functoriality of~$\Q{(-)}(\spec{X})$ and~\eqref{eq:Q0}.
\end{remark}


\section{Homotopy theory of orthogonal \texorpdfstring{$G$}{G}-spectra}
\label{sec:homotopy-theory}

In this section we discuss homotopy groups, \piiso s, good and almost good spectra, $G$-$\Omega$-spectra, and their properties.
In the definition of homotopy groups we choose a complete $G$-universe in the sense of the following definition, but the homotopy groups are independent of this choice.
We emphasize that the choice of a universe plays no role in the definition of the category of orthogonal $G$-spectra.

\begin{definition}[Universes]
\label{def:universe}
A \emph{$G$-universe} is a countably infinite-dimensional $G$-representation~$\CU$ such that $\CU^G\neq0$ and for every \fd\ $G$-representation~$U$, if $U$ embeds in~$\CU$, then also $\oplus_\IN U$ embeds in~$\CU$.
A~$G$\=/universe~$\CU$ is \emph{complete} if every \fd\ $G$-representation embeds in~$\CU$.
Complete $G$-universes exist and are unique up to equivariant linear isometric isomorphisms.
Uniqueness follows immediately from the isotypical decomposition; see e.g.~\cite{BtD}*{Proposition~III.1.7 on page~128}.
\end{definition}

\begin{definition}[Homotopy groups]
\label{def:pi_*}
Let $\spec{X}$ be an orthogonal $G$-spectrum, let $H$ be a subgroup of~$G$, and let $n\in\IZ$ be an integer.
Choose a complete $G$-universe~$\CU$.
Define
\begin{equation}
\label{eq:pi}
\pi_n^H(\spec{X})=
\begin{cases}
\ds\colim_{U\in\CP(\CU)}
\pi_n
\Bigl(\bigl(\Omega^U\spec{X}(U)\bigr)^H\Bigr)
&\text{if $n\geq0$\,;}
\\
\ds\colim_{U\in\CP(\CU)}
\pi_{\mathrlap{0}\phantom{n}}
\Bigl(\bigl(\Omega^U\spec{X}(U\oplus\IR^{-n})\bigr)^H\Bigr)
&\text{if $n<0$\,.}
\end{cases}
\end{equation}
In order to describe the structure maps used for the above colimits, observe that
\[
\Omega^U\spec{X}(U)
\cong
\Q{U}(\spec{X})(0)
\AND
\Omega^U\spec{X}(U\oplus\IR^{-n})
=
\Q{U}(\spec{X})(\IR^{-n})
\,.
\]
By Remark~\ref{rem:tau}, $\MOR{\Q{(-)}(\spec{X})}{\CP(\CU)}{\Sp^G}$ is a functor, and so for any \fd\ $G$-representation~$V$ (e.g., $V=0$ or $V=\IR^{-n}$) and any~$m\geq0$ we can consider the composition
\[
\begin{tikzcd}[column sep=large]
\CP(\CU)
\arrow{r}{\Q{(-)}(\spec{X})}
&
\Sp^G
\arrow{r}{\ev_V}
&
\CT^G
\arrow{r}{(-)^H}
&
\CT
\arrow{r}{\pi_m}
&
\Ab
\end{tikzcd}
\]
and then take its colimit.
Here $\ev_V$ is given by evaluating an orthogonal $G$-spectrum at~$V$. 
Since any two complete $G$-universes are equivariantly and isometrically isomorphic, it follows that the abelian groups defined in~\eqref{eq:pi} do not depend on the choice of the universe~$\CU$.
\end{definition}

\begin{definition}[Weak equivalences, \piiso s, level equivalences]
\label{def:piiso}
Let $\MOR{f}{X}{Y}$ be a $G$-map of pointed or unpointed $G$-spaces.
\begin{enumerate}
\item 
\label{i:piiso}
We say that $f$ is a \emph{$G$-weak equivalence} if for all subgroups $H$ of~$G$ the induced map
\(
\MOR{f^H}{X^H}{Y^H}
\)
is a weak equivalence of unpointed non-equivariant spaces.
\end{enumerate}
Let $\MOR{f}{\spec{X}}{\spec{Y}}$ be a $G$-map of orthogonal $G$-spectra.
\begin{enumerate}[resume]
\item We say that $f$ is a \emph{\piiso} if for all integers~$n$ and all subgroups $H$ of~$G$ the induced map
\(
\MOR{\pi_n^H(f)}{\pi_n^H(\spec{X})}{\pi_n^H(\spec{Y})}
\)
is an isomorphism.
\item We say that $f$ is a \emph{level equivalence} if for all \fd\ $G$-representations~$U$ the map
\(
\MOR{f(U)}{\spec{X}(U)}{\spec{Y}(U)}
\)
is a $G$-weak equivalence. 
\end{enumerate}
Notice that $G$-homotopy equivalences are in particular level equivalences.
\end{definition}

In the setup of Lewis-May-Steinberger, the following definition of good spectra appears in~\cite{HM-top}*{Appendix~A on pages~96--98}.

\begin{definition}[Good and almost good spectra]
\label{def:good}
An orthogonal $G$-spectrum~$\spec{X}$ is \emph{good} if for all \fd\ $G$-representations $U$ the spectrum-level structure map $\MOR{\sigma_U}{\Sigma^U\spec{X}}{\shift^U\spec{X}}$ from~\eqref{eq:sigma} is a levelwise closed embedding, i.e., if for all \fd\ $G$-representations $U$ and~$V$, the structure map
\[
\MOR{{\sigma}_{U,V}}{\Sigma^U\spec{X}(V)}{\spec{X}(U\oplus V)}
\]
from~\eqref{eq:structure-map} is a closed embedding.
We say that $\spec{X}$ is \emph{almost good} if the adjoints of the structure maps are closed embeddings.
\end{definition}

In Lemma~\ref{lem:good} we show that cofibrant implies good, and that good implies almost good.
We also show that being good is preserved by certain left adjoint functors, and almost good by certain right adjoints; see Lemma~\ref{lem:good} and Theorem~\ref{thm:H-in-G}.
In Addendum~\ref{add:almost} we prove that Theorem~\ref{thm:Q} remains true for almost good spectra.
However, our proof of Main Theorem~\ref{thm:main} does not generalize to almost good spectra.

We now collect some properties of closed embeddings in the category of compactly generated weak Hausdorff spaces, referring to~\cite{Strickland} for a thorough account and proofs.
Notice that closed embeddings are called closed inclusions in \cite{Strickland} and~\cite{HM-top}.
Given a commutative square
\begin{equation}
\label{eq:square}
\begin{tikzcd}
W
\arrow{r}{i}
\arrow{d}[swap]{f}
&
X
\arrow{d}{g}
\\
Y
\arrow{r}[swap]{j}
&
Z\mathrlap{\,,}
\end{tikzcd}
\end{equation}
we say that the square is \emph{admissible} if all maps are closed embeddings and also the map $Y\cup_W X\TO Z$ is a closed embedding.
Notice that if \eqref{eq:square} is a pullback and all maps are closed embeddings, then it is admissible.

\begin{facts}[Properties of closed embeddings]
\label{facts:clemb}
\ \newline\vspace{-3.5ex}
\begin{enumerate}

\item
\label{i:comp-clemb}
Given $X\TO[i]Y\TO[j]Z$, if $i$ and $j$ are closed embeddings, then so is~$j\circ i$;\linebreak if~$j\circ i$ is a closed embedding, then so is~$i$.

\item
\label{i:push-pull-clemb}
Consider the commutative square~\eqref{eq:square}.
If the square is a pullback and $j$ is a closed embedding, then so is~$i$.
If the square is a pushout and $i$ is a closed embedding, then so is~$j$, and the square is a pullback.

\item
\label{i:cube-clemb}
Consider a commutative cube
\begin{equation*}
\begin{tikzcd}[row sep=tiny, column sep=tiny]
W'
\arrow[rr]
\arrow[dd, swap]
\arrow[dr]
&&
X'
\arrow[dd]
\arrow[dr]
\\
&
W
\arrow[rr, crossing over]
&&
X
\arrow[dd]
\\
Y'
\arrow[rr, swap]
\arrow[dr]
&&
Z'
\arrow[dr]
\\
&
Y
\arrow[rr, swap, "j"]
\arrow[uu, leftarrow, crossing over]
&&
Z
\mathrlap{\,.}
\end{tikzcd}
\end{equation*}
Assume that the back face is admissible and that both the left and the right faces are pushouts.
Then the front face is also admissible, and in particular $j$ is a closed embedding.

\item
\label{i:adjoint-clemb}
If $\MOR{i}{X\sma Y}{Z}$ is a closed embedding, then its adjoint $\MOR{\widetilde{i}}{Y}{\map(X,Z)}$ is also a closed embedding.

\item
\label{i:(co)prod-clemb}
Products and coproducts of closed embeddings are closed embeddings.

\item
\label{i:sma-map-clemb}
Given any based space~$Z$, if $\MOR{i}{X}{Y}$ is a closed embedding then so are $\MOR{\id_Z\sma i}{Z\sma X}{Z\sma Y}$ and $\MOR{i_*}{\map(Z,X)}{\map(Z,Y)}$.

\item
\label{i:colim-clemb}
If $X_0\TO[i_0]X_1\TO[i_1]X_2\TO[i_2]\dotsb$ is a sequence of closed embeddings, then the maps $X_m\TO\colim_n X_n$ are closed embeddings, and for every compact based space~$Z$ the natural map $\colim_n\map(Z,X_n)\TO\map(Z,\colim_n X_n)$ is a homeomorphism.
In particular, $\colim_n\pi_*(X_n)\cong\pi_*(\colim_n X_n)$.

\item
\label{i:colim-ladder-clemb}
Given a commutative diagram
\begin{equation}
\label{eq:ladder}
\begin{tikzcd}
X_0
\arrow{r}{i_0}
\arrow{d}[swap]{f_0}
&
X_1
\arrow{r}{i_1}
\arrow{d}[swap]{f_1}
&
X_2
\arrow{r}{i_2}
\arrow{d}[swap]{f_2}
&
\dotsb
\\
Y_0
\arrow{r}[swap]{j_0}
&
Y_1
\arrow{r}[swap]{j_1}
&
Y_2
\arrow{r}[swap]{j_2}
&
\dotsb
\end{tikzcd}
\end{equation}
where all maps are closed embeddings and all squares are admissible, then
\begin{equation}
\label{eq:colim-ladder}
\begin{tikzcd}
X_m
\arrow{r}
\arrow{d}[swap]{f_m}
&
\colim_n X_n
\arrow{d}{\colim_n f_n}
\\
Y_m
\arrow{r}
&
\colim_n Y_n
\end{tikzcd}
\end{equation}
is admissible, and in particular $\MOR{\colim_n f_n}{\colim_n X_n}{\colim_n Y_n}$ is a closed embedding.
If additionally all squares in~\eqref{eq:ladder} are pullbacks, then \eqref{eq:colim-ladder} is also a pullback.

\item
\label{i:fix-clemb}
If $\MOR{i}{X}{Y}$ is a $G$-equivariant closed embedding and $H$ is a subgroup of~$G$, then $\MOR{i^H}{X^H}{Y^H}$ is also a closed embeddings.

\item
\label{i:fix-colim-clemb}
If $X_0\TO[i_0]X_1\TO[i_1]X_2\TO[i_2]\dotsb$ is a sequence of $G$-equivariant closed embeddings and $H$ is a subgroup of~$G$, then the natural map $\colim_n(X_n^H)\TO(\colim_n X_n)^H$ is a homeomorphism.

\end{enumerate}
\end{facts}

\begin{proof}
For~\ref{i:comp-clemb} see~\cite{Strickland}*{Proposition~2.31 on page~8}, and 
for~\ref{i:push-pull-clemb}~\cite{Strickland}*{Propositions 2.33 and~2.35 on page~9}.

These two facts, and the fact that points are closed in weak Hausdorff spaces, imply that the category~$\CT$ together with the class of closed embeddings is a category with cofibrations in the sense of Waldhausen~\cite{Waldhausen}*{Section~1.1 on page~320}.
Then \ref{i:cube-clemb} is a formal consequence true in any category with cofibrations; see~\cite{Waldhausen}*{Lemma~1.1.1 on page~322}.

For~\ref{i:adjoint-clemb} see~\cite{Strickland}*{Corollary~5.11 on page~19};
for~\ref{i:(co)prod-clemb}~\cite{Strickland}*{Theorem~3.1 on page~11};
for~\ref{i:sma-map-clemb}~\cite{Strickland}*{Corollary~5.4 on page~18, Corollary~5.8(b) on page~19 and Theorem~4.8(b) on page~17};
and for~\ref{i:colim-clemb}~\cite{Strickland}*{Lemmas 3.3, 3.6, and~3.8 on pages~12--13}.

For~\ref{i:colim-ladder-clemb}, under the additional assumption that all squares in~\eqref{eq:ladder} are pullbacks, it is proved in~\cite{Strickland}*{Lemma~3.9 on page~14} that \eqref{eq:colim-ladder} is a pullback square of closed embeddings.
If the squares in~\eqref{eq:ladder} are only admissible, consider the following diagram.
\[
\begin{tikzcd}[row sep=small]
X_0
\arrow[r]
\arrow[d]
&
X_1
\arrow[r]
\arrow[d]
&
X_2
\arrow[r]
\arrow[d]
&
\dotsb
\\
Y_0
\arrow[r]
&
Y_0\cup_{X_0}\!X_1
\arrow[r]
\arrow[d]
&
(Y_0\cup_{X_0}\!X_1)\cup_{X_1}\!X_2
\arrow[r]
\arrow[d]
&
\dotsb
\\
&
Y_1
\arrow[r]
&
Y_1\cup_{X_1}\!{X_2}
\arrow[r]
\arrow[d]
&
\dotsb
\\
&
&
Y_2
\arrow[r]
&
\dotsb
\end{tikzcd}
\]
All maps are closed embeddings and all squares are pushouts, and therefore pullbacks by~\ref{i:push-pull-clemb}.
So the already established part of~\ref{i:colim-ladder-clemb} applies to any two consecutive rows, and by taking their colimits we get a sequence of closed embeddings
\(
\colim_n X_n
\TO
\colim_n (Y_0\cup\ldots)
\TO
\colim_n (Y_1\cup\ldots)
\TO
\dotsb
\)\,.
Now apply~\ref{i:colim-clemb}.

Fact~\ref{i:fix-clemb} follows from~\ref{i:comp-clemb}, since $X^H\TO X$ is a closed embedding.
For~\ref{i:fix-colim-clemb}, we first use \ref{i:fix-clemb} and~\ref{i:colim-ladder-clemb} to conclude that $\MOR{f}{\colim_n(X_n^H)}{\colim_n X_n}$ is a closed embedding.
Since $f$ factors through $\MOR{g}{\colim_n(X_n^H)}{(\colim_n X_n)^H}$, $g$~is also a closed embedding by~\ref{i:comp-clemb}.
By inspection, $g$ is surjective, thus proving~\ref{i:fix-colim-clemb}.
\end{proof}

\begin{lemma}[Good and almost good spectra]
\label{lem:good}
Let $\spec{X}$ be an orthogonal $G$-spectrum.
\begin{enumerate}

\item
\label{i:cofib=>good}
If $\spec{X}$ is cofibrant in the stable model category structure on~$\Sp_G$ then $\spec{X}$ is good.

\item
\label{i:good-tau}
If $\spec{X}$ is good then $\spec{X}$ is almost good, and
$\spec{X}$ is almost good if and only if for all \fd\ $G$-representations $U\subseteq V$ the map $\MOR{\tau_{U\subseteq V}}{\Q{U}(\spec{X})}{\Q{V}(\spec{X})}$ defined in~\eqref{eq:tau} is a levelwise closed embedding.

\item
\label{i:pi-good}
If $\spec{X}$ is almost good then we can interchange the colimit and the homotopy groups in Definition~\ref{def:pi_*}, i.e., for all subgroups $H$ of~$G$ and for all~$n\in\IZ$ we have
\[
\pi_n^H(\spec{X})\cong
\begin{cases}
\ds\pi_n
\Bigl(\colim_{U\in\CP(\CU)}
\bigl(\Omega^U\spec{X}(U)\bigr)^H\Bigr)
&\text{if $n\geq0$\,;}
\\
\ds\pi_{\mathrlap{0}\phantom{n}}
\Bigl(\colim_{U\in\CP(\CU)}
\bigl(\Omega^U\spec{X}(U\oplus\IR^{-n})\bigr)^H\Bigr)
&\text{if $n<0$\,.}
\end{cases}
\]

\item
\label{i:(co)prod-good}
Coproducts of good spectra are good.
Products of almost good spectra are almost good.

\item
\label{i:sma-map-good}
Let $Z$ be a pointed $G$-space.
If $\spec{X}$ is good then so is~$Z\sma\spec{X}$.
If $\spec{X}$ is almost good then so is~$\map(Z,\spec{X})$.

\item
\label{i:good-naive}
$\spec{X}$ is (almost) good if and only if its underlying naive spectrum~$\iota^*\spec{X}$ is (almost) good.

\item
\label{i:good-and-Q}
If $\spec{X}$ is almost good, then so is $\Q{\CU}(\spec{X})$, and $\MOR{r}{\spec{X}}{\Q{\CU}(\spec{X})}$ is a levelwise closed embedding.

\end{enumerate}
\end{lemma}

\begin{proof}
For~\ref{i:cofib=>good}, we first establish a general claim about pushout squares
\begin{equation*}
\begin{tikzcd}
\spec{A}
\arrow[r]
\arrow[d, swap, "f"]
&
\spec{W}
\arrow[d, "g"]
\\
\spec{B}
\arrow[r]
&
\spec{Y}
\end{tikzcd}
\end{equation*}
of orthogonal $G$-spectra.
Assume that $\spec{A}$, $\spec{B}$, and $\spec{W}$ are good, that $f$ is a levelwise closed embedding, and that for every \fd\ $G$-representation~$U$ the square $\MOR{\sigma_U}{\Sigma^U f}{\shift^U f}$ is levelwise admissible as defined in~\eqref{eq:square}.
Then $\spec{Y}$ is good, $g$~is a levelwise closed embedding, and for every~$U$ the square $\MOR{\sigma_U}{\Sigma^U g}{\shift^U g}$ is levelwise admissible.
This claim follows at once from Fact~\ref{facts:clemb}\ref{i:cube-clemb}.
(In fact, the full subcategory of good orthogonal $G$-spectra, together with the class of levelwise closed embeddings~$f$ such that for every~$U$ the square $\MOR{\sigma_U}{\Sigma^U f}{\shift^U f}$ is levelwise admissible, is a category with cofibrations.)

Next we show that $FI$-cell complexes in the sense of~\citelist{\cite{MM}*{Definitions III.2.2 and~III.1.1 on pages 42 and~38--39} \cite{MMSS}*{Definition~5.4 on page~457}} are good.
For any \fd\ $G$-representation~$V$, the free spectrum~$F_V=\Th(V,-)$ is good by inspection.
Then for any pointed $G$-space $A$ also $F_V A = A \sma \Th(V,-)$ is good by part~\ref{i:sma-map-good}, and for any closed embedding $\MOR{i}{A}{B}$
and any \fd\ $G$-representation $U$ 
the square $\MOR{\sigma_U}{\Sigma^U F_V i}{\shift^U F_V i}$ is levelwise a pullback square of closed embeddings, and so in particular levelwise admissible.
It follows that for any coproduct $f=\bigvee_\lambda F_{V_\lambda}i_\lambda$ such that each $i_\lambda$ is a closed embedding, the source and target of~$f$ are good by part~\ref{i:(co)prod-good}, and the square $\MOR{\sigma_U}{\Sigma^U f}{\shift^U f}$ is levelwise admissible.
Now recall that an $FI$-cell complex is a sequential colimit $\spec{X}=\colim_n\spec{X}_n$\,, where $\spec{X}_0=\spec{\pt}$ and each $\MOR{g_n}{\spec{X}_n}{\spec{X}_{n+1}}$ is obtained by cobase change from a map of the form $f=\bigvee_\lambda F_{V_\lambda}i_\lambda$, with $\MOR{i_\lambda}{(G/H_\lambda\times S^{n_\lambda-1})_+}{(G/H_\lambda\times D^{n_\lambda})_+}$.
The claim from the beginning of the proof then implies inductively that all $\spec{X}_n$ are good and all squares $\MOR{\sigma_U}{\Sigma^U g_n}{\shift^U g_n}$ are levelwise admissible.
Hence we can apply Fact~\ref{facts:clemb}\ref{i:colim-ladder-clemb} to conclude that each $FI$-cell complex $\spec{X}$ is good.
Since cofibrant spectra are retracts of $FI$-cell complexes by~\cite{MM}*{Theorem~III.2.4(iii) on page~42}, and retracts of good spectra are good by Fact~\ref{facts:clemb}\ref{i:comp-clemb}, the result follows.

For \ref{i:good-tau}, the fact that good implies almost good follows from Fact~\ref{facts:clemb}\ref{i:adjoint-clemb}.
If~$\spec{X}$ is almost good, then also the adjoints of the internal structure maps are closed embeddings, and so $\tau_{U\subseteq V}$ is a levelwise closed embedding by~Fact~\ref{facts:clemb}\ref{i:sma-map-clemb}.
The other implication is obvious from~\eqref{eq:Q0}.

\ref{i:pi-good} follows from~\ref{i:good-tau} and Facts~\ref{facts:clemb}\ref{i:fix-clemb} and~\ref{i:colim-clemb}, since for any complete $G$-universe~$\CU$ the poset~$\CP(\CU)$ contains a cofinal sequence.

\ref{i:(co)prod-good} and \ref{i:sma-map-good} follow from Facts~\ref{facts:clemb}\ref{i:(co)prod-clemb} and~\ref{i:sma-map-clemb}.

For the non-trivial direction of~\ref{i:good-naive}, notice that if $U$ and~$V$ are $G$\=/representations of dimensions $m$ and~$n$, then there are commutative squares
\[
\begin{tikzcd}[column sep=large]
S^m\sma\spec{X}(\IR^n)
\arrow{r}{\sigma_{\IR^m,\IR^n}}
\arrow{d}[swap]{\cong}
&
\spec{X}(\IR^{m+n})
\arrow{d}{\cong}
&
\spec{X}(\IR^n)
\arrow{r}{\widetilde{\sigma}_{\IR^m,\IR^n}}
\arrow{d}[swap]{\cong}
&
\Omega^m\spec{X}(\IR^{m+n})
\arrow{d}{\cong}
\\
S^U\sma\spec{X}(U)
\arrow{r}[swap]{\sigma_{U,V}\vphantom{\widetilde{\sigma}}}
&
\spec{X}(U\oplus V)
&
\spec{X}(U)
\arrow{r}[swap]{\widetilde{\sigma}_{U,V}}
&
\Omega^U\spec{X}(U\oplus V)
\mathrlap{\,,}
\end{tikzcd}
\]
whose vertical maps are not necessarily equivariant homeomorphisms; compare Remark~\ref{rem:alternative}\ref{i:ax-compat}.

Finally, for~\ref{i:good-and-Q} use Facts~\ref{facts:clemb}\ref{i:adjoint-clemb} and~\ref{i:colim-clemb}.
\end{proof}

\begin{definition}[$G$-$\Omega$-spectra]
\label{def:Omega}
An orthogonal $G$-spectrum~$\spec{X}$ is a \emph{$G$-$\Omega$-spectrum} if for all \fd\ $G$-representations $U$ the adjoint $\MOR{\widetilde{\sigma}_U}{\spec{X}}{\Q{U}(\spec{X})}$ from~\eqref{eq:sigma-adj} of the spectrum-level structure map is a level equivalence, i.e., if for all \fd\ $G$-representations $U$ and~$V$, the adjoint
\[
\MOR{\widetilde{\sigma}_{U,V}}{\spec{X}(V)}{\Omega^U\spec{X}(U\oplus V)}
\]
of the structure map~\eqref{eq:structure-map} is a $G$-weak equivalence.
\end{definition}

\begin{lemma}[$G$-$\Omega$-spectra]
\label{lem:Omega}
Let $\spec{X}$ be an orthogonal $G$-spectrum.
\begin{enumerate}

\item
\label{i:fib=>Omega}
$\spec{X}$ is fibrant in the stable model category structure on~$\Sp_G$ if and only if $\spec{X}$ is a $G$-$\Omega$-spectrum.

\item
\label{i:Omega-tau}
$\spec{X}$ is a $G$-$\Omega$-spectrum if and only if for all \fd\ $G$-representations $U\subseteq V$ the map $\MOR{\tau_{U\subseteq V}}{\Q{U}(\spec{X})}{\Q{V}(\spec{X})}$ defined in~\eqref{eq:tau} is a level equivalence.

\item
\label{i:pi-Omega}
If $\spec{X}$ is a $G$-$\Omega$-spectrum then we can omit the colimit in Definition~\ref{def:pi_*}, i.e., for all subgroups $H$ of~$G$ and for all~$n\in\IZ$ we have
\[
\pi^H_n(\spec{X}) \cong
\begin{cases}
\pi_n\bigl(\spec{X}(0)^H\bigr)
&\text{if $n\geq0$\,;}
\\[\smallskipamount]
\pi_n\bigl(\spec{X}(\IR^{-n})^H\bigr)
&\text{if $n<0$\,.}
\end{cases}
\]

\item
\label{i:product-Omega}
Products of $G$-$\Omega$-spectra are $G$-$\Omega$-spectra.

\item
\label{i:map-Omega}
Let $Z$ be a pointed $G$-CW-complex.
If $\spec{X}$ is a $G$-$\Omega$-spectrum then so is $\map(Z,\spec{X})$.

\end{enumerate}
\end{lemma}

\begin{proof}
For~\ref{i:fib=>Omega} see~\cite{MM}*{Theorem III.4.2 on page~47}.
\ref{i:Omega-tau} follows from the commutative triangle~\eqref{eq:triangle-Q}, and \ref{i:Omega-tau} directly implies~\ref{i:pi-Omega}.
For~\ref{i:product-Omega} use Fact~\ref{facts:clemb}\ref{i:(co)prod-clemb} and that the adjoint structure maps of the product are the product of the adjoint structure maps of the factors.
For~\ref{i:map-Omega} see for example~\cite{Schwede}*{last paragraph on page~47}.
\end{proof}

We conclude this section with the following well-known facts.

\begin{lemma}
\label{lem:well-known}
Let $\MOR{f}{\spec{X}}{\spec{Y}}$ be a $G$-map of orthogonal $G$-spectra, and let $Z$ be a pointed $G$-CW-complex.
\begin{enumerate}

\item
\label{i:le-pi-iso}
If $f$ is a level equivalence, then $f$ is a \piiso, and the converse is true provided that $\spec{X}$ and $\spec{Y}$ are $G$-$\Omega$-spectra.


\item
\label{i:map-pi-iso}
If $f$ is a \piiso\ and $Z$ has only finitely many $G$-cells, then\linebreak $\MOR{\map(Z,f)}{\map(Z,\spec{X})}{\map(Z,\spec{Y})}$ is a \piiso.

\item
\label{i:sma-pi-iso}
If $f$ is a \piiso, then $\MOR{\id_Z\sma f}{Z\sma\spec{X}}{Z\sma\spec{Y}}$ is a \piiso.

\item
\label{i:sma-F(N)}
If $N$ is a normal subgroup of~$G$, $\pi_*^H(f)$ is an isomorphism for all $H\in\CF(N)$, and $Z$ is a $G$-$\CF(N)$-CW-complex, then $\MOR{\id_Z\sma f}{Z\sma\spec{X}}{Z\sma\spec{Y}}$ is a \piiso\ of orthogonal $G$-spectra, and $\MOR{\id_Z\sma_N f}{Z\sma_N\spec{X}}{Z\sma_N\spec{Y}}$ is a \piiso\ of orthogonal $G/N$-spectra.

\end{enumerate}
\end{lemma}

\begin{proof}
For \ref{i:le-pi-iso} see \cite{MM}*{Lemma~III.3.3 and Theorem~III.3.4 on page~45} or~\cite{Schwede}*{Propositions~5.18(ii) and~5.19(ii) on page~55}.
For \ref{i:map-pi-iso} and~\ref{i:sma-pi-iso} see \cite{MM}*{Theorem~III.3.9 and Proposition~III.3.11 on pages~46--47} or \cite{Schwede}*{Proposition~5.4 on page~48}.
For~\ref{i:sma-F(N)} we use arguments similar to those given in~\cite{Schwede} for~\ref{i:sma-pi-iso}, and proceed by induction on the cellular filtration of~$Z$.
The induction begin is the case $Z=G/H_+$ with $H\cap N=1$.
By Example~\ref{eg:ind-res} there is a $G$-isomorphism $G/H_+\sma\spec{X}\cong\ind_{H\leq G}\res_{H\leq G}\spec{X}$.
Because $H\cap N=1$, we can identify $H$ with a subgroup of~$G/N$ and get a $G/N$-isomorphism~$G/H_+\sma_N\spec{X}\cong\ind_{H\leq G/N}\res_{H\leq G}\spec{X}$.
Since each subgroup $K$ of~$H$ also belongs to~$\CF(N)$, the assumption on~$f$ implies that $\MOR{\res_{H\leq G}f}{\res_{H\leq G}\spec{X}}{\res_{H\leq G}\spec{Y}}$ is a \piiso\ of orthogonal $H$-spectra; compare~\eqref{eq:pi-res}.
By Theorem~\ref{thm:H-in-G}\ref{i:res-co-ind-piiso}, induction preserves \piiso s, and therefore both $\id_{G/H_+}\sma f$ and~$\id_{G/H_+}\sma_N f$ are \piiso s.
The induction steps are then carried out exactly as at the end of Section~\ref{sec:proof}.
\end{proof}


\section{Change of groups}
\label{sec:change-of-groups}

We first recall some basic space-level constructions.
Let $\MOR{\alpha}{\Gamma}{G}$ be a group homomorphism.
The \emph{restriction} functor
\[
\MOR{\res_\alpha=\alpha^*}{\CT^G}{\CT^\Gamma}
\]
has both a left adjoint called \emph{induction}
\begin{alignat}{2}
\label{eq:ind}
\MOR{\ind_\alpha}{&\CT^\Gamma}{\CT^G}
,
&\quad
X&\mapsto G_+\sma_\Gamma X
\,,
\shortintertext{and a right adjoint called \emph{coinduction}}
\label{eq:coind}
\MOR{\coind_\alpha}{&\CT^\Gamma}{\CT^G}
,
&
X&\mapsto \map(G_+,X)^\Gamma
,
\end{alignat}
where in~\eqref{eq:ind} we consider $G$ as right $\Gamma$-set and left $G$-set,
and in~\eqref{eq:coind} we consider $G$ as left $\Gamma$-set and right $G$-set.
Restriction, induction, and coinduction are all continuous functors.
If $\alpha$ is surjective, then $G\cong \Gamma/N$ with $N=\ker\alpha$, and
\begin{alignat}{3}
\label{eq:ind-orb-coind-fix}
\ind_\alpha(X)
&\cong
X/N
&\quad
&\text{and}&
\coind_\alpha(X)
&\cong
X^N
.
\shortintertext{For all pointed $G$-spaces~$X$ there are natural $G$-equivariant homeomorphism}
\label{eq:indco-res}
\qquad
\ind_\alpha\res_\alpha X
&\cong
G/\alpha(\Gamma)_+ \sma X
&
&\text{and}&
\coind_\alpha\res_\alpha X
&\cong
\map(\alpha(\Gamma)\backslash G _+ , X)
\shortintertext{given by the following formulas:}
[g\sma x]&\mapsto g\alpha(\Gamma)\sma gx
&&&
f&\mapsto\bigl(\alpha(\Gamma)g\mapsto g^{-1}f(g)\bigr)
\\
[g\sma g^{-1}x]&\mapsfrom g\alpha(\Gamma)\sma x
&&&
\bigl(g\mapsto gf(\alpha(\Gamma)g)\bigr)&\mapsfrom f
\end{alignat}
Under the isomorphisms~\eqref{eq:indco-res},
the counit of the adjunction $\ind_\alpha\dashv\res_\alpha$ is induced by the projection $G/\alpha(\Gamma)\TO\pt$,
and the unit of the adjunction $\res_\alpha\dashv\coind_\alpha$ is induced by $\alpha(\Gamma)\backslash G\TO\pt$.

\begin{construction}[Change of groups for naive orthogonal spectra]
\label{constr:irc-naive}
We define restriction, induction, and coinduction for naive orthogonal spectra by composing with the corresponding continuous functors for spaces.
Explicitly, if $\MOR{\spec{X}}{\Th}{\CT^G}$ is a naive orthogonal $G$-spectrum and $\MOR{\alpha}{\Gamma}{G}$ a group homomorphism, then $\res_\alpha\spec{X}$ is the composition
\[
\Th\TO[\spec{X}]\CT^G\xrightarrow{\res_\alpha}\CT^\Gamma
,
\]
and analogously for induction and coinduction.
We thus get the following continuous functors and adjunctions.
\begin{equation}
\label{eq:irc-naive}
\begin{tikzcd}[row sep=6.5ex]
\NSp^G
\arrow{d}[pos=.53, description]{\dashv\ \,\res_\alpha\ \dashv}
\\
\NSp^\Gamma
\arrow[bend left =60]{u}[pos=.47      ]{  \ind_\alpha}
\arrow[bend right=60]{u}[pos=.47, swap]{\coind_\alpha}
\end{tikzcd}
\end{equation}
\end{construction}

\begin{construction}[Change of groups for non-naive orthogonal spectra]
\label{constr:irc-non-naive}
We use the equivalence of categories in Theorem~\ref{thm:naive} to extend the functors and adjunctions in~\eqref{eq:irc-naive} to non-naive orthogonal spectra.
Explicitly, if $\MOR{\spec{X}}{\Th_G}{\CT_G}$ is an orthogonal $G$-spectrum, then the orthogonal $\Gamma$-spectrum $\res_\alpha\spec{X}$ is defined by first considering the restriction
\begin{equation}
\label{eq:naive-res}
\Th\TO[\iota]\Th^G\TO[\spec{X}]\CT^G\xrightarrow{\res_\alpha}\CT^\Gamma
\end{equation}
of the underlying naive orthogonal $G$-spectrum of~$\spec{X}$,
then view~\eqref{eq:naive-res} as a continuous $\Gamma$-functor~$\Th^\tr_\Gamma\TO\CT_\Gamma$,
and finally take the left Kan extension along~$\MOR{\iota}{\Th^\tr_\Gamma}{\Th_\Gamma}$.
Induction and coinduction are defined analogously.
Thus we obtain continuous functors and adjunctions as in~\eqref{eq:irc-naive} for $\Sp$ instead of~$\NSp$.
Notice that $\iota^*\res_\alpha\spec{X}=\res_\alpha\iota^*\spec{X}$, and analogously for induction and coinduction.

For induction and coinduction, but not for restriction, this process simplifies.
Given an orthogonal $\Gamma$-spectrum~$\spec{X}$, the underlying continuous functor \mbox{$\Th^G\TO\CT^G$} of the orthogonal $G$-spectrum $\MOR{\indco_\alpha\spec{X}}{\Th_G}{\CT_G}$ is isomorphic to the composition
\begin{equation}
\label{eq:non-naive-indco}
\Th^G\xrightarrow{\res_\alpha}
\Th^\Gamma\TO[\spec{X}]\CT^\Gamma\xrightarrow{\indco_\alpha}\CT^{G}
.
\end{equation}
In fact, since $\res_{\alpha}\circ\res_{G\to1}=\res_{\Gamma\to1}$, the underlying naive orthogonal $G$-spectra of $\indco_\alpha\spec{X}$ and~\eqref{eq:non-naive-indco} agree, and so the claim follows from the last sentence in Theorem~\ref{thm:naive}.
Explicitly, this means that for any \fd\ $G$-representation~$U$ we have
\begin{equation}
\label{eq:indco}
\bigl(\indco_\alpha\spec{X}\bigr)(U)
\cong
\indco_\alpha\bigl(\spec{X}(\res_\alpha U)\bigr)
,
\end{equation}
and this applies in particular to the case $G=\Gamma/N$, $\ind_\alpha(-)=(-)/N$ and $\coind_\alpha(-)=(-)^N$.
Notice that an analogous simplification does not exist for restriction, because there is no continuous functor $\MOR{f}{\Th^\Gamma}{\Th^G}$ such that $f\circ\res_{\Gamma\to1}=\res_{G\to1}$.
However, given an orthogonal $G$-spectrum~$\spec{X}$ the underlying continuous functor \mbox{$\Th^\Gamma\TO\CT^\Gamma$} of the orthogonal $\Gamma$-spectrum $\MOR{\res_\alpha\spec{X}}{\Th_\Gamma}{\CT_\Gamma}$ is isomorphic to the left Kan extension of
\begin{equation*}
\Th^G\TO[\spec{X}]\CT^G\xrightarrow{\res_\alpha}\CT^\Gamma
\end{equation*}
along $\MOR{\res_\alpha}{\Th^G}{\Th^\Gamma}$, by a similar argument.
Explicitly, for any \fd\ $G$-representation~$U$ we have
\begin{equation}
\label{eq:res-res}
\bigl(\res_\alpha\spec{X}\bigr)(\res_\alpha U)
\cong
\res_\alpha\bigl(\spec{X}(U)\bigr)
.
\end{equation}
\end{construction}

\begin{remark}[Naive is not naive]
\label{rem:naive-not-naive}
Going back and forth between non-naive and naive orthogonal spectra it is easy to turn change of group results for spaces into corresponding results for orthogonal spectra.
Essentially every natural formula or construction involving change of groups that we are used to for spaces remains valid for orthogonal spectra.
We illustrate this in two examples. 
\end{remark}

\begin{example}
\label{eg:ind-res}
For all naive orthogonal $G$-spectra~$\spec{X}$ there are natural $G$\=/isomorphisms
\[
\ind_\alpha\res_\alpha\spec{X}
\cong
G/\alpha(\Gamma)_+ \sma\spec{X}
\AND
\coind_\alpha\res_\alpha\spec{X}
\cong
\map(\alpha(\Gamma)\backslash G _+ ,\spec{X})
\,,
\]
induced by the space-level isomorphisms in~\eqref{eq:indco-res}.
The same isomorphisms then also hold for non-naive orthogonal $G$-spectra $\spec{X}$, even though the definitions of (co)induction and in particular of restriction are more complicated.
This follows at once from the last sentence in Theorem~\ref{thm:naive}, since the desired isomorphisms hold for the underlying naive orthogonal spectra.
\end{example}

\begin{example}
\label{eg:res-smash}
For spaces restriction is a strict monoidal functor, i.e., for all pointed $G$-spaces $W$ and~$X$,
\(
\res_\alpha W\sma\res_\alpha X
=
\res_\alpha(W\sma X)
\).
So it follows that for all pointed $G$-spaces~$W$ and naive orthogonal $G$-spectra~$\spec{X}$ there are natural $\Gamma$-isomorphisms
\[
\res_\alpha W\sma\res_\alpha\spec{X}
\cong
\res_\alpha(W\sma\spec{X})
\,,
\]
and by the same argument as above the same is also true for non-naive orthogonal $G$-spectra~$\spec{X}$.
\end{example}

\begin{construction}[Underlying non-equivariant fixed point spectra]
Let $H$ be a subgroup of~$G$, not necessarily normal.
For any orthogonal $G$-spectrum~$\spec{X}$ we denote by~$\res_1\spec{X}^H$ the underlying non-equivariant orthogonal spectrum given by the composition 
\[
\res_1\spec{X}^H
\colon
\Th\TO[\iota]\Th^G\TO[\spec{X}]\CT^G\xrightarrow{(-)^H}\CT
\,.
\]
\end{construction}

The next result collects some important properties of the change of group functors in the special case where $\alpha$ is the inclusion of a subgroup $H\leq G$.

\begin{theorem}
\label{thm:H-in-G}
Let $H$ be a subgroup of~$G$.
\begin{enumerate}
\item
\label{i:Wirth}
For any orthogonal $H$-spectrum $\spec{X}$ there is a natural \piiso
\[
\MOR{W}{\ind_{H\leq G}\spec{X}}{\coind_{H\leq G}\spec{X}}
\,,
\]
called \emph{Wirthm\"uller isomorphism}.



\item
\label{i:res-co-ind-piiso}
Restriction, induction, and coinduction
\[
\MOR{\res_{H\leq G}}{\Sp^G}{\Sp^H}
\,,\qquad
\MOR{\indco_{H\leq G}}{\Sp^H}{\Sp^G}
\]
all preserve \piiso s.

\item
\label{i:res-good}
If $\spec{X}$ is an orthogonal $G$-spectrum, then $\spec{X}$ is (almost) good if and only if $\res_{H\leq G}\spec{X}$ is (almost) good.

\item
\label{i:co-ind-good}
Induction preserves good spectra, and coinduction preserves almost good spectra.

\end{enumerate}
\end{theorem}

\begin{proof}
For~\ref{i:Wirth} see~\cite{Wirth}, \cite{tD-transf}*{Proposition~II.6.12 on page~148}, \cite{LMS}*{Theorem~II.6.2 on page~89}, and in particular~\cite{Schwede}*{Theorem~4.9 on page~29} for a proof in the setting of orthogonal spectra.
Here we only recall the definition of the Wirthm\"uller map~$W$, because its explicit description is needed in Section~\ref{sec:wirth}.
On the space level there is a natural transformation~$W$ between the functors
\[
\ind_{H\leq G}=G_+\sma_H-
\AND
\coind_{H\leq G}=\map(G_+,-)^H
\]
from $\CT^H$ to~$\CT^G$ introduced in \eqref{eq:ind} and~\eqref{eq:coind}.
For any pointed $H$-space~$X$, $W$ is defined as
\begin{align}
\label{eq:preWirth}
\MOR{W_X=W_X^{H\leq G}}{G_+\sma_H X}{&\map(G_+,X)^H}
,
\\
(g\sma x)\xmapsto{\quad}&\left(
\gamma\mapsto
\begin{cases}
(\gamma g) x
&\text{if $\gamma g\in H$}
\\
\pt
&\text{if $\gamma g\not\in H$}
\end{cases}
\,\right)
.
\end{align}
If $\MOR{\spec{X}}{\Th}{\CT^H}$ is a naive orthogonal $H$-spectrum, $W$ induces a $G$-map of naive orthogonal $G$-spectra $G_+\sma_H\spec{X}\TO\map(G_+,\spec{X})^H$.
Using Theorem~\ref{thm:naive}, as explained in Remark~\ref{rem:naive-not-naive}, we get for any non-naive orthogonal $H$-spectrum a $G$-map of orthogonal $G$-spectra
\begin{equation*}
\label{eq:Wirth}
\MOR{W_{\spec{X}}}{G_+\sma_H\spec{X}}{\map(G_+,\spec{X})^H}
,
\end{equation*}
that for each \fd\ $G$-representation~$U$ is given by
\[
\MOR{W_{\spec{X}(\res_{H\leq G}U)}}%
{G_+\sma_H\spec{X}(\res_{H\leq G}U)}%
{\map\bigl(G_+,\spec{X}(\res_{H\leq G}U)\bigr)^H}
\]
as in~\eqref{eq:preWirth}.



\ref{i:res-co-ind-piiso} For restriction notice that, if $\spec{X}$ is an orthogonal $G$-spectrum, then
\begin{equation}
\label{eq:pi-res}
\pi^H_*(\spec{X})\cong\pi^H_*(\res_{H\leq G}\spec{X})
\,.
\end{equation}
This follows easily from the definitions, using~\eqref{eq:res-res} and the fact that $\res_{H\leq G}\CU$ is a complete $H$-universe if $\CU$ is a complete $G$-universe.

For coinduction we use that, if $\spec{Y}$ is an orthogonal $H$-spectrum, then
\begin{equation}
\label{eq:pi-coind}
\pi^H_*(\spec{Y})\cong\pi^G_*(\coind_{H\leq G}\spec{Y})
\,,
\end{equation}
see for example~\cite{Schwede}*{(4.6) on page~28}.
The statement then follows from~\eqref{eq:pi-coind}, \eqref{eq:pi-res}, and the double coset formula; see for example~\cite{Schwede}*{first display on page~31}.

The statement that also induction preserves \piiso s follows then from the Wirthm\"uller isomorphism~\ref{i:Wirth}.

\ref{i:res-good} follows immediately from Lemma~\ref{lem:good}\ref{i:good-naive}.

\ref{i:co-ind-good} is implied by \ref{i:res-good} and Lemma~\ref{lem:good}\ref{i:(co)prod-good}, as for any orthogonal $H$-spectrum~$\spec{Y}$ we have $\res_{1\leq G}\ind_{H\leq G}\spec{Y}\cong\bigvee_{G/H}\spec{Y}$ and $\res_{1\leq G}\coind_{H\leq G}\spec{Y}\cong\prod_{H\backslash G}\spec{Y}$.
\end{proof}

We close this section with the following well-known lemma.
It says that for orthogonal $G$-$\Omega$-spectra the naive fixed points are well-behaved homotopically.
Notice that analogous results for orbits exist only under very strong cofibrancy assumptions; compare the proof of Corollary~\ref{cor:LMS}.
Since in the construction and proof of the Adams isomorphism we need to control homotopically both orbits and fixed points at the same time, we are led to use equivariant homotopy equivalences in the sense of Definition~\ref{def:homotopy}; see for example Proposition~\ref{prop:univchange}.

\begin{lemma}
\label{lem:Omega-fix}
Let $\spec{X}$ be an orthogonal $G$-$\Omega$-spectrum.
\begin{enumerate}

\item\label{i:pi-Omega-fix}
For all subgroups $H$ of~$G$ and for all~$n\in\IZ$ we have 
\[
\pi^H_n(\spec{X}) \cong \pi^1_n\bigl(\res_1\spec{X}^H\bigr) \cong
\begin{cases}
\pi_n\bigl(\spec{X}(0)^H\bigr)
&\text{if $n\geq0$\,;}
\\[\smallskipamount]
\pi_n\bigl(\spec{X}(\IR^{-n})^H\bigr)
&\text{if $n<0$\,.}
\end{cases}
\]

\item\label{i:Omega-fix}
If $N$ is a normal subgroup of~$G$ then $\spec{X}^N$ is a $G/N$-$\Omega$-spectrum, and for all subgroups $H$ of~$G$ containing~$N$ and for all~$n\in\IZ$ we have
\[
\pi^{H/N}_n\bigl(\spec{X}^N\bigr)
\cong
\pi^H_n(\spec{X})
\,.
\]

\item\label{i:pi-iso-Omega-fix}
If $\MOR{f}{\spec{X}}{\spec{Y}}$ is a \piiso\ of $G$-$\Omega$-spectra and $N$ is a normal subgroup of~$G$, then $\MOR{f^N}{\spec{X}^N}{\spec{Y}^N}$ is a \piiso\ of $G/N$-$\Omega$-spectra.
\end{enumerate}

\end{lemma}

\begin{proof}
\ref{i:pi-Omega-fix} follows immediately from Lemma~\ref{lem:Omega}\ref{i:pi-Omega}.

To prove~\ref{i:Omega-fix}, let $\MOR{p}{G}{G/N}$ be the projection.
We have to show that for all \fd\ $G/N$-representations~$U$ and $V$ the adjoint
$\widetilde{\sigma}^{\spec{X}^N}_{U, V}$ of the structure map \eqref{eq:structure-map} is a $G/N$-weak equivalence.
Using \eqref{eq:indco}, \eqref{eq:ind-orb-coind-fix}, and the isomorphism
\[
\Omega^U \spec{X}( p^*U \oplus p^*V)^N
\cong
\bigl( \Omega^{p^* U} \spec{X}( p^*U \oplus p^*V) \bigr)^N
,
\]
the map~$\widetilde{\sigma}^{\spec{X}^N}_{U, V}$ can be identified with
\[
\MOR{\bigl(\widetilde{\sigma}^{\spec{X}}_{p^*U, p^*V}\bigr)^N}%
{\spec{X}(p^*U)^N}%
{\bigl( \Omega^{p^* U} \spec{X}( p^*U \oplus p^*V) \bigr)^N}
.
\]
Since $\widetilde{\sigma}^{\spec{X}}_{p^*U, p^*V}$ is by assumption a $G$-weak equivalence, the claim follows.
So $\spec{X}^N$ is an orthogonal $G/N$-$\Omega$-spectrum, and using~\ref{i:pi-Omega-fix} we conclude that for all $n\in\IZ$
\[
\pi^{H/N}_n\bigl(\spec{X}^N\bigr)
\cong
\pi_n\bigl(\res_1(\spec{X}^N)^{H/N}\bigr)
=
\pi_n\bigl(\res_1\spec{X}^H\bigr)
\cong
\pi^H_n(\spec{X})
.
\]

Finally, \ref{i:pi-iso-Omega-fix} is an immediate consequence of \ref{i:pi-Omega-fix} and~\ref{i:Omega-fix}.
\end{proof}


\section{Functoriality of \texorpdfstring{$\Omega\operatorname{sh}$}{Omega-shift}}
\label{sec:Q-fin}

This section is devoted to the proof of the following proposition, on which the construction of~$\Q[h]{\CU}(\spec{X})$ in Theorem~\ref{thm:Q} is based.

\begin{proposition}
\label{prop:fun-in-U-fin}
Let $\spec{X}$ be an orthogonal $G$-spectrum.
Then the assignment
\[
U\longmapsto\Q{U}(\spec{X})=\Omega^U\shift^U\spec{X}
\]
from Definition~\ref{def:slsQ}
extends to a continuous $G$-functor
\[
\MOR{\Q{(-)}(\spec{X})}{\LL[fin]{G}}{\Sp_G}
\,,\]
i.e., for any pair of \fd\ $G$\=/representations $U$ and~$V$ there is  a $G$-map of orthogonal $G$-spectra
\begin{equation}
\label{eq:fun-in-U-fin}
\MOR{\varphi_{U,V}}{\CL(U,V)_+\sma\Q{U}(\spec{X})}{\Q{V}(\spec{X})}
\end{equation}
such that, for all \fd\ $G$-representations $U$, $V$, and~$W$, the diagram
\begin{equation}
\label{eq:fun-in-U-fin-assoc}
\begin{tikzcd}[column sep=large]
\CL(V,W)_+\sma\CL(U,V)_+\sma\Q{U}(\spec{X})
\arrow{r}{\id\sma\varphi_{U,V}}
\arrow{d}[swap]{\circ\sma\id}
&
\CL(V,W)_+\sma\Q{V}(\spec{X})
\arrow{d}{\varphi_{V,W}}
\\
\CL(U,W)_+\sma\Q{U}(\spec{X})
\arrow{r}[swap]{\varphi_{U,W}}
&
\Q{W}(\spec{X})
\end{tikzcd}
\end{equation}
commutes, and the composition
\begin{equation}
\label{eq:fun-in-U-fin-unit}
S^0\sma\Q{U}(\spec{X})
\xrightarrow{i\sma\id}
\CL(U,U)_+\sma\Q{U}(\spec{X})
\xrightarrow{\varphi_{U,U}}
\Q{U}(\spec{X})
\end{equation}
is the canonical isomorphism, where $i$ is the inclusion of the identity.
\end{proposition}

We first illustrate the idea behind the definition of the maps in~\eqref{eq:fun-in-U-fin}.
Let $U$ and~$V$ be \fd\ $G$-representations and let $f\in\CL(U,V)$ be a linear isometry.
Then $\varphi_{U,V}$ yields a map
\begin{equation*}
\MOR{f_*}{\Q{U}(\spec{X})}{\Q{V}(\spec{X})}
\,,
\end{equation*}
and $f_*$ is $G$-equivariant if~$f$~is, i.e., if $f\in\CL(U,V)^G$.
Here are two examples.

\begin{example}
\label{eg:fun-in-U-iso}
If $\dim U=\dim V$ then~$f_*$ is induced by ``conjugation by~$f$'', i.e., for any \fd\ $G$-representation~$Z$, the map
\begin{equation*}
\MOR{f_*}%
{\Omega^U\spec{X}(U\oplus Z)=\Q{U}(\spec{X})(Z)}%
{\Q{V}(\spec{X})(Z)=\Omega^V\spec{X}(V\oplus Z)}
\end{equation*}
sends $\MOR{\omega}{S^U}{\spec{X}(U\oplus Z)}$ to the composition
\[
S^V
\xrightarrow{S^{f^{-1}}}
S^U
\TO[\omega]
\spec{X}(U\oplus Z)
\xrightarrow{\spec{X}(f\oplus\id)}
\spec{X}(V\oplus Z)
\,.
\]
\end{example}

\begin{example}
\label{eg:fun-in-U-incl}
If $U$ is a $G$-subrepresentation of~$V$ and $f$ is the $G$-equivariant inclusion, then $f_*$ is the map
\[
\MOR{\tau_{U\subseteq V}}{\Q{U}(\spec{X})}{\Q{V}(\spec{X})}
\]
defined in~\eqref{eq:tau}.
In particular, if $U=0$ and $f$ is the zero map to $V$, then the composition of~$f_*$ with the natural isomorphism $\spec{X}\cong\Q{0}(\spec{X})$ is
$\widetilde{\sigma}_V$,
the adjoint of the spectrum-level structure map defined in~\eqref{eq:sigma-adj}; compare Remark~\ref{rem:tau}.
\end{example}

For the proof of Proposition~\ref{prop:fun-in-U-fin} we need the following observation about the functoriality in~$U$ of~$\shift^U\spec{X}$, which follows at once from Lemma~\ref{lem:plus}.
Notice the difference in the source categories between Proposition~\ref{prop:fun-in-U-fin} and Lemma~\ref{lem:fun-sh}.

\begin{lemma}
\label{lem:fun-sh}
Let $\spec{X}$ be an orthogonal $G$-spectrum.
Then the assignment
\[
U\longmapsto\shift^U\spec{X}=\spec{X}\circ(U\oplus-)
\]
extends to a continuous $G$-functor
\[
\MOR{\shift^{(-)}\spec{X}}{\Th_G}{\Sp_G}
\,,\]
i.e., for any pair of \fd\ $G$\=/representations $U$ and~$V$ there is $G$-map of orthogonal $G$-spectra
\begin{equation*}
\label{eq:fun-sh}
\MOR{\psi_{U,V}}{\Th(U,V)\sma\shift^U\spec{X}}{\shift^V\spec{X}}
\end{equation*}
satisfying the associativity and unitality axioms.
\end{lemma}

\begin{remark}
From the commutativity of diagram~\eqref{eq:plus-a} in Lemma~\ref{lem:plus} it follows that when $U=0$ the composition
\begin{equation*}
S^V\sma\spec{X}
\cong
\Th(0,V)\sma\shift^0\spec{X}
\xrightarrow{\psi_{0,V}}
\shift^V\spec{X}
\end{equation*}
is the spectrum-level structure map $\sigma_V$ defined in~\eqref{eq:sigma}, where the isomorphism in the source is the evident one.
Similarly, from the commutativity of diagram~\eqref{eq:plus-a-sub} it follows that when $U$ is a $G$-subrepresentation of~$V$ the composition
\begin{equation*}
S^{V-U}\sma\spec{X}
\xrightarrow{a_{U\subseteq V}\sma\id}
\Th(U,V)\sma\shift^U\spec{X}
\xrightarrow{\psi_{U,V}}
\shift^V\spec{X}
\end{equation*}
is the internal spectrum-level structure map $\sigma_{U\subseteq V}$ defined in~\eqref{eq:sigma-sub}, where $a_{U\subseteq V}$ is defined in~\eqref{eq:S->Th-sub}.
\end{remark}

\newcommand*{\varphiadj}{\widehat{\varphi}}

\begin{proof}[Proof of Proposition~\ref{prop:fun-in-U-fin}]
In order to construct $G$-maps of orthogonal $G$-spectra
\[
\MOR{\varphi_{U,V}}{\CL(U,V)_+\sma\Q{U}(\spec{X})}{\Q{V}(\spec{X})}
\]
it is enough, by adjunction, to construct $G$-maps
\[
\MOR{\varphiadj_{U,V}}{S^V\sma\CL(U,V)_+\sma\Omega^U\shift^U\spec{X}}{\shift^V\spec{X}}
\,.
\]
We define~$\varphiadj_{U,V}$ as the following composition:
\begin{equation}
\label{eq:phi-adj}
\begin{tikzcd}[row sep=scriptsize]
S^V\sma\CL(U,V)_+\sma\Omega^U\shift^U\spec{X}
\arrow{d}{\ \untw^{-1}\sma\id}[swap]{\cong}
\\
\Th(U,V)\sma S^U\sma\Omega^U\shift^U\spec{X}
\arrow{d}{\ \id\sma\ev}
\\
\Th(U,V)\sma\shift^U\spec{X}
\arrow{d}{\ \psi_{U,V}}
\\
\shift^V\spec{X}
\mathrlap{\,,}
\end{tikzcd}
\end{equation}
where $\psi_{U,V}$ is the map from Lemma~\ref{lem:fun-sh}.
Since all the maps above are $G$-equivariant, the same is true for $\varphiadj_{U,V}$ and $\varphi_{U,V}$.

From the definition it is clear that the composition in~\eqref{eq:fun-in-U-fin-unit} is the canonical isomorphism.
It thus remains to show that diagram~\eqref{eq:fun-in-U-fin-assoc} commutes, or equivalently, that the adjoint diagram
\begin{equation}
\label{eq:adjoint-diagram}
\begin{tikzcd}[column sep=huge]
S^W \sma \CL(V,W)_+\sma\CL(U,V)_+\sma\Omega^U\shift^U\spec{X}
\arrow{r}{\id\ssma\id\ssma\varphi_{U,V}}
\arrow{d}[swap]{\id\sma\circ\sma\id}
&
S^W\sma\CL(V,W)_+\sma\Omega^V\shift^V\spec{X}
\arrow{d}{\varphiadj_{V,W}}
\\
S^W\sma\CL(U,W)_+\sma\Omega^U\shift^U\spec{X}
\arrow{r}[swap]{\varphiadj_{U,W}}
&
\shift^W\spec{X}
\end{tikzcd}
\end{equation}
commutes.

Now consider the diagram~\eqref{eq:huge-diagram} on page~\pageref{eq:huge-diagram}, where we use the abbreviations $\CL_{U,V}=\CL(U,V)_+$ and $\Th_{U,V}=\Th(U,V)$.
The diagrams labeled \framed{C}, \framed{E}, and \framed{H} commute by definition of~$\varphiadj$, and so the outer boundary of~\eqref{eq:huge-diagram} is just diagram~\eqref{eq:adjoint-diagram}.
Moreover:
\framed{A} and~\framed{F} evidently commute;
\framed{B}~commutes because untwisting is compatible with composition by Lemma~\ref{lem:untw-comp};
\framed{D}~commutes because $\ev$ is the counit of the adjunction;
\framed{G}~commutes because $\shift^{(-)}\spec{X}$ is a functor $\Th_G\TO\Sp_G$ by Lemma~\ref{lem:fun-sh}.

This proves the commutativity of diagrams \eqref{eq:huge-diagram} and~\eqref{eq:adjoint-diagram}, and therefore also of~\eqref{eq:fun-in-U-fin-assoc}.
\end{proof}

\begin{landscape}
\begin{equation}
\label{eq:huge-diagram}
\end{equation}
\[
\begin{tikzcd}[column sep=.85em]
S^W \ssma \CL_{V,W}\ssma\CL_{U,V}\ssma\Q{U}\,\spec{X}
\arrow{rrrrrr}{\id\sma\id\sma\varphi_{U,V}}
\arrow{dddddd}[swap]{\id\sma\circ\sma\id}
\arrow{rrdd}[description]{\untw^{-1}\sma\id\sma\id}
&
&
&
&
&
&
S^W\ssma\CL_{V,W}\ssma\Q{V}\,\spec{X}
\arrow{dd}[swap]{\untw^{-1}\sma\id}
\arrow[rounded corners,
       to path={    (\tikztostart.east)
                 -| +(1,-4) [at end]\tikztonodes
                 |- (\tikztotarget.east)}
      ]{dddddd}{\varphiadj_{V,W}}
\\
&&&&{\framed{A}}
\\
&
&
\Th_{V,W}\ssma S^V\ssma\CL_{U,V}\ssma\Q{U}\,\spec{X}
\arrow{rrrr}{\id\sma\id\sma\varphi_{U,V}}
\arrow{dd}[swap]{\id\sma\untw^{-1}\sma\id}
\arrow{ddrrrr}[description]{\id\sma\varphiadj_{U,V}}
&
&
&
&
\Th_{V,W}\ssma S^V\ssma\Q{V}\,\spec{X}
\arrow{dd}[swap]{\id\sma\ev}
\\
&\mathllap{\framed{B}\quad}&&{\framed{C}}&&{\framed{D}}&&\mathllap{\framed{E}\ \quad}
\\
&
&
\Th_{V,W}\ssma\Th_{U,V}\ssma S^U\ssma\Q{U}\,\spec{X}
\arrow{rr}[swap]{\id\ssma\id\ssma\ev}
\arrow{dd}[swap]{\circ\sma\id\sma\id}
&
&
\Th_{V,W}\ssma\Th_{U,V}\ssma\shift^U\!\spec{X}
\arrow{rr}[swap]{\id\sma\psi_{U,V}}
\arrow{dd}[swap]{\circ\sma\id}
&
&
\Th_{V,W}\ssma\shift^V\!\spec{X}
\arrow{dd}[swap]{\psi_{V,W}}
\\
&&&{\framed{F}}&&{\framed{G}}
\\
S^W\ssma\CL_{U,W}\ssma\Q{U}\,\spec{X}
\arrow{rr}[swap]{\untw^{-1}\sma\id}
\arrow[rounded corners,
       to path={    (\tikztostart.south)
                 |- +(8,-2) [at end]\tikztonodes
                 -| (\tikztotarget.south)}
      ]{rrrrrr}[swap]{\varphiadj_{U,W}}
&
&
\Th_{U,W}\ssma S^U\ssma\Q{U}\,\spec{X}
\arrow{rr}[swap]{\id\sma\ev}
&
&
\Th_{U,W}\ssma\shift^U\!\spec{X}
\arrow{rr}[swap]{\psi_{U,W}}
&
&
\shift^W\!\spec{X}
\\
&&&{\framed{H}}
\end{tikzcd}
\]
\end{landscape}

We close this section with the following observation.

\begin{lemma}
\label{lem:iso-bifun}
For all \fd\ $G$-representations $U$ and~$W$ there is a $G$\=/isomorphism
\[
\MOR[\cong]{\vartheta_{U,W}}{\Q{W}\Q{U}(\spec{X})}{\Q{U\oplus W}(\spec{X})}
\]
which is a natural isomorphism of bifunctors, i.e., the following diagram commutes.
\begin{equation}
\label{eq:iso-bifun}
\begin{tikzcd}
\CL(U,V)_+\sma\CL(W,Z)_+\sma\Q{W}\Q{U}(\spec{X})
\arrow{r}
\arrow{d}[swap]{-\oplus-\sma\vartheta_{U,W}}
&
\Q{Z}\Q{V}(\spec{X})
\arrow{d}{\vartheta_{V,Z}}
\\
\CL(U\oplus W,V\oplus Z)_+\sma\Q{U\oplus W}(\spec{X})
\arrow{r}
&
\Q{V\oplus Z}(\spec{X})
\end{tikzcd}
\end{equation}
\end{lemma}

\begin{proof}
The isomorphism~$\vartheta_{U,W}$ comes from the fact $\shift^W\Omega^U=\Omega^U\shift^W$ combined with the natural isomorphisms $\Omega^W\Omega^U\cong\Omega^{U\oplus W}$ and $\shift^W\shift^U\cong\shift^{U\oplus W}$.
To check that \eqref{eq:iso-bifun} commutes, we consider the adjoint diagram, as in the proof of Proposition~\ref{prop:fun-in-U-fin}.
This adjoint diagram commutes because all three maps in~\eqref{eq:phi-adj} satisfy an appropriate compatibility with direct sums.
\end{proof}


\section{Bifunctorial replacements by \texorpdfstring{$\Omega$}{Omega}-spectra}
\label{sec:Q}

Let $\spec{X}$ be an orthogonal $G$-spectrum and let $\CU$ be a $G$-representation, not necessarily finite dimensional.
This section is devoted to the construction of~$\Q{\CU}(\spec{X})$ and to the proof of Theorem~\ref{thm:Q}.
The key ingredient for the construction is given by Proposition~\ref{prop:fun-in-U-fin} in the previous section, which says that $\MOR{\Q{(-)}(\spec{X})}{\LL[fin]{G}}{\Sp_G}$
is a continuous $G$-functor.

\begin{definition}
\label{def:Q}
Given an orthogonal $G$-spectrum~$\spec{X}$ and a $G$-representation~$\CU$, define
\[
\Q{\CU}(\spec{X})
=
\biggl(\bigLan_{\LL[fin]{G}\subset\LL{G}}\Q{(-)}(\spec{X})\biggr) (\CU)
\]
to be the evaluation at~$\CU$ of the left Kan extension of $\MOR{\Q{(-)}(\spec{X})}{\LL[fin]{G}}{\Sp_G}$ along the inclusion $\LL[fin]{G}\subset\LL{G}$.
The isomorphism $\spec{X}\cong\Q{0}(\spec{X})$ together with the inclusion $0\subseteq\CU$ gives a natural $G$-map
\(
\MOR{\fr}{\spec{X}}{\Q{\CU}(\spec{X})}
\).
\end{definition}

Left Kan extensions are reviewed in the proof of Proposition~\ref{prop:colim->Lan}.
Since Kan extensions of continuous $G$-functors are again continuous $G$-functors, Definition~\ref{def:Q} immediately implies part~\ref{i:fun} of Theorem~\ref{thm:Q} and Corollary~\ref{cor:fun-in-U}.
Here is another consequence of continuity.

\begin{corollary}
\label{cor:Q-2-universes}
If $\MOR{f}{\CU}{\CV}$ is a $G$-equivariant isometry between complete $G$\=/universes, then the induced map $\MOR{f_*}{\Q{\CU}(\spec{X})}{\Q{\CV}(\spec{X})}$ is $G$-homotopic to an isomorphism.
\end{corollary}

\begin{proof}
Since $\CU$ and~$\CV$ are complete $G$-universes, there is also a $G$-isomorphism $\MOR{f'}{\CU}{\CV}$; compare Definition~\ref{def:universe}.
By Lemma~\ref{lem:LMS-II.1.5} there is a $G$-homotopy between $f$ and~$f'$.
The fact that $\Q{(-)}(\spec{X})$ is a continuous $G$-functor implies that $f_*$ is $G$-homotopic to the isomorphism $f'_*$.
\end{proof}

The remaining parts of Theorem~\ref{thm:Q} rely on the following result.
Recall that $\CP(\CU)$ denotes the poset of all \fd\ $G$-subrepresentations of~$\CU$ ordered by inclusion, and consider $\CP(\CU)$ as a discrete subcategory of~$\LL[fin]{G}$.

\begin{proposition}
\label{prop:colim->Lan}
For any orthogonal $G$-spectrum~$\spec{X}$ and any $G$-representation~$\CU$ there is a natural $G$-isomorphism
\[
{\colim_{\CP(\CU)}\Q{(-)}(\spec{X})}
\TO[\cong]
{\biggl(\bigLan_{\LL[fin]{G}\subset\LL{G}}\Q{(-)}(\spec{X})\biggr) (\CU)
=\Q{\CU}(\spec{X})}
\,.
\]
\end{proposition}

\begin{proof}
This is based on the following general principle.
Consider a diagram of topological $G$-categories and continuous $G$-functors
\begin{equation*}
\begin{tikzcd}
\CA
\arrow{r}{\gamma}
\arrow{d}[swap]{\alpha}
&
\CC
\arrow{rr}{\spec{Y}}
\arrow{d}{\delta}
\arrow[Rightarrow, shorten <=1em, shorten >=1em]{ld}
&
&
\Sp_G
\\
\CB
\arrow{r}[swap]{\beta}
&
\CD
\end{tikzcd}
\end{equation*}
and let $\nu$ be a natural transformation from $\delta\gamma$ to $\beta\alpha$.
Then $\nu$ induces a natural transformation
\[
\MOR{\nu_*}{\Lan_\alpha\res_\gamma\spec{Y}}{\res_\beta\Lan_\delta\spec{Y}}
\]
as we proceed to explain.
First we need some notation.
Given continuous $G$-functors
\(
\MOR{W}{\CC^\op}{\CT_G}
\)
and
\(
\MOR{Y}{\CC}{\CT_G}
\)
we write $W\sma_\CC Y$ for their coend; compare for example~\cite{HV}*{Definition~2.3 on page~116}.
If $\MOR{\spec{Y}}{\CC}{\Sp_G}$ is a continuous $G$-functor, we can think of~$\spec{Y}$ as a continuous $G$-functor $\MOR{\spec{Y}}{\CC\times\Th_G}{\CT_G}$, and so
\(
W\sma_\CC\spec{Y}
\)
is a continuous $G$-functor $\Th_G\TO\CT_G$, i.e., an orthogonal $G$-spectrum.
If $\CD$ is another topological $G$-category and $\MOR{W}{\CC^\op\times\CD}{\CT_G}$ is a continuous $G$-functor, then we get a continuous $G$-functor
\(
\MOR{W\sma_\CC\spec{Y}}{\CD}{\Sp_G}
\).
For a continuous $G$-functor $\MOR{\delta}{\CC}{\CD}$ define 
\begin{align*}
\MOR{\modu{}{\CD}{\delta}}{\CC^\op\times\CD}{\CT_G}
,\quad
(c,d)&\mapsto\CD\bigl(\delta(c),d\bigr)
\,,
\\
\MOR{\modu{\delta}{\CD}{}}{\CD^\op\times\CC}{\CT_G}
,\quad
(d,c)&\mapsto\CD\bigl(d,\delta(c)\bigr)
\,.
\end{align*}
Notice that by Yoneda Lemma, for any continuous $G$-functor $\MOR{\spec{Z}}{\CD}{\Sp_G}$, we have $\res_\delta\spec{Z}\cong\modu{\delta}{\CD}{}\sma_\CD\spec{Z}$.
Recall that
\(
\Lan_\delta\spec{Y}=\CD_\delta\sma_\CC\spec{Y}
\).

Therefore we have
\begin{align*}
\Lan_\alpha\res_\gamma\spec{Y}
&\cong
\modu{}{\CB}{\alpha} \sma_\CA \bigl(\modu{\gamma}{\CC}{} \sma_\CC\spec{Y}\bigr)
\cong
\bigl(\modu{}{\CB}{\alpha} \sma_\CA \modu{\gamma}{\CC}{}\bigr) \sma_\CC\spec{Y}
,\AND
\\
\res_\beta\Lan_\delta\spec{Y}
&\cong
\modu{\beta}{\CD}{} \sma_\CD \bigl(\modu{}{\CD}{\delta} \sma_\CC\spec{Y}\bigr)
\cong
\bigl(\modu{\beta}{\CD}{} \sma_\CD \modu{}{\CD}{\delta}\bigr) \sma_\CC\spec{Y}
\cong
\modu{\beta}{\CD}{\delta} \sma_\CC\spec{Y}
.
\end{align*}
Now let $\overline{\nu}$ be the composition
\begin{equation}
\label{eq:overline-nu}
\overline{\nu}
\colon
\modu{}{\CB}{\alpha} \sma_\CA \modu{\gamma}{\CC}{}
\TO
\modu{\beta}{\CD}{\beta\alpha} \sma_\CA \modu{\beta\alpha}{\CD}{\delta}
\TO
\modu{\beta}{\CD}{} \sma_\CD \modu{}{\CD}{\delta}
\cong
\modu{\beta}{\CD}{\delta}
\end{equation}
where the first map is induced by the functors $\beta$ and~$\delta$ together with the natural transformation~$\nu$.
Then define~$\nu_*$ by applying $-\sma_\CC\spec{Y}$ to~$\overline{\nu}$ and using the isomorphisms above.
It is then clear that if $\overline{\nu}$ is a natural $G$-isomorphism of continuous $G$-functors $\CC^\op\times\CB\TO\CT_G$, then $\nu_*$ is a natural $G$-isomorphism of continuous $G$-functors $\CB\TO\Sp_G$, natural in~$\spec{Y}$.

We now apply this general principle to the diagram
\begin{equation*}
\label{eq:colim->Lan}
\begin{tikzcd}
\CP(\CU)
\arrow[hookrightarrow]{r}
\arrow{d}[swap]{\pr}
&
\LL[fin]{G}
\arrow{rr}{\Q{(-)}(\spec{X})}
\arrow[hookrightarrow]{d}
\arrow[Rightarrow, shorten <=1em, shorten >=1em]{ld}
&
&
\Sp_G
\\
\{\CU\}
\arrow[hookrightarrow]{r}
&
\LL{G}
\arrow{rru}[swap]{\Lan\Q{(-)}(\spec{X})}
\end{tikzcd}
\end{equation*}
where $\{\CU\}$ denotes the trivial category with only one object and one morphism, and the hooked arrows denote the obvious inclusion functors.
Notice that the left-hand square in this diagram does not commute, but there is an obvious natural transformation~$\nu$ in the indicated direction, given by the inclusions $U\subset\CU$.

Since $\Lan_{\pr}=\colim_{\CP(\CU)}$ we get a natural $G$-map of orthogonal $G$-spectra
\[
\MOR{\nu_*}%
{\colim_{\CP(\CU)}\Q{(-)}(\spec{X})}
{\biggl(\bigLan_{\LL[fin]{G}\subset\LL{G}}\Q{(-)}(\spec{X})\biggr) (\CU)
=\Q{\CU}(\spec{X})}
\,.
\]
In the case at hand the map~$\overline{\nu}$ from~\eqref{eq:overline-nu} is
\[
\MOR{\overline{\nu}}{\colim_{U\in\CP(\CU)}\CL(V,U)}{\CL(V,\CU)}
\]
for any \fd\ $G$-representation~$V$.
Since $\overline{\nu}$ is a $G$-isomorphism, natural in~$V$, the result follows.
\end{proof}

In the proofs below of parts \ref{i:repl}, \ref{i:Omega}, and~\ref{i:res-Q} of Theorem~\ref{thm:Q} we use Proposition~\ref{prop:colim->Lan} and identify $\Q{\CU}(\spec{X})\cong\colim_{\CP(\CU)}\Q{(-)}(\spec{X})$.
We begin with a well-known lemma.

\begin{lemma}
\label{lem:Omega-sh}
For any orthogonal $G$-spectrum~$\spec{X}$ and any \fd\ $G$\=/representation~$U$ the spectrum-level structure map and its adjoint
\[
\MOR{\sigma_U}{\Sigma^U\spec{X}}{\shift^U\spec{X}}
\AND
\MOR{\widetilde{\sigma}_U}{\spec{X}}{\Q{U}(\spec{X})}
\]
defined in \eqref{eq:sigma} and~\eqref{eq:sigma-adj} are both \piiso s.
For any \fd\ $G$-representation~$V$ such that $U\subseteq V$, the map
\[
\MOR{\tau_{U\subseteq V}}{\Q{U}(\spec{X})}{\Q{V}(\spec{X})}
\]
defined in~\eqref{eq:tau} is also a \piiso.
\end{lemma}

\begin{proof}
For the first statement, see~\cite{Schwede}*{Proposition~3.17(ii) on page~19}.
The second statement then follows from the commutativity of diagram~\eqref{eq:triangle-Q}.
\end{proof}

\begin{proof}[Proof of Theorem~\ref{thm:Q}\ref{i:repl}]
By Lemmas \ref{lem:good}\ref{i:good-tau} and~\ref{lem:Omega-sh}, for all \fd\ $G$\=/representations $U\subseteq V$ the map~$\MOR{\tau_{U\subseteq V}}{\Q{U}(\spec{X})}{\Q{V}(\spec{X})}$ is a levelwise closed embedding and a \piiso.
The conclusion then follows from the following more general claim: If $\CP$ is a poset containing a cofinal sequence and $\MOR{\spec{F}}{\CP}{\Sp^G}$ is a functor such that for all $p\leq q$ in~$\CP$ the map $\spec{F}_p\TO\spec{F}_q$ is a levelwise closed embedding and a \piiso, then for all $o\in\CP$ the map
\(
\spec{F}_o\TO\colim_\CP\spec{F}
\)
is a \piiso.
To prove this claim consider a subgroup $H$ of~$G$ and $n\geq 0$.
Then we have:
\begin{align*}
\pi^H_n\spec{F}_o
&\cong
\colim_{p\in\CP}\pi^H_n\spec{F}_p
&\text{since $\spec{F}_p\TO\spec{F}_q$ is a \piiso,}
\\
&=
\colim_{p\in\CP}\colim_{U\in\CP(\CU)}\pi_n\Bigl(\bigl(\Omega^U\spec{F}_p(U)\bigr)^H\Bigr)
&\text{by definition,}
\\
&\cong
\colim_{U\in\CP(\CU)}\colim_{p\in\CP}\pi_n\Bigl(\bigl(\Omega^U\spec{F}_p(U)\bigr)^H\Bigr)
&\text{by commuting the colimits,}
\\
&\cong
\colim_{U\in\CP(\CU)}\pi_n\Bigl(\colim_{p\in\CP}\bigl(\Omega^U\spec{F}_p(U)\bigr)^H\Bigr)
&\one
\\
&\cong
\colim_{U\in\CP(\CU)}\pi_n\Bigl(\Bigl(\colim_{p\in\CP}\Omega^U\spec{F}_p(U)\Bigr)^H\Bigr)
&\two
\\
&\cong
\colim_{U\in\CP(\CU)}\pi_n\Bigl(\Bigl(\Omega^U\colim_{p\in\CP}\spec{F}_p(U)\Bigr)^H\Bigr)
\mathrlap{\,=\pi^H_n\Bigl(\colim_{p\in\CP}\spec{F}_p\Bigr)}
&\three
\end{align*}
For the isomorphisms $\one$, $\two$, and~$\three$ we use the fact that~$\CP$ contains a cofinal sequence.
$\three$~follows from Fact~\ref{facts:clemb}\ref{i:colim-clemb};
$\one$~follows from the same fact applied to the functor
\(
p\mapsto\bigl(\Omega^U\spec{F}_p(U)\bigr)^H
\),
which also sends any $p\leq q$ to a closed embedding by Facts~\ref{facts:clemb}\ref{i:sma-map-clemb} and~\ref{i:fix-clemb};
and finally $\two$ follows from Fact~\ref{facts:clemb}\ref{i:fix-colim-clemb}.
The case $n<0$ is proved analogously.
\end{proof}

\begin{proof}[Proof of Theorem~\ref{thm:Q}\ref{i:Omega}]
We begin with a general observation.
Let $\CP$ be a poset containing a cofinal sequence and let $\MOR{\spec{F}}{\CP}{\Sp^G}$ be a functor such that for all $p\leq q$ in~$\CP$ the map $\spec{F}_p\TO\spec{F}_q$ is a levelwise closed embedding.
For any \fd\ $G$-representation~$V$, the adjoint~\eqref{eq:sigma-adj} of the spectrum level structure map for $\colim_\CP\spec{F}$ factors as
\[
\begin{tikzcd}[column sep=huge]
\ds\colim_\CP\spec{F}
\arrow{r}[swap]{\colim\widetilde{\sigma}_V}
\arrow[rounded corners,
       to path={    (\tikztostart.north)
                 |- +(4,.4) [at end]\tikztonodes
                 -| (\tikztotarget.north)}
      ]{rr}{\widetilde{\sigma}_V}
&
\ds\colim_\CP \Q{V}(\spec{F})
\arrow{r}[description]{\cong}
&
\ds\Q{V}\bigl(\colim_\CP\spec{F}\bigr)
\,,
\end{tikzcd}
\]
where the isomorphism on the right follows from Fact~\ref{facts:clemb}\ref{i:colim-clemb}.
Therefore $\colim_\CP\spec{F}$ is an orthogonal $G$-$\Omega$-spectrum if and only if the map $\colim\widetilde{\sigma}_V$ is a level equivalence for all~$V$.

We now apply this to $\CP=\CP(\CU)$ and $F=\Q{(-)}(\spec{X})$.
So we need to show that the top horizontal map in the following diagram is a level equivalence.
\begin{equation}
\label{eq:Q-is-omega}
\begin{tikzcd}[row sep=small, column sep=-2em]
\ds\Q{\CU}(\spec{X})\cong
\colim_{U\in\CP(\CU)}\Q{U}(\spec{X})
\arrow{rr}{\colim\widetilde{\sigma}_V}
\arrow[equal]{d}{\wr}
&&
\ds\colim_{U\in\CP(\CU)}\Q{V}\Q{U}(\spec{X})
\arrow[equal]{d}{\wr}
\\
\ds\colim_{U\in\CP(\CU)}\Q{U\oplus 0}(\spec{X})
\arrow{rr}{\ds\three}
\arrow{rdd}[swap]{\ds\one}
&&
\ds\colim_{U\in\CP(\CU)}\Q{U\oplus V}(\spec{X})
\arrow{ldd}{\ds\two}
\\
{}
\\
&
\ds\colim_{W\in\CP(\CU\oplus V)}\Q{W}(\spec{X})
\cong
\Q{\CU\oplus V}(\spec{X})
\end{tikzcd}
\end{equation}
The vertical maps in the square are induced from the isomorphisms in Lemma~\ref{lem:iso-bifun}, and the same lemma also implies that the square exists and commutes.
The maps $\one$ and~$\two$ are induced by the functors $\bone$ and~$\btwo$
\[
\begin{tikzcd}[column sep=small]
U
\arrow[mapsto]{d}
&
\CP(\CU)
\arrow[xshift=-.8ex, swap]{d}{\ds\bone}
\arrow[xshift=+.8ex      ]{d}{\ds\btwo}
&
U
\arrow[mapsto]{d}
\\
U\oplus 0
&
\CP(\CU\oplus V)
&
U\oplus V
\end{tikzcd}
\]
and the map~$\three$ is induced by the natural transformation $U\oplus 0\subseteq U\oplus V$ from $\bone$~to~$\btwo$.

The functor~$\btwo$ is cofinal, since for any $W\in\CP(\CU\oplus V)$ the comma category $W\downarrow\btwo$ has an initial object given by the image of~$W$ under $\MOR{\pr_1}{\CU\oplus V}{\CU}$.
Therefore the map~$\two$ is a $G$-isomorphism for all~$V$.

Notice that the functor~$\bone$ is not cofinal: if $V\neq0$ then $(0\oplus V)\downarrow\bone=\emptyset$.
But we argue that the composition from $\Q{\CU}(\spec{X})$ to $\Q{\CU\oplus V}(\spec{X})$ along the left side of~\eqref{eq:Q-is-omega} is $G$-homotopic to an isomorphism, and so in particular it is a level equivalence.
In fact, this composition is induced by the $G$-equivariant isometry $\MOR{f}{\CU}{\CU\oplus V}$ given by the inclusion of the first summand.
Since $\CU\oplus V$ is again a complete $G$-universe, the result then follows from Corollary~\ref{cor:Q-2-universes}.
\end{proof}

Notice that in the proofs of Theorem~\ref{thm:Q}\ref{i:repl} and~\ref{i:Omega} we only use the assumption that $\spec{X}$ is good to conclude that for all \fd\ $G$-representations $U\subseteq V$ the map $\MOR{\tau_{U\subseteq V}}{\Q{U}(\spec{X})}{\Q{V}(\spec{X})}$ is a levelwise closed embedding.
Since this is true also for almost good spectra by Lemma~\ref{lem:good}\ref{i:good-tau}, our proofs yield the following slightly stronger statement, which is needed in the proof of Key Lemma~\ref{lem:key}.

\begin{addendum}
\label{add:almost}
Parts \ref{i:repl} and~\ref{i:Omega} of Theorem~\ref{thm:Q} remain true if $\spec{X}$ is only assumed to be almost good.
\end{addendum}

\begin{proof}[Proof of Theorem~\ref{thm:Q}\ref{i:res-Q}]
Let $\MOR{\alpha}{\Gamma}{G}$ be a group homomorphism.
By Theorem~\ref{thm:naive} it is enough to show that there is a natural $\Gamma$-isomorphism between the underlying naive orthogonal $\Gamma$-spectra of $\res_\alpha\Q{\CU}(\spec{X})$ and $\Q{\res_\alpha\CU}(\spec{X})$.
Let $T$ be a \fd\ inner product space, viewed as a trivial $\Gamma$-representation.
Then
\[
\bigl(\res_\alpha\Q{\CU}(\spec{X})\bigr)(T)
=
\res_\alpha\colim_{U\in\CP(\CU)}\bigl(\Q{U}(\spec{X})(T)\bigr)
=
\colim_{U\in\CP(\CU)}\res_\alpha\bigl(\Q{U}(\spec{X})(T)\bigr)
\,.
\]
Moreover
\[
\res_\alpha\bigl(\Q{U}(\spec{X})(T)\bigr)
=
\res_\alpha\bigl(\Omega^{U}\spec{X}(U\oplus T)\bigr)
\cong
\Omega^{\res_\alpha U}\res_\alpha\bigl(\spec{X}(U\oplus T)\bigr)
\]
and by~\eqref{eq:res-res} we have that
\begin{multline*}
\res_\alpha\bigl(\spec{X}(U\oplus T)\bigr)
\cong
(\res_\alpha\spec{X})(\res_\alpha(U\oplus T))
\cong
(\res_\alpha\spec{X})((\res_\alpha U)\oplus T)
\\
=
\bigl(\shift^{\res_\alpha U}(\res_\alpha\spec{X})\bigr)(T)
\end{multline*}
where in the last two terms $T$ is viewed as a trivial $G$-representation.
Therefore
\(
\res_\alpha\bigl(\Q{U}(\spec{X})(T)\bigr)
\cong
\Q{\res_\alpha U}(\res_\alpha\spec{X})(T)
\).
Since for any $V\in\CP(\res_\alpha\CU)$ there is a $U\in\CP(\CU)$ such that $V \subseteq \res_\alpha U$, the poset $\SET{\res_\alpha U}{U \in \CP(\CU)}$ is cofinal in $\CP(\res_\alpha\CU)$, and the result follows.
\end{proof}

We close this section with another consequence of Theorem~\ref{thm:Q}.

\begin{corollary}
\label{cor:pi(Q)}
If $\spec{X}$ is a good orthogonal $G$-spectrum and $\CU$ is a complete $G$-universe, then for all subgroups $H$ of~$G$ and all integers~$n$ there are isomorphisms
\[
\pi_n^H(\spec{X})
\cong
\pi_n^1\bigl(\res_1\Q[h]{\CU}(\spec{X})^H\bigr)
.
\]
\end{corollary}

\begin{proof}
We have
\[
\pi_n^H(\spec{X})
\cong
\pi_n^H(\Q[h]{\CU}(\spec{X}))
\cong
\pi_n^1\bigl(\res_1\Q[h]{\CU}(\spec{X})^H\bigr),
\]
where the first isomorphism follows from Theorem~\ref{thm:Q}\ref{i:repl} and the second from Lemma~\ref{lem:Omega-fix}\ref{i:pi-Omega} since $\Q[h]{\CU}(\spec{X})$ is a $G$-$\Omega$-spectrum by Theorem~\ref{thm:Q}\ref{i:Omega}.
\end{proof}


\section{Assembly and coassembly maps}
\label{sec:asbl-coasbl}

In this section we prove the following consequence of Theorem~\ref{thm:Q}.

\begin{proposition}
\label{prop:asbl-coasbl}
Let $H$ be a subgroup of~$G$, $\CU$ a $G$-respresentation, $\spec{X}$ and~$\spec{Y}$ orthogonal $G$ and $H$-spectra, respectively, and $Z$ a pointed space.
Then there are natural $G$-maps
\begin{align*}
\MOR{\asbl}%
{&Z\sma\Q{\CU}(\spec{X})}%
{\Q{\CU}(Z\sma\spec{X})}
\,,
\\
\MOR{\coasbl}%
{&\Q{\CU}\bigl(\coind_{H\leq G}\spec{Y}\bigr)}%
{\coind_{H\leq G}\Q{\res_{H\leq G}\CU}(\spec{Y})}
\,,
\end{align*}
called \emph{assembly} and \emph{coassembly map}.
If $\spec{X}$ is good and $Z$ is a pointed CW-complex, then $\asbl$ is a \piiso.
If $\spec{Y}$ is good, then $\coasbl$ is a \piiso.
\end{proposition}

We first explain the definition of the assembly and coassembly maps.
It is based on the following categorical observation, that allows us to treat the examples we need in a unified way.

\begin{construction}
\label{constr:asbl}
Consider the following not necessarily commutative diagram of categories and functors
\begin{equation*}
\begin{tikzcd}[column sep=huge]
\CC
\arrow{r}{E}
\arrow[draw=none]{d}[description]{\dashv}
\arrow[xshift=-1ex]{d}[swap]{L}
&
\CC'
\arrow[draw=none]{d}[description]{\dashv}
\arrow[xshift=-1ex]{d}[swap]{L'}
\\
\CD
\arrow{r}[swap]{F}
\arrow[xshift=+1ex]{u}[swap]{R}
&
\CD'
\arrow[xshift=+1ex]{u}[swap]{R'}
\end{tikzcd}
\end{equation*}
where the vertical functors form two adjoint pairs.
Then there are inverse bijections
\begin{equation}
\label{eq:ASBL}
\begin{tikzcd}[column sep=large]
\nat(ER,R'F)
\arrow[yshift=+1ex]{r}{\ASBL}
&
\nat(L'E,FL)
\arrow[yshift=-1ex]{l}{\COASBL}
\end{tikzcd}
\end{equation}
between the indicated sets of natural transformations; pictorially:
\[
\begin{tikzcd}[row sep=small, column sep=small]
{}
\arrow[Rightarrow, shorten <=1ex, shorten >=1.5ex]{rd}
&
\phantom{\cdot}
&
{}
\arrow[draw=none]{d}[description]{\ds\leftrightarrow}
&
{}
\phantom{\cdot}
\arrow[to path=|- (\tikztotarget)]{dr}
\arrow[to path=-| (\tikztotarget)]{dr}
&
{}
\arrow[Rightarrow, shorten <=1ex, shorten >=1.5ex]{ld}
\\
\phantom{\cdot}
\arrow[to path=|- (\tikztotarget)]{ur}
\arrow[to path=-| (\tikztotarget)]{ur}
&
{}
&
{}
&
{}
&
\phantom{\cdot}
\end{tikzcd}
\]
In the examples that we are interested in, there is a distinguished natural transformation on one side of~\eqref{eq:ASBL} and we then consider the corresponding natural transformation on the other side, and we refer to it as the corresponding assembly or coassembly map.

In order to describe the bijection~\eqref{eq:ASBL}, let $c$ and $d$ be objects in $\CC$ and~$\CD$, respectively.
Given a natural transformation $\MOR{\nu_d}{ER(d)}{R'F(d)}$ we define $\MOR{\ASBL(\nu)_c}{L'E(c)}{FL(c)}$ as the adjoint of the composition
\[
E(c)\xrightarrow{E\mu_c} ERL(c) \xrightarrow{\nu_{L(c)}} R'FL(c)
\,,
\]
where the first map is induced by the unit~$\mu$ of the adjunction $L\dashv R$ and the second one is induced by~$\nu$.
Coassembly is defined dually using the counit of the adjunction.
These constructions also work for topological categories and continuous functors.
Here are two examples.
\end{construction}

\begin{example}[Assembly]
\label{eg:assembly}
Let $\Gamma$ and $G$ be groups and let $\MOR{F}{\Sp^\Gamma}{\Sp^G}$ be a continuous functor. Given an orthogonal $\Gamma$-spectrum $\spec{X}$ consider the following diagram.
\begin{equation*}
\begin{tikzcd}[column sep=8em]
\CT
\arrow{r}{\id}
\arrow[draw=none]{d}[description]{\dashv}
\arrow[xshift=-1ex]{d}[swap]{-\sma\spec{X}}
&
\CT
\arrow[draw=none]{d}[description]{\dashv}
\arrow[xshift=-1ex]{d}[swap]{-\sma F(\spec{X})}
\\
\Sp^\Gamma
\arrow{r}[swap]{F}
\arrow[xshift=+1ex]{u}[swap]{\Sp^\Gamma(\spec{X},-)}
&
\Sp^G
\arrow[xshift=+1ex]{u}[swap]{\Sp^\Gamma(F(\spec{X}),-)}
\end{tikzcd}
\end{equation*}
A continuous natural transformation
\(
\Sp^\Gamma(\spec{X},-)\TO\Sp^G(F(\spec{X}),F(-))
\)
is induced by the functor $F$ itself.
The corresponding natural transformation on the other side of~\eqref{eq:ASBL} evaluated at $Z$ yields
\(
\MOR{\asbl_Z}{Z \sma F(\spec{X})}{F(Z \sma\spec{X})}
\), a natural $G$-map of orthogonal $G$-spectra.
Pictorially:
\[
\MOR{F}{\Sp^\Gamma(\spec{X},-)}{\Sp^G(F(\spec{X}),F(-))}
\quad\leftrightarrow\quad
\MOR{\asbl}{-\sma F(\spec{X})}{F(-\sma\spec{X})}
\,.
\]
Notice that if $\MOR{\tau}{F}{F'}$ is a continuous natural transformation then the following diagram commutes for any pointed space~$Z$.
\begin{equation}
\label{eq:asbl-natural-in-F}
\begin{tikzcd}[column sep=large]
Z\sma F(\spec{X})\phantom{'}
\arrow{r}{\asbl}
\arrow{d}[swap]{\id_Z\sma\tau_\spec{X}}
&
F(Z\sma\spec{X})
\arrow{d}{\tau_{Z\sma\spec{X}}}
\\
Z\sma F'(\spec{X})
\arrow{r}[swap]{\asbl'}
&
F'(Z\sma\spec{X})
\end{tikzcd}
\end{equation}
\end{example}

\begin{example}[Coassembly]
\label{eg:coind-coassembly}
Let $H$ be a subgroup of~$G$ and let $\CU$ be a $G$\=/representation.
Abbreviate $\res=\res_{H\leq G}$ and $\coind=\coind_{H\leq G}$.
Consider the following diagram.
\begin{equation*}
\begin{tikzcd}[column sep=8em]
\Sp^G
\arrow{r}{\Q[h]{\CU}}
\arrow[draw=none]{d}[description]{\dashv}
\arrow[xshift=-1ex]{d}[swap]{\res}
&
\Sp^G
\arrow[draw=none]{d}[description]{\dashv}
\arrow[xshift=-1ex]{d}[swap]{\res}
\\
\Sp^H
\arrow{r}[swap]{\Q[h]{\res\CU}}
\arrow[xshift=+1ex]{u}[swap]{\coind}
&
\Sp^H
\arrow[xshift=+1ex]{u}[swap]{\coind}
\end{tikzcd}
\end{equation*}
By Theorem~\ref{thm:Q}\ref{i:res-Q} for every orthogonal $G$-spectrum~$\spec{X}$ there is a natural $H$\=/isomorphism
\(
\res\Q[h]{\CU}(\spec{X})\cong\Q[h]{\res\CU}(\res\spec{X})
\).
The corresponding natural transformation on the other side of~\eqref{eq:ASBL} evaluated at an orthogonal $H$-spectrum~$\spec{Y}$ yields
\[
\MOR{\coasbl}{\Q[h]{\CU}(\coind\spec{Y})}{\coind\Q[h]{\res\CU}(\spec{Y})}
\,.
\]
Pictorially:
\[
\MOR{\coasbl}{\Q[h]{\CU}(\coind(-))}{\coind\Q[h]{\res\CU}(-)}
\quad\leftrightarrow\quad
\res\Q[h]{\CU}(-)\cong\Q[h]{\res\CU}(\res(-))
\,.
\]
Similarly, if $W$ is a pointed $G$-space we can consider the functors $W\sma-$ and $(\res W)\sma-$ instead of $\Q{\CU}$ and~$\Q{\res\CU}$, respectively.
Then we get
\[
\MOR{\coasbl}{W\sma\coind(-)}{\coind(\res W\sma-)}
\quad\leftrightarrow\quad
\res(W\sma-)\cong\res W\sma\res(-)
\,,
\]
where the last isomorphism comes from Example~\ref{eg:res-smash}.
\end{example}

\begin{proof}[Proof of Proposition~\ref{prop:asbl-coasbl}]
The assembly map comes from Example~\ref{eg:assembly} applied to the case $\Gamma=G$ and $F=\Q[h]{\CU}(-)$.
Applying~\eqref{eq:asbl-natural-in-F} to the natural transformation~$\MOR{\fr}{\id}{\Q[h]{\CU}(-)}$ from Theorem~\ref{thm:Q} we get the following commutative triangle.
\begin{equation}
\label{eq:triangle-asbl}
\begin{tikzcd}[column sep=0pt]
&Z\sma\spec{X}
\arrow{ld}[swap]{\id\sma\fr_{\spec{X}}}
\arrow{rd}      {\fr_{Z\sma\spec{X}}}
\\
Z\sma\Q[h]{\CU}(\spec{X})
\arrow{rr}[swap]{\asbl}
&&
\Q[h]{\CU}(Z\sma\spec{X})
\end{tikzcd}
\end{equation}
If $\spec{X}$ is good then so is $Z\sma\spec{X}$ by Lemma~\ref{lem:good}\ref{i:sma-map-good}, and therefore the replacement maps~$\fr$ are \piiso s by Theorem~\ref{thm:Q}\ref{i:repl}.
From Lemma~\ref{lem:well-known}\ref{i:sma-pi-iso} it follows that $\asbl$ is a \piiso\ if $Z$ is a pointed CW-complex.

The coassembly map comes from Example~\ref{eg:coind-coassembly}.
Analogously to~\eqref{eq:triangle-asbl} the following triangle commutes.
\begin{equation*}
\label{eq:triangle-coasbl}
\begin{tikzcd}[column sep=0pt]
&\coind\spec{Y}
\arrow{ld}[swap]{\fr_{\coind\spec{Y}}}
\arrow{rd}      {\coind\fr_{\spec{Y}}}
\\
\Q{\CU}(\coind\spec{Y})
\arrow{rr}[swap]{\coasbl}
&&
\coind\Q{\res\CU}(\spec{Y})
\end{tikzcd}
\end{equation*}
Now assume that $\spec{Y}$ is good.
Then $\coind\fr_{\spec{Y}}$ is a \piiso\ by Theorem~\ref{thm:Q}\ref{i:repl} and Theorem~\ref{thm:H-in-G}\ref{i:res-co-ind-piiso}.
Moreover $\coind\spec{Y}$ is almost good by Theorem~\ref{thm:H-in-G}\ref{i:co-ind-good} and so $\fr_{\coind\spec{Y}}$ is a \piiso\ by Addendum~\ref{add:almost}.
Therefore we conclude that $\coasbl$ is a \piiso.
\end{proof}


\section{Replacements using incomplete universes}
\label{sec:free}

Let $N$ be a normal subgroup of~$G$.
Recall from the introduction that $\CF(N)$ is the family of subgroups $H$ of~$G$ such that $H\cap N=1$.
For any family $\CF$ of subgroups of~$G$, we denote by $E\CF$ a universal space for~$\CF$, i.e., a $G$-CW-complex $E\CF$ such that, for any subgroup $H$ of~$G$, $E\CF^H$ is empty if $H\not\in\CF$, and $E\CF^H$ is non-equivariantly contractible if $H\in\CF$.
Notice that $E\CF(G)=EG$ and that $\res_{N\leq G}E\CF(N) = EN$.

The following proposition says that after smashing with~$E\CF(N)$ the replacements with respect to a complete $G$-universe or with respect to a smaller complete $G/N$-universe are equivariantly homotopy equivalent.
Notice that equivariant homotopy equivalence is the strongest kind of equivalence we consider.
It passes immediately to orbits and fixpoints.
This is important in our proof of the Adams isomorphism.
The proposition plays a role analogous to the results summarized under the slogan \emph{``free spectra live in the trivial universe''} in the setup of~\cite{LMS}*{see in particular the last paragraph of page~65}.
We recall that all orthogonal spectra already live in the trivial universe; compare Theorem~\ref{thm:naive}.

\begin{proposition}
\label{prop:univchange}
Let $N$ be a normal subgroup of~$G$, let $\CU$ be a complete $G$-universe, and let $\spec{X}$ be an orthogonal $G$-spectrum.
The inclusion $\MOR{\iota}{\CU^N}{\CU}$ induces a $G$-map
\begin{equation}
\label{eq:id-iota}
\MOR{\id\sma\iota_*}%
{E\CF(N)_+\sma\Q[h]{\CU^N\!}(\spec{X})}%
{E\CF(N)_+\sma\Q[h]{\CU  }(\spec{X})}
\end{equation}
which is a natural $G$-homotopy equivalence, i.e., there is a $G$-map
\begin{equation*}
\MOR{\kappa}%
{E\CF(N)_+\sma\Q[h]{\CU  }(\spec{X})}%
{E\CF(N)_+\sma\Q[h]{\CU^N\!}(\spec{X})}
\,,
\end{equation*}
natural in $\spec{X}$, such that both compositions are $G$-homotopic as maps of orthogonal $G$-spectra to the respective identities.
\end{proposition}

The only choice involved in the construction of the homotopy inverse $\kappa$ is a map $\MOR{\cw}{E\CF(N)}{\CL(\CU,\CU^N)}$ which is a $G$-weak equivalence.
The existence of such a map is assured by the following proposition. 

\begin{proposition}
\label{prop:EFN}
Let $N$ be a normal subgroup of~$G$, and let $\CU$ be a complete $G$-universe.
Let $X$ be a $G$-CW-approximation of~$\CL(\CU,\CU^N)$, i.e., let $X$ be a $G$-CW-complex together with a $G$-weak equivalence $\MOR{\cw}{X}{\CL(\CU,\CU^N)}$.
Then $X$ is a model for~$E\CF(N)$.
\end{proposition}

This proposition is a consequence of the following two lemmas.

\begin{lemma}[\cite{LMS}*{Lemma~II.1.5 on page~60}]
\label{lem:LMS-II.1.5}
Let $U$ be a $G$-representation and $\CV$ be a $G$-universe.
Let $H$ be a subgroup of~$G$.
If $\CL(U,\CV)^H\neq\emptyset$ then $\CL(U,\CV)$ is $H$-equivariantly contractible, and hence $\CL(U,\CV)^H$ is contractible.
\end{lemma}

\begin{lemma}[\cite{LMS}*{Lemma~II.2.4(ii) on page~63}]
\label{lem:LMS-II.2.4(ii)}
Let $\CU$ be a complete $G$-universe.
Let $N$ be a normal subgroup of~$G$.
Then, for each subgroup $H$ of~$G$, $\CL(\CU,\CU^N)^H\neq\emptyset$ if and only if $H\in\CF(N)$.
\end{lemma}

\begin{remark}
In fact, by~\cite{Waner}*{Theorem~4.9(iii) on page~359} and~\cite{LMS}*{Remarks~VI.2.19 on page~322}, the space $\CL(\CU,\CU^N)$ has the $G$-homotopy type of a $G$-CW-complex; see also~\cite{LMS}*{Lemma~II.2.11 on page~66}.
Therefore the map~$\cw$ of Proposition~\ref{prop:EFN} is even a $G$-homotopy equivalence.
\end{remark}

\begin{proof}[Proof of Proposition~\ref{prop:univchange}]
We first explain how the homotopy inverse~$\kappa$ is defined.
Notice that the map~$\id\sma\iota_*$ in~\eqref{eq:id-iota} factors as
\begin{equation}
\label{eq:id-iota-factors}
\begin{tikzcd}
E\CF(N)_+\sma\Q[h]{\CU^N\!}(\spec{X})
\arrow{r}{\alpha\sma\id}
\arrow{dr}[below]{\id\sma\iota_*}
&
E\CF(N)_+\sma\CL(\CU^N,\CU)_+\sma\Q[h]{\CU^N\!}(\spec{X})
\arrow{d}{\id\sma\Phi_{\CU^N,\CU}}
\\
&
E\CF(N)_+\sma\Q[h]{\CU  }(\spec{X})
\mathrlap{\,,}
\end{tikzcd}
\end{equation}
where $\MOR{\alpha}{E\CF(N)_+}{E\CF(N)_+\sma\CL(\CU^N,\CU)}$ is the composition
\[
E\CF(N)_+
\cong
E\CF(N)_+\sma\{\iota\}_+
\xhookrightarrow{\quad}
E\CF(N)_+\sma\CL(\CU^N,\CU)_+
\]
and $\Phi_{\CU^N,\CU}$ is the map from Corollary~\ref{cor:fun-in-U}.

Now define a pointed $G$-map $\MOR{\beta}{E\CF(N)_+}{E\CF(N)_+\sma\CL(\CU,\CU^N)_+}$ as
\[
E\CF(N)_+
\xrightarrow{(\id,\cw)_+}
\bigl(E\CF(N)\times\CL(\CU,\CU^N)\bigr)_+
\cong
E\CF(N)_+\sma\CL(\CU,\CU^N)_+
\,,
\]
where $\MOR{\cw}{E\CF(N)}{\CL(\CU,\CU^N)}$ is the chosen $G$-CW-approximation from Proposition~\ref{prop:EFN}.
Then~$\kappa$ is defined as the composition
\begin{equation}
\label{eq:kappa}
\begin{tikzcd}
E\CF(N)_+\sma\Q[h]{\CU  }(\spec{X})
\arrow{r}{\beta\sma\id}
\arrow{dr}[below]{\kappa}
&
E\CF(N)_+\sma\CL(\CU,\CU^N)_+\sma\Q[h]{\CU}(\spec{X})
\arrow{d}{\id\sma\Phi_{\CU,\CU^N}}
\\
&
E\CF(N)_+\sma\Q[h]{\CU^N\!}(\spec{X})
\mathrlap{\,,}
\end{tikzcd}
\end{equation}
in complete analogy with the factorization of~$\id\sma\iota_*$ in~\eqref{eq:id-iota-factors}.

We now show that the composition $\kappa\circ(\id\sma\iota_*)$ is $G$-homotopic as a map of orthogonal $G$-spectra to the identity.
Consider the following diagram, in which we use the abbreviation $E=E\CF(N)$.
\[
\begin{tikzcd}[column sep=large]
E_+\sma\Q[h]{\CU^N\!}(\spec{X})
\arrow{d}[swap]{\alpha\sma\id}
\arrow[bend left=15]{rd}{\id\sma\iota_*}
\\
E_+\sma\CL(\CU^N,\CU)_+\sma\Q[h]{\CU^N\!}(\spec{X})
\arrow{d}[swap]{\beta\sma\id\sma\id}
\arrow{r}{\id\sma\Phi}
&
E_+\sma\Q[h]{\CU}(\spec{X})
\arrow{d}{\beta\sma\id}
\arrow[rounded corners,
       to path={    (\tikztostart.east)
                 -| +(1.5,-1.5) [at end]\tikztonodes
                 |- (\tikztotarget.east)}
      ]{dd}{\kappa}
\\
E_+\sma\CL(\CU,\CU^N)_+\sma\CL(\CU^N,\CU)_+\sma\Q[h]{\CU^N\!}(\spec{X})
\arrow{d}[swap]{\id\sma\circ\sma\id}
\arrow{r}{\id\sma\id\sma\Phi}
&
E_+\sma\CL(\CU,\CU^N)_+\sma\Q[h]{\CU}(\spec{X})
\arrow{d}{\id\sma\Phi}
\\
E_+\sma\CL(\CU^N,\CU^N)_+\sma\Q[h]{\CU^N\!}(\spec{X})
\arrow[dashed, xshift=1.5ex]{d}
\arrow{r}{\id\sma\Phi}
\arrow[draw=none]{ur}[description]{\framed{A}}
&
E_+\sma\Q[h]{\CU^N\!}(\spec{X})
\\
E_+\sma\{\id_{\CU^N}\}_+\sma\Q[h]{\CU^N\!}(\spec{X})
\arrow{u}{\id\sma\incl\sma\id}
\arrow[bend right=15]{ru}[swap, pos=.6]{\id\sma\text{canonical iso}}[description]{\cong}
\arrow[draw=none, pos=.44]{ur}[description]{\framed{B}}
\end{tikzcd}
\]
The two inner diagrams with solid arrows labeled \framed{A} and~\framed{B} commute because of Corollary~\eqref{cor:fun-in-U}; compare \eqref{eq:fun-in-U-assoc} and~\eqref{eq:fun-in-U-unit}.
The other three remaining inner diagrams commute by definition and direct inspection.

By Lemma~\ref{lem:LMS-II.1.5} the bottom left solid vertical map~$\id\sma\incl\sma\id$ has a $G$-homotopy inverse, indicated by the dashed arrow, and therefore the bottom triangle with the dashed arrow commutes up to $G$-homotopy.
Since the bottom diagonal map is an isomorphism, the composition of it and all the downward left vertical maps is then $G$-homotopic to a $G$-map of the form
\[
\MOR{\gamma\sma\id}%
{E\CF(N)_+\sma\Q[h]{\CU^N\!}(\spec{X})}%
{E\CF(N)_+\sma\Q[h]{\CU^N\!}(\spec{X})}
\,,
\]
and therefore it is $G$-homotopic to the identity since this is the case for every self $G$-map~$\gamma$ of~$E\CF(N)$.

The proof that the other composition $(\id\sma\iota_*)\circ\kappa$ is also $G$-homotopic to the identity is completely analogous: one considers the same diagram but with the roles of $\CU^N$ and~$\CU$, $\alpha$ and~$\beta$, and $\id\sma\iota_*$ and~$\kappa$ interchanged.
\end{proof}

\begin{remark}
\label{rem:kappa}
In the proofs of Proposition~\ref{prop:norm} and Key Lemma~\ref{lem:key} we use that the definition of~$\kappa$ in~\eqref{eq:kappa} implies immediately the following equality: 
\[
\MOR{\Phi_{\CU,\CU^N} \circ (\cw \sma \id) = \pr_2 \circ \kappa}%
{E\CF(N)_+\sma\Q[h]{\CU}(\spec{X})}{\Q[h]{\CU^N\!}(\spec{X})}
\,.
\]
\end{remark}


\section{Transfer maps}
\label{sec:transfer}

In this section we describe transfer maps in the category of equivariant orthogonal spectra.
The bifunctorial replacement construction from Section~\ref{sec:Q} is used in order to ensure uniqueness of transfer maps up to homotopy.
The geometric heart is the Pontryagin-Thom construction, and the corresponding results in the stable homotopy category are of course well-known; see for example
\cite{tD-transf-rep}*{Section~7.6 pages~188--193}
and
\cite{LMS}*{Section~II.5 on pages~84--88}.
On first reading the reader should probably focus on the case where $H$ is trivial and $\CU$ is a complete universe, which is all that is needed for the definition of the Adams map in Construction~\ref{constr:adams}.
The more general case is needed for Key Lemma~\ref{lem:key} in the proof of the Adams isomorphism; see Remark~\ref{rem:not-large}.

First some notation.
A $\Gamma$-$H$-space is a space $A$ equipped with a left $\Gamma$-action and a right $H$-action such that the two actions commute, i.e., $(\gamma a) h = \gamma (a h )$.
For a pointed $\Gamma$-$H$-space $A$ and a pointed left~$(\Gamma \times H)$-space $X$ we write \[ A \sma_H X \] for the left $\Gamma$-space obtained from $A \sma X$ by identifying $ah \sma x$ with $a \sma hx$, and equipped with the diagonal $\Gamma$-action.
We use the same notation if $X$ is a spectrum.

Given a $\Gamma$-$H$-map $\MOR{p}{A}{B}$ of finite $\Gamma$-$H$-sets and an orthogonal $(\Gamma \times H)$-spectrum $\spec{X}$ one obtains a $\Gamma$-map of orthogonal $\Gamma$-spectra
\[
\MOR{p \sma \id}{A_+ \sma_H \spec{X}}{B_+ \sma_H \spec{X}}
\]
which is clearly natural in $\spec{X}$.
In this section we construct and study a natural transformation
\[
\MOR{p^!}{B_+ \sma_H \spec{X}}{\Q[h]{\CU}\bigl(A_+ \sma_H \spec{X}\bigr)}
\]
in the opposite direction, that we call the \emph{transfer map} associated with $p$, $\spec{X}$ and~$\CU$. 
The following condition relating $p$ and $\CU$ will ensure the existence of~$p^!$ for all~$\spec{X}$.

\begin{definition}[Large enough universes]
\label{def:large}
Given a $\Gamma$-$H$-map $\MOR{p}{A}{B}$ between finite $\Gamma$-$H$-sets we say that a (not necessarily complete) $\Gamma$-universe $\CU$ is \emph{large enough for}~$p$ if there exists an injective $\Gamma$-$H$-equivariant map $\MOR{i}{A}{\CU \times B}$ such that
\[
\begin{tikzcd}
\hspace{0ex} &  \CU \times B \arrow{d}{\pr_2}\\
 A \arrow{ur}{i} \arrow[swap]{r}{p} & B 
\end{tikzcd}
\]
commutes.
Here we always equip $\CU$ with the trivial right $H$-action.
\end{definition}

Notice that $\CU$ being large enough for $p$ also imposes a condition on~$p$. 
If there exists a $\Gamma$-universe $\CU$ that is large enough for~$p$, then
\begin{equation}
\label{eq:large}
\text{for all $a\in A$ the restriction~$p_{|{aH}}$ to the right $H$-orbit of $a$ is injective.}
\end{equation}
A complete $\Gamma$-universe is large enough for all~$p$ satisfying~\eqref{eq:large}:
for $i$ use the map $A \TO A/H \times B$, $a \mapsto (aH,p(a))$ together with a $\Gamma$-equivariant embedding of $A/H$ into $\CU$.

\begin{theorem}
\label{thm:transfer}
Suppose that $\MOR{p}{A}{B}$ is a $\Gamma$-$H$-map of finite $\Gamma$-$H$-sets.
Let $\CU$ be a $\Gamma$-universe that is large enough for $p$.
\begin{enumerate}

\item
\label{i:transfer-existence}
\emph{(Existence)}
For every orthogonal $(\Gamma \times H)$-spectrum $\spec{X}$ Construction~\ref{constr:shriek} yields a $\Gamma$-map
\[
\MOR{p^!}{B_+ \sma_H \spec{X}}{\Q[h]{\CU}\bigl(A_+ \sma_H \spec{X}\bigr)} 
\]
which is natural in $\spec{X}$, i.e., $p^!$ is a natural transformation of functors from $\Sp^{\Gamma \times H}$ to $\Sp^{\Gamma}$.

\item
\label{i:transfer-uniqueness}
\emph{(Uniqueness)}
The natural transformation $p^!$ depends on the choice of a certain embedding~$u$, see \eqref{eq:embedding}, and hence $p^! = p^!_{u}$. 
But for different choices $u$ and~$u'$ there is a $\Gamma$-equivariant homotopy between $p^!_{u}$ and $p^!_{u'}$ that is natural in $\spec{X}$.


\item
\label{i:transfer-pullback}
\emph{(Pullbacks)}
Suppose 
\[
\begin{tikzcd}
\overline{A}
\arrow{r}{\overline{p}}
\arrow{d}[swap]{a}
& 
\overline{B}
\arrow{d}{b}
\\
A
\arrow{r}[swap]{p}
&
B
\end{tikzcd}
\]
is a pullback diagram of $\Gamma$-$H$-sets.
Then $\CU$ is also large enough for $\overline{p}$, and for suitable choices of $p^!$ and $\overline{p}^!$ the diagram
\[
\begin{tikzcd}
\ds \overline{B}_+ \sma_H \spec{X}
\arrow{r}{\overline{p}^!}
\arrow{d}[swap]{b \sma \id}
&
\ds\Q[h]{\CU}\bigl(\overline{A}_+ \sma_H \spec{X}\bigr)
\arrow{d}{\Q[h]{\CU}(a \sma \id)}
\\
\ds B_+ \sma_H \spec{X}
\arrow{r}[swap]{p^!}
&
\ds\Q[h]{\CU}\bigl(A_+ \sma_H \spec{X}\bigr)
\end{tikzcd}
\]
commutes.

\item
\label{i:transfer-ABC}
\emph{(Products)}
Suppose that $C$ is a third finite $\Gamma$-$H$-set and assume that the right $H$-action on $A$ and $B$ is trivial.
Then $\CU$ is also large enough for $p \times \id_C$, and we can choose the transfer $(p \times \id_C)^!$ associated with 
\[
\MOR{p \times \id_C}{A \times C}{B \times C}
\AND
\spec{X}
\]
such that under the obvious isomorphisms it coincides with the transfer $p^!$ associated with 
\[
\MOR{p}{A}{B}
\AND
C_+ \sma_H \spec{X}
\,.
\]

\item
\label{i:transfer-triangle}
\emph{(Change of universes)}
If $\MOR{f}{\CU}{\CV}$ is a $\Gamma$-equivariant isometry between $\Gamma$-universes, then for suitable choices of transfers the diagram
\[
\begin{tikzcd}[row sep=0ex, column sep=large]
&
\ds\Q[h]{\CU}\bigl(A_+ \sma_H \spec{X}\bigr)
\arrow{dd}{f_*}
\\
\ds B_+ \sma_H \spec{X}
\arrow{ur}{p^!}
\arrow{dr}[swap]{p^!}
\\
&
\ds\Q[h]{\CV}\bigl(A_+ \sma_H \spec{X}\bigr)
\end{tikzcd}
\]
commutes.
Here $f_*$ comes from Theorem~\ref{thm:Q}\ref{i:fun}, and if $\CU$ is complete then $f_*$ is a $\Gamma$-homotopy equivalence by Corollary~\ref{cor:Q-2-universes}.

\item
\label{i:transfer-res}
\emph{(Change of groups)}
If $\MOR{\alpha}{\Gamma'}{\Gamma}$ is a group homomorphism, then for a suitable choice of the transfer $(\res_\alpha p)^!$ associated with the map of $\Gamma'$-$H$-sets
\(
\MOR{\res_\alpha p}{\res_\alpha A}{\res_\alpha B}
\),
the orthogonal $(\Gamma'\times H)$-spectrum $\res_\alpha\spec{X}$, and the $\Gamma'$-universe~$\res_\alpha\CU$, the diagram
\[
\begin{tikzcd}[column sep=large]
\ds\res_\alpha B_+ \sma_H \res_\alpha \spec{X}
\arrow{r}{(\res_\alpha p)^!}
&
\ds\Q[h]{\res_\alpha\CU}\bigl(\res_\alpha A_+ \sma_H \res_\alpha \spec{X}\bigr)
\\
\ds\res_\alpha\bigl(B_+\sma_H\spec{X}\bigr)
\arrow{r}[swap]{\res_\alpha (p^!)}
\arrow{u}{\cong}
&
\ds\res_\alpha\Q[h]{\CU}\bigl(A_+\sma_H\spec{X}\bigr)
\arrow{u}[swap]{\cong}
\end{tikzcd}
\]
commutes.
Both vertical maps are isomorphisms. 
The map on the left is the isomorphism from Example~\ref{eg:res-smash}, and the map on the right is induced from that isomorphism composed with the natural isomorphism from Theorem~\ref{thm:Q}\ref{i:res-Q}.

\end{enumerate}
\end{theorem}

In the remainder of this section, we describe two important examples of Theorem~\ref{thm:transfer}, then we explain the definition of the transfer map~$p^!$ in Constructions~\ref{constr:pt} and~\ref{constr:shriek}, collect some facts about these constructions, and finally complete the proof of Theorem~\ref{thm:transfer}.
But first a remark about functoriality of transfers.

\begin{remark}[Functoriality in the stable homotopy category]
If $\spec{X}$ is good then in the stable homotopy category of $\Gamma$-spectra~$\Ho\Sp^\Gamma$ the replacement maps $\fr$ are invertible by Theorem~\ref{thm:Q}\ref{i:repl}.
If we define $p^{!!} = {\fr}^{-1} \circ p^!$, then using Facts \ref{fact:ptiii} and~\ref{fact:shriekii} it is easy to check that $(q\circ p)^{!!}=p^{!!} \circ q^{!!}$ and $\id^{!!}=\id$ in~$\Ho\Sp^\Gamma$.
\end{remark}

The following two special cases of Theorem~\ref{thm:transfer} are used later in the construction of the Adams map and in the proof of the Adams isomorphism.

\begin{example}
\label{eg:transfer-c}
If $H=1$, $\CU$ is a complete $\Gamma$-universe, $A$ is a finite $\Gamma$-set, and $\MOR{p=c}{A}{\pt}$ is the projection to a point, then we obtain for every orthogonal $\Gamma$-spectrum~$\spec{X}$ a $\Gamma$-map
\[
\MOR{c^!}{\spec{X}}{\Q[h]{\CU}(A_+\sma\spec{X})}
\,.
\]
If $\spec{X}$ is good, then so is~$A_+\sma\spec{X}$ by Lemma~\ref{lem:good}\ref{i:sma-map-good}, and after taking homotopy groups of the underlying non-equivariant spectra the transfer map fits into the following commutative diagram.
\[
\begin{tikzcd}[column sep=large]
\pi_*^1(\spec{X})
\arrow{r}{\pi_*^1(c^!)}
\arrow{d}[swap]{\diag}
&
\pi_*^1\bigl(\Q[h]{\CU}(A_+\sma\spec{X})\bigr)
\\
\ds\bigoplus_A\pi_*^1(\spec{X})
\arrow{r}[description]{\cong}
&
\pi_*^1(A_+\sma\spec{X})
\arrow{u}[description]{\cong}[right]{\ \pi_*^1(\fr)}
\end{tikzcd}
\]
Here the left-hand vertical map is the inclusion of the diagonal, the right-hand vertical isomorphism is induced by the replacement map from Theorem~\ref{thm:Q}, and the bottom isomorphism is induced by the (non-equivariant) inclusions of the wedge summands in $A_+ \sma \spec{X} \cong \bigvee_{a \in A} \spec{X}$.
The commutativity follows easily by inspecting the definition of $c^!$ in Construction~\ref{constr:shriek}, and noticing that, with the notation from Construction~\ref{constr:pt}, for every $a \in A$ the composition 
\[
S^U
\xrightarrow{\coll_u}
S^U\sma A_+
\cong
\bigvee_{a\in A} S^U
\xrightarrow{\ \pr_a\ }
S^U
\]
is non-equivariantly homotopic to the identity.
\end{example}

\begin{example}
\label{eg:transfer-GNH}
Suppose that $N$ is a normal subgroup of $\Gamma=G$, and denote by $\MOR{p}{G}{G/N}$ the projection to the quotient.
For a subgroup $H \leq G$ we consider $p$ as a map
\[
\MOR{p}{G}{p^*G/N}
\]
of $G$-$H$-sets with respect to the actions given by left and right multiplication.
If $H \cap N = 1$, then the restriction of $p$ to each right $H$-orbit is injective, and therefore for any complete $G$-universe $\CU$ and any orthogonal $H$-spectrum $\spec{X}$ we obtain a transfer map
\[
\MOR{p^!}{p^*G/N_+ \sma_H \spec{X}}{\Q[h]{\CU}(G_+ \sma_H \spec{X})}
\,.
\]
\end{example}

We now explain the construction of the transfer map associated with a map of finite $\Gamma$-$H$-sets $\MOR{p}{A}{B}$, an orthogonal $(\Gamma\times H)$-spectrum $\spec{X}$, and a $\Gamma$-universe~$\CU$.
Of course the main ingredient is the Pontryagin-Thom collapse construction.

\begin{construction}[Pontryagin-Thom collapse]
\label{constr:pt}
Let $U$ be a finite dimensional $\Gamma$-representation.
Equip it with the trivial right $H$-action.
Consider the trivial vector bundles with fiber~$U$ over~$A$ and over~$B$.
Suppose that $\MOR{u}{U \times A}{U \times B}$ is a $\Gamma$-$H$-equivariant embedding between the total spaces such that the diagram 
\[
\begin{tikzcd}
U \times A \arrow{r}{u} \arrow{d}[swap]{\pr_2} & U \times B \arrow{d}{\pr_2} \\
A \arrow{r}[swap]{p} & B
\end{tikzcd}
\]
commutes.
Then on the associated Thom spaces we get the usual collapse map
\[
\MOR{\coll_{u}}{S^U \sma B_+}{S^U \sma A_+}
\]
which sends every point in the image of $u$ in 
\[
U \times B \subseteq \Th (U \times B \TO B )\cong S^U \sma B_+
\]
to its preimage and all other points to the basepoint.
\end{construction}

We need the following facts about collapse maps.

\begin{fact}[Homotopy]
\label{fact:pti}
A $\Gamma$-$H$-equivariant homotopy $u_t$ of embeddings leads to an embedding $U \times A \times I \TO U \times B \times I$ and hence to a pointed $\Gamma$-$H$-equivariant homotopy between $\coll_{u_0}$ and $\coll_{u_1}$.
\end{fact}

\begin{fact}[Stabilization]
\label{fact:ptii}
Let $V$ be another \fd\ $\Gamma$-representation that contains $U$ and let $\incl$ denote the inclusion $U \subseteq V$.
If $\MOR{\incl_* u}{V \times A}{V \times B}$ is the embedding that via the obvious isomorphism $(V-U) \times U \cong V$ corresponds to $\id_{V-U} \times u$,
then $\incl_* u$ is again an embedding and the associated collapse map 
\[
\MOR{\coll_{\incl_*u}}{S^V \sma B_+ }{S^V \sma A_+}
\]
corresponds under the obvious isomorphism $S^{V-U} \sma S^U \cong S^V$ to 
\[
\MOR{\id_{S^{V-U}} \sma \coll_u}{S^{V-U}\sma S^U \sma B_+ }{S^{V-U}\sma S^U \sma A_+}
\,.
\]
\end{fact}

\begin{fact}[Composition]
\label{fact:ptiii}
If $q:B \to C$ is another $\Gamma$-$H$-map and $\MOR{v}{U \times B}{U \times C}$ is an equivariant embedding with $\pr_2 \circ v = q \circ \pr_2$, then 
$\coll_{v \circ u} = \coll_u \circ \coll_v$.
\end{fact}

\begin{construction}[Transfer]
\label{constr:shriek}
Suppose that the $\Gamma$-universe $\CU$ is large enough for $\MOR{p}{A}{B}$ in the sense of Definition~\ref{def:large}.
Since $A$ is finite, there exists a \fd\ $\Gamma$-subrepresentation $U \subset \CU$ and a $\Gamma$-$H$-map $\MOR{j}{A}{U}$ such that 
\[
\MOR{(j , p)}{A}{U\times B}
,\quad
a \mapsto (j(a),p(a))
\] 
is injective.
For suitable $\epsilon>0$ the $\Gamma$-$H$-map 
\begin{equation}
\label{eq:embedding}
\MOR{u=u_{j,\epsilon}}{U \times A}{U \times B}
,\quad
\bigl(x,a\bigr) \mapsto \bigl( j(a) + \epsilon\arctan(|x|)\, x/|x| ,\, p(a) \bigr)
\end{equation}
is then an embedding, and so Construction~\ref{constr:pt} produces a collapse map
\[
\MOR{\coll_{u}}{S^U \sma B_+}{S^U \sma A_+}
\,.
\]
The transfer map $p^!=p^!_u$ will depend on $u$ and hence on the choice of $U$, $j$ and $\epsilon$.

Now we apply the functor $ - \sma_H \spec{X}$ to $\coll_u$ and get a $\Gamma$-map
\begin{equation}
\label{eq:coll-sma-id}
S^U \sma \bigl( B_+ \sma_H \spec{X} \bigr)
\cong
(S^U \sma B_+) \sma_H \spec{X}
\xrightarrow{\coll_{u}\sma\id}
(S^U \sma A_+) \sma_H \spec{X}
\cong
S^U \sma \bigl( A_+ \sma_H \spec{X} \bigr)
\,,
\end{equation}
where for the isomorphisms in the source and target it is important that the right $H$-action on $S^U$ is trivial.
The desired transfer map $p^!_{u}$ is defined as the composition
\[
\begin{tikzcd}[row sep=scriptsize]
\ds B_+ \sma_H \spec{X}
\arrow{d}[swap]{\ds\one}{\ \widetilde{\coll_{u}\sma\id}}
\\
\ds \Omega^U\bigl(S^U \sma \bigl(A_+ \sma_H \spec{X}\bigr)\bigr)
\arrow{d}[swap]{\ds\two}
\\
\ds \Omega^U\shift^U\bigl( A_+ \sma_H \spec{X}\bigr)
=
\ds \Q{U} \bigl( A_+ \sma_H \spec{X} \bigr)
\arrow{d}[swap]{\ds\three}
\\
\ds \Q[h]{\CU}\bigl( A_+ \sma_H \spec{X}\bigr)
\end{tikzcd}
\]
where
$\one$ is the adjoint of the map in~\eqref{eq:coll-sma-id},
$\two$ is induced by the natural transformation $\sigma_U$ from $S^U \sma - $ to $\shift^U (-)$ in~\eqref{eq:sigma},
and $\three$ is the natural map induced by the inclusion~$U\subset\CU$.
It is clear that $p^!_u$ is natural in $\spec{X}$.
\end{construction}

The following facts follow from Facts~\ref{fact:pti} and~\ref{fact:ptii}.

\begin{fact}[Homotopy]
\label{fact:shrieki}
A $\Gamma$-$H$-equivariant homotopy $u_t$ of embeddings leads to a $\Gamma$-homotopy between $p^!_{u_0}$ and $p^!_{u_1}$ that is natural in $\spec{X}$.
\end{fact}

\begin{fact}[Stabilization]
\label{fact:shriekii}
If $U \subseteq V \subseteq \CU$ are $\Gamma$-subrepresentations, then associated with the embeddings $u$ and $\incl_* u$ from Fact~\ref{fact:ptii} we have the following diagram.
\[
\begin{tikzcd}[row sep=tiny, column sep=large]
& 
\ds\Q{U}\bigl(A_+ \sma_H \spec{X}\bigr) 
\arrow{dd}{\tau_{U \subseteq V}} 
\arrow{dr}{\three}
& \\
\ds B_+ \sma_H \spec{X} 
\arrow{ur}{\two \circ \widetilde{\coll_u \ssma \id}} 
\arrow{dr}[swap]{\two \circ \widetilde{\coll_{\incl_* u} \ssma \id}} 
& & 
\ds\Q[h]{\CU}\bigl(A_+ \sma_H \spec{X}\bigr)\\
& \ds\Q{V}\bigl(A_+ \sma_H \spec{X}\bigr) 
\arrow{ur}[swap]{\three}
&
\end{tikzcd}
\]
The triangle on the left commutes by inspection.
The triangle on the right commutes by functoriality of~$\Q{(-)}(A_+\sma_H\spec{X})$.
The two horizontal compositions are the transfer maps associated with $u$ and $\incl_* u$, respectively.
\end{fact}

\begin{proof}[Proof of Theorem~\ref{thm:transfer}]

\ref{i:transfer-existence}
The existence of $p^!$ and its naturality in $\spec{X}$ is discussed in Construction~\ref{constr:shriek}.

\ref{i:transfer-uniqueness}
We use Facts \ref{fact:shrieki} and \ref{fact:shriekii} repeatedly. 
Consider $u=u_{j,\epsilon}$ with $\MOR{j}{A}{U}$ as in~\eqref{eq:embedding}. 
Clearly shrinking $\epsilon$ yields a homotopy through embeddings.
So we assume $\epsilon$ sufficiently small for all constructions and fixed in the following. 
With the notation from Facts~\ref{fact:ptii} and \ref{fact:shriekii}, there exists a homotopy between $\incl_* u_{j,\epsilon}$ and $u_{\incl \circ j, \epsilon}$ and hence between the associated transfer maps. 
If $u'=u_{j',\epsilon}$ is a second embedding with $\MOR{j'}{A}{U'}$ we can hence always enlarge the representations and assume that $j$ and $j'$ both map to $V=U + U'$.
Now if $U \cap U' = \{ 0 \}$ the linear homotopy from $j$ to $j'$ induces a homotopy between the associated transfer maps. 
If not, we can arrange this: use that $\CU$ is a universe in order to enlarge the representation further so that it contains another copy of $V$ orthogonal to the original one, and then homotope one of the $j$'s into the new copy of $V$.


\ref{i:transfer-pullback}
Consider $u=u_{j,\epsilon}$ as in \eqref{eq:embedding}.
With $\overline{j} = j \circ a$ also $\overline{u}=u_{\overline{j},\epsilon}$ is an embedding, and the square
\[
\begin{tikzcd}
U \times \overline{A}
\arrow{d}[swap]{\id \times a}
\arrow{r}{\overline{u}}
& 
U \times \overline{B}
\arrow{d}{\id \times b}
\\
U \times A
\arrow{r}[swap]{u}
&
U \times B
\end{tikzcd}
\]
commutes.
Using that by assumption $\overline{A}$ is the pullback one checks that 
\[
(\id \times b)^{-1}( \im u ) = \im \overline{u}
\,.
\]
This is needed to verify that the associated square with the collapse maps (horizontal arrows reversed) also commutes. 
Applying the remaining steps in Construction~\ref{constr:shriek} clearly leads to the desired commutative diagram.

\ref{i:transfer-ABC}
To construct $(p \times \id_C)^!$ use $(a,c) \mapsto j(a)$ instead of $j$ and hence $u \times \id_C$ instead of $u$.
Since $\coll_{u \times \id_C}$ corresponds to $\coll_u \sma \id_{C_+}$, the claim follows.

\ref{i:transfer-triangle}
Define $f_*u$ such that  
\[
\begin{tikzcd}
U \times A
\arrow{d}[swap]{ f| \times \id}{\cong}
\arrow{r}{u}
& 
U \times B
\arrow{d}[swap]{\cong}{f| \times \id }
\\
f(U) \times A
\arrow{r}[swap]{f_* u}
&
f(U) \times B
\end{tikzcd}
\]
commutes.
This leads to the following commutative diagram.
\[
\begin{tikzcd}[row sep=tiny, column sep=large]
& 
\ds\Q{U}\bigl(A_+ \sma_H \spec{X}\bigr) 
\arrow{dd}{f|_*} 
\arrow{r}{\three}
& 
\ds\Q[h]{\CU}\bigl(A_+ \sma_H \spec{X}\bigr)
\arrow{dd}{f_*}
\\
\ds B_+ \sma_H \spec{X} 
\arrow{ur}{\two \circ \widetilde{\coll_u \ssma \id}} 
\arrow{dr}[swap]{\two \circ \widetilde{\coll_{f_* u} \ssma \id}} 
& & 
\\
& \ds\Q{f(U)}\bigl(A_+ \sma_H \spec{X}\bigr) 
\arrow{r}[swap]{\three}
&
\ds\Q[h]{\CV}\bigl(A_+ \sma_H \spec{X}\bigr)
\end{tikzcd}
\]

\ref{i:transfer-res}
To construct $(\res_\alpha p )^!$ start with $\res_\alpha u$ instead of $u$ and inspect the definitions.
\end{proof}


\section{Wirthm\"uller isomorphism and transfer}
\label{sec:wirth}

This section is devoted to the following theorem, which says that under the Wirthm\"uller isomorphisms from Theorem~\ref{thm:H-in-G}\ref{i:Wirth} the transfer map $p^!$ of Example~\ref{eg:transfer-GNH} corresponds to the usual functoriality of $\map(-,\spec{X})^H$.

\begin{theorem}
\label{thm:transfer-Wirth}
Let $N$ be a normal subgroup of $G$ and let $\MOR{p}{G}{G/N}$ denote the quotient map.
Let $H$ be a subgroup of~$G$ with $H \cap N = 1$, and consider $H$ also as a subgroup of~$G/N$.
Let $\CU$ be a complete $G$-universe.
Then the following diagram of orthogonal $G$-spectra commutes up to a $G$-equivariant homotopy that is natural in~$\spec{X}$.
\begin{equation}\label{eq:Wirth-shriek}
\begin{tikzcd}[column sep=large]
\ds p^*G/N_+\sma_H\spec{X}
\arrow{r}{p^!}
\arrow{d}[swap]{p^*W^{H\leq G/N}}
&
\ds\Q[h]{\CU}\bigl(G_+\sma_H\spec{X}\bigr)
\arrow{d}{\Q[h]{\CU}(W^{H\leq G})}
\\
\ds p^*\map(G/N_+,\spec{X})^H
\arrow{d}[swap]{\fr_*}
&
\ds\Q[h]{\CU}\bigl(\map(G_+,\spec{X})^H\bigr)
\arrow{d}{\coasbl}
\\
\ds p^*\map\bigl(G/N_+,\Q[h]{\res\CU}(\spec{X})\bigr)^H
\arrow{r}[swap]{p^*}
&
\ds\map\bigl(G_+,\Q[h]{\res\CU}(\spec{X})\bigr)^H
\end{tikzcd}
\end{equation}
\end{theorem}

The map $\fr_*$ is induced by the replacement map from Theorem~\ref{thm:Q}, the coassembly map $\coasbl$ comes from Proposition~\ref{prop:asbl-coasbl}, and the Wirthm\"uller maps~$W$ are defined explicitly in the proof of Theorem~\ref{thm:H-in-G}\ref{i:Wirth}.
For the definition of the transfer map~$p^!$, consider $U=\IR[G/H]$.
Since $H \cap N =1$, the $G$-$H$-equivariant map
\[
\MOR{u}{U \times G}{U \times G/N}
,\quad
\bigl(y,g\bigr) \mapsto \bigl(gH + \epsilon \arctan(|y|)\, y/|y| ,\, gN \bigr)
\]
is an embedding for sufficiently small $\epsilon >0$.
The associated collapse map from Construction~\ref{constr:pt} can then be used to define $p^!=p^!_u$ as in Construction~\ref{constr:shriek}.

The main ingredient in the proof of Theorem~\ref{thm:transfer-Wirth}
is the following lemma.

\begin{lemma}
\label{lem:Wirth}
For every pointed $H$-space $X$ the two ways through the following diagram are $G$-equivariantly homotopic.
The homotopy is natural in $X$.
\[
\begin{tikzcd}[column sep=large]
\ds S^U \sma p^*G/N_+\sma_H X
\arrow{r}{\coll_u \sma \id}
\arrow{d}[swap]{\id \sma W^{H \leq G/N}}
&
\ds S^U \sma G_+\sma_H X
\arrow{d}{\id \sma W^{H \leq G}}
\\
\ds S^U \sma \map(p^*G/N_+, X )^H
\arrow{r}{\id \sma p^*}
&
\ds S^U \sma \map(G_+, X )^H
\arrow{d}{\coasbl}
\\
&
\ds \map\bigl(G_+,\res_{H\leq G}S^U \sma X)^H
\end{tikzcd}
\]
\end{lemma}

\begin{proof}
It suffices to show that the adjoints of the two maps under the adjunction
between $\res=\res_{H \leq G}$ and $\map(G_+ , - )^H$ are $H$-homotopic. 
Hence we apply $\res$ to the diagram and postcompose with the counit of the adjunction, which is given by evaluation at the unit element and denoted $\ev_1$.
This leads to the left half of the following diagram.
\[
\hspace{-.25em}\adjustbox{scale=.94}{\begin{tikzcd}[column sep=small]
\ds \res \bigl( S^U \sma p^*G/N_+\sma_H X \bigr)
\arrow{r}[swap]{\ds\one}{\coll_u \ssma \id}
\arrow{d}[swap]{\ds\four}{\id \ssma W^{H \leq G/N}}
\arrow[rounded corners,
       to path={    (\tikztostart.north)
                 |- +(4.8,.4) [at end]\tikztonodes
                 -| (\tikztotarget.north)}
      ]{rr}{\coll_{u_t}\ssma \id}
&
\ds \res \bigl( S^U \sma G_+\sma_H X \bigr)
\arrow{d}{\id \ssma W^{H \leq G}}
\arrow{r}[swap]{\ds\two}{\id \ssma \pr_H \ssma \id}
&
\ds \res S^U \sma H_+\sma_H X 
\arrow{d}[swap]{\ds\three}{\id \ssma W^{H \leq H}}
\\
\ds \res \bigl( S^U \ssma \map(p^*G/N_+, X )^H \bigr)
\arrow{r}[swap]{\ds\five}{\id \ssma p^*}
&
\ds \res \bigl( S^U \sma \map(G_+, X )^H \bigr)
\arrow{d}{\coasbl}
\arrow{r}[swap]{\ds\six}{\id \ssma i_H^*}
&
\ds \res S^U \sma \map(H_+, X )^H 
\arrow{d}{\id \ssma \ev_1}
\\
&
\ds \res \map(G_+,\res S^{U} \sma X)^H
\arrow{r}{\ev_1}
&
\ds \res S^{U} \sma X
\end{tikzcd}}\hspace{-.26em}
\]
The map $\six$ is induced from the $H$-equivariant inclusion $\MOR{i_H}{H_+}{G_+}$, and the map $\two$ from the projection $\MOR{p_H}{G_+}{H_+}$ that leaves $H$ fixed and sends $G-H$ to the base point.
Both squares on the right commute.
Let $u_0$ and $u_1$ be the $H$-$H$-equivariant embeddings 
\(
\res U \times H \TO \res U \times \res p^*G/N 
\)
given by 
\[
u_0( y,h) = (eH + \epsilon \arctan(|y|)\, y/|y| ,\, hN)
\AND
u_1(y,h)=(y,hN)
\,.
\]
A computation shows that for the associated collapse maps we have 
\[
\two \circ \one = \coll_{u_0} \sma \id
\AND
\six \circ \five \circ \four = \three \circ (\coll_{u_1} \sma \id)
\,.
\] 
Clearly there exists an $H$-$H$-equivariant homotopy through embeddings between $u_0$ and $u_1$.
With Fact~\ref{fact:pti} the claim follows.
\end{proof}

The rest of the proof of Theorem~\ref{thm:transfer-Wirth} is now formal.
Since we managed twice to produce a proof of this using a non-existing arrow, we give the argument.

\begin{proof}[Proof of Theorem~\ref{thm:transfer-Wirth}]
Consider the diagram~\eqref{eq:huge-Wirth} on page~\pageref{eq:huge-Wirth}, whose boundary is~\eqref{eq:Wirth-shriek}.
The unlabeled parts commute by naturality or definition.
Diagram~\framed{A} commutes by inspection,
\framed{B} by definition of~$\fr$,
and \framed{C} by naturality of coassembly, compare Section~\ref{sec:asbl-coasbl}.
The statement of Lemma~\ref{lem:Wirth} remains true for orthogonal spectra $\spec{X}$ in place of $X$; compare Remark~\ref{rem:naive-not-naive}.
Therefore the two ways through diagram~\framed{L} composed with~$\one$ are naturally homotopic.
From these facts it follows that the two ways along the boundary of the diagram are naturally homotopic.
\end{proof}

\begin{landscape}
\begin{equation}
\label{eq:huge-Wirth}
\end{equation}
\[
\begin{tikzcd}[column sep=small, row sep=3.5ex]
\ds G/N_+\sma_H\spec{X}
\arrow{rr}
\arrow{dd}[swap]{W}
\arrow[rounded corners,
       to path={    (\tikztostart.north)
                 |- +(8.25,.9) [at end]\tikztonodes
                 -| (\tikztotarget.north)}
      ]{rrrrrr}{p^!}
&&
\ds\Omega^U\Sigma^U\bigl(G/N_+\sma_H\spec{X}\bigr)
\arrow{rr}
\arrow{dd}[swap]{\Omega^U\Sigma^U  W }
&&
\ds\Omega^U\Sigma^U\bigl(G_+\sma_H\spec{X}\bigr)
\arrow{rr}
\arrow{dd}{\Omega^U\Sigma^U  W }
&&
\ds\Q[h]{\CU}\bigl(G_+\sma_H\spec{X}\bigr)
\arrow{dd}{\Q[h]{\CU}(W)}
\\
&&&{\framed{L}}
\\
\ds\map(G/N_+,\spec{X})^H
\arrow{rr}
\arrow{dddddd}[swap]{\map(G/N_+,\fr)^H}
\arrow{ddddrr}[swap]{p^*}
&&
\ds\Omega^U\Sigma^U\bigl(\map(G/N_+,\spec{X})^H\bigr)
\arrow{rr}
&&
\ds\Omega^U\Sigma^U\bigl(\map(G_+,\spec{X})^H\bigr)
\arrow{rr}
\arrow{dd}{\ts\one}
&&
\ds\Q[h]{\CU}\bigl(\map(G_+,\spec{X})^H\bigr)
\arrow{dddddd}{\coasbl}
\\
{\phantom{\framed{P}}}
\\
&&
&&
\ds\Omega^U\bigl(\map(G_+,\Sigma^{\res U}\spec{X})^H\bigr)
\arrow{dd}
\\
&&&{\framed{A}}&&{\framed{C}}
\\
&&
\ds\map(G_+,\spec{X})^H
\arrow{uuuurr}
\arrow{rr}
\arrow{ddrrrr}[swap]{\map(G_+,r)^H}
&&
\ds\map(G_+,\Omega^{\res U}\Sigma^{\res U}\spec{X})^H
\arrow{ddrr}
\arrow[draw=none]{dd}[description, pos=.2]{\ds\framed{B}}
\\
{\phantom{\framed{P}}}
\\
\ds\map\bigl(G/N_+,\Q[h]{\res\CU}(\spec{X})\bigr)^H
\arrow{rrrrrr}[swap]{p^*}
&&
&&
{}
&&
\ds\map\bigl(G_+,\Q[h]{\res\CU}(\spec{X})\bigr)^H
\end{tikzcd}
\]
\end{landscape}


\section{Construction of the Adams map}
\label{sec:construction}

In this section we define the Adams map of Main Theorem~\ref{thm:main} and show that it satisfies the naturality condition~\ref{i:adams-natural} in that theorem.
Given a normal subgroup $N$ of~$G$, an orthogonal $G$-spectrum~$\spec{X}$, and a complete $G$-universe~$\CU$, we want to define an equivariant map of orthogonal $G/N$-spectra 
\begin{equation}
\label{eq:adams}
\MOR{\adams}{E\CF(N)_+\sma_N\spec{X}}{\Q[h]{\CU}(\spec{X})^N}
\end{equation}
and show that $A$ is natural in~$\spec{X}$, i.e., $\adams$ is a natural transformation of functors $\Sp^G\TO\Sp^{G/N}$.
The construction of~$\adams$ uses the semidirect product~$N\semi G$ and the following blueprint.

\begin{construction}[Blueprint]
\label{constr:blueprint}
The conjugation action of $G$ on $N$ gives rise to the semidirect product $N\semi G$, the group with underlying set $N\times G$ and multiplication given by 
\[
(n,g)(\overline{n},\overline{g}) = (n g \overline{n} g^{-1} , g \overline{g})
\,.
\]
Define group homomorphisms $\MOR{q,\mu}{N\semi G}{G}$ by $q(n,g)=g$ and $\mu(n,g)= ng$,
and notice that the following square is a pullback.
\begin{equation}
\label{eq:pullback}
\begin{tikzcd}
N\semi G
\arrow{d}[swap]{\mu}
\arrow{r}{q}
&
G
\arrow{d}{p}
\\
G
\arrow{r}[swap]{p}
&
G/N
\end{tikzcd}
\end{equation}
Define $\overline{N}$ to be the set $N$ equipped with the $N\semi G$ action given by $(n,g)\overline{n} = n g \overline{n} g^{-1}$.
Then, for every $G$-spectrum~$\spec{X}$, the action map $N_+\sma\spec{X}\TO\spec{X}$ yields an $(N\semi G)$-equivariant map
\begin{equation}
\label{eq:act}
\MOR{\act}{\overline{N}_+\sma q^*\spec{X}}{\mu^*\spec{X}}
.
\end{equation}
This is obviously true for spaces and so also for naive orthogonal spectra, and therefore it remains true for non-naive orthogonal spectra; compare Remark~\ref{rem:naive-not-naive}.

Observe that $q$~induces an isomorphism $(N\semi G)/(N\semi1)\cong G$, and therefore for any orthogonal $(N\semi G)$-spectrum~$\spec{Z}$ we can consider $\spec{Z}^{N\semi1}$ as an orthogonal $G$\=/spectrum.
There are natural isomorphisms of orthogonal $G/N$-spectra
\begin{equation}
\label{eq:blueprint}
\bigl((  q^*\spec{X})^{N\semi1}\bigr)/N\cong\spec{X}/N
\AND
\bigl((\mu^*\spec{X})^{N\semi1}\bigr)/N\cong\spec{X}^N
.
\end{equation}
So, in order to define a $G/N$-map of orthogonal $G/N$-spectra
\[
\MOR{f}{\spec{X}/N}{\spec{Y}^N}
,
\]
where $\spec{X}$ and~$\spec{Y}$ are orthogonal $G$-spectra, it is enough to construct an $(N\semi G)$-map
\[
\MOR{\widetilde{f}}{q^*\spec{X}}{\mu^*\spec{Y}}
\,,
\]
and then let
\[
\MOR{f = \bigl(\widetilde{f}^{N\semi1}\bigr)/N}%
{\spec{X}/N\cong\bigl((q^*\spec{X})^{N\semi1}\bigr)/N}%
{\bigl((\mu^*\spec{X})^{N\semi1}\bigr)/N\cong\spec{X}^N}
.
\]
Notice that then by definition the diagram
\begin{equation}
\label{eq:blueprint-square}
\begin{tikzcd}
\spec{X} \arrow[two heads]{d} \arrow{r}{\widetilde{f}}
&
\spec{Y}
\\
\spec{X}/N
\arrow{r}[swap]{f}
&
\spec{Y}^N
\arrow[hook]{u}
\end{tikzcd}
\end{equation}
commutes.
\end{construction}

Notice that in the special case $N=G$ Construction~\ref{constr:blueprint} simplifies; see the beginning of Section~\ref{sec:N=G}.

\begin{construction}[Adams map]
\label{constr:adams}
Using Construction~\ref{constr:blueprint}, in order to define the Adams map~$\adams$ in~\eqref{eq:adams} it is enough to construct an equivariant map of orthogonal $(N\semi G)$-spectra
\begin{equation}
\label{eq:preadams}
\MOR{\preadams}{q^*\bigl(E\CF(N)_+\sma\spec{X}\bigr)}{\mu^*\Q[h]{\CU}(\spec{X})}
\,,
\end{equation}
which we call the \emph{pre-Adams map}.
Then we define $\adams$ by taking the $(N\semi1)$-fixed points followed by the $N$-orbits of~$\preadams$, i.e.,
\[
\adams = \bigl(\preadams^{N\semi1}\bigr)/N
\,.
\]

Choose a complete $(N\semi G)$-universe~$\CV$.
Then define the pre-Adams map~$\preadams$ in~\eqref{eq:preadams} as the composition
\[
\begin{tikzcd}[row sep=scriptsize]
q^*\bigl(E\CF(N)_+\sma\spec{X}\bigr)
\arrow{d}[swap]{\ds\one}{\ \cong}
\\
q^*E\CF(N)_+\sma q^*\spec{X}
\arrow{d}[swap]{\ds\two}{\ \id\sma c^!}
\\
q^*E\CF(N)_+\sma\Q[h]{\CV}\bigl(\overline{N}_+ \sma q^*\spec{X}\bigr)
\arrow{d}[swap]{\ds\three}{\ \id\sma\Q[h]{\CV}(\act)}
\\
q^*E\CF(N)_+\sma\Q[h]{\CV}(\mu^*\spec{X})
\arrow{d}[swap]{\ds\four}{\ d\sma\id}
\\
\CL(\CV,\mu^*\CU)_+\sma\Q[h]{\CV}(\mu^*\spec{X})
\arrow{d}[swap]{\ds\five}{\ \Phi}
\\
\Q[h]{\mu^*\CU}(\mu^*\spec{X})
\arrow{d}[swap]{\ds\six}{\ \cong}
\\
\mu^*\Q[h]{\CU}(\spec{X})
\end{tikzcd}
\]
of the maps $\one$ through~$\six$ explained below.

\noindent$\one$
is the natural $(N\semi G)$-isomorphism from Example~\ref{eg:res-smash}.

\noindent$\two$
is induced by the transfer map
$\MOR{c^!}{q^*\spec{X}}{\Q[h]{\CV}\bigl(\overline{N}_+ \sma q^*\spec{X}\bigr)}$
from Theorem~\ref{thm:transfer} associated with the projection~$\MOR{c}{\overline{N}}{\pt}$, a map of~$(N\semi G)$-$1$-sets; see Example~\ref{eg:transfer-c}.

\noindent$\three$
is induced by applying the functor~$\Q[h]{\CV}(-)$ to the $(N\semi G)$-equivariant map~$\act$ in~\eqref{eq:act}.

\noindent$\four$
is induced by the map $\MOR{d}{q^*E\CF(N)}{\CL(\CV,\mu^*\CU)}$ defined as follows.
First of all, $q^*E\CF(N)=q^*E_G\CF(N)=E_{N\semi G}q^*\CF(N)$, where $H\in q^*\CF(N)$ if and only if $q(H)\in\CF(N)$.
By inspecting the definitions we see that $q^*\CF(N)\subseteq\CF(\ker\mu)$.
Therefore there is an $(N\semi G)$-map
\begin{equation}
\label{eq:j}
\MOR{j}
{q^*E_G\CF(N)}
{E_{N\semi G}\CF(\ker\mu)}
\,,
\end{equation}
unique up to $(N\semi G)$-homotopy.
By Proposition~\ref{prop:EFN} there is an $(N\semi G)$-weak equivalence
\[
\MOR{\cw}{E_{N\semi G}\CF(\ker\mu)}{\CL(\CV,\CV^{\ker\mu})}
\,.
\]
Since $\CV^{\ker\mu}$ is a complete $G$-universe, we can choose an $(N\semi G)$-isomorphism $\MOR[\cong]{k}{\CV^{\ker\mu}}{\mu^*\CU}$, and $k$ then induces an $(N\semi G)$-isomorphism
\[
\MOR[\cong]{k_*}{\CL(\CV,\CV^{\ker\mu})}
{\CL(\CV,\mu^*\CU)}
\,.
\]
The map $d$ is defined as the composition $k_* \circ \cw \circ j$ of these three maps.

\noindent$\five$
is the map $\Phi_{\CV,\mu^*\CU}$ expressing the functoriality of $\Q{(-)}(\mu^*\spec{X})$ from Corollary~\ref{cor:fun-in-U}.

\noindent$\six$
is the inverse of the $(N\semi G)$-isomorphism from Theorem~\ref{thm:Q}\ref{i:res-Q}.

This completes the definition of the pre-Adams map~$\preadams$ in~\eqref{eq:preadams} and hence of the Adams map~\eqref{eq:adams}.

The definition depends on the choice of~$\CV$, on the choice of a transfer map in step~$\two$, and on the choices of $j$, $\cw$, and~$k$ in step~$\four$.
Using the uniqueness of complete universes together with Theorem~\ref{thm:transfer}\ref{i:transfer-uniqueness} and~\ref{i:transfer-triangle}, the uniqueness up to homotopy of universal spaces and CW-approximations, and Lemma~\ref{lem:LMS-II.1.5}, we can easily see that different choices yield equivariantly homotopic maps.
\end{construction}

By inspecting the definition, we see that $\preadams$ and hence~$\adams$ are natural in~$\spec{X}$.
This proves part~\ref{i:adams-natural} of Main Theorem~\ref{thm:main}.

We conclude this section with the following result, which describes the effect of the pre-Adams map on the homotopy groups of the underlying non-equivariant spectra.

\begin{proposition}
\label{prop:norm}
For every good orthogonal $G$-spectrum~$\spec{X}$ the diagram
\[
\begin{tikzcd}[column sep=huge]
\pi_*^1\bigl(E\CF(N)_+\sma\spec{X}\bigr)
\arrow{d}[left]{\pi_*^1(\preadams)}
\arrow{r}[swap]{\cong}{\pi_*^1(\pr_2)}
&
\pi_*^1(\spec{X})
\arrow{d}[right]{\ell_{\sum n}}
\\
\pi_*^1\bigl(\Q[h]{\CU}(\spec{X})\bigr)
&
\pi_*^1(\spec{X})
\arrow{l}[swap]{\cong}{\pi_*^1(\fr)}
\end{tikzcd}
\]
commutes, where the right-hand vertical homomorphism is left multiplication with $\ds\sum_{n\in N} n \in \IZ[N]\leq\IZ[G]$, the norm element of~$N$.
\end{proposition}

\begin{proof}
Consider the following diagram.
\[
\begin{tikzcd}
\pi_*^1\bigl(E\CF(N)_+\sma\spec{X}\bigr)
\arrow{d}[left]{\pi_*^1(\id\sma c^!)}
\arrow{r}{\pi_*^1(\pr_2)}[below]{\cong}
&
\pi_*^1\bigl(\spec{X}\bigr)
\arrow{d}[left]{\pi_*^1(c^!)}
\arrow{r}{\diag}
&
\ds\bigoplus_{N}\pi_*^1(\spec{X})
\arrow[shorten <=-1ex]{d}[description, pos=.4]{\cong}
\\
\pi_*^1\bigl(E\CF(N)_+\sma\Q[h]{\CV}\bigl(\overline{N}_+ \sma q^*\spec{X}\bigr)\bigr)
\arrow{d}[left]{\pi_*^1(\id\sma Q(\act))}
\arrow{r}{\pi_*^1(\pr_2)}[below]{\cong}
&
\pi_*^1\bigl(\Q[h]{\CV}\bigl(\overline{N}_+ \sma q^*\spec{X}\bigr)\bigr)
\arrow{d}[left]{\pi_*^1(Q(\act))}
&
\pi_*^1\bigl(\overline{N}_+ \sma \spec{X}\bigr)
\arrow{d}[left]{\pi_*^1(\act)}
\arrow{l}[above]{\pi_*^1(\fr)}[below]{\cong}
\\
\pi_*^1\bigl(E\CF(N)_+\sma\Q[h]{\CV}(\mu^*\spec{X})\bigr)
\arrow{d}[left]{\pi_*^1(\Phi\circ(d\sma\id))}
\arrow{r}{\pi_*^1(\pr_2)}[below]{\cong}
&
\pi_*^1\bigl(\Q[h]{\CV}(\mu^*\spec{X})\bigr)
&
\pi_*^1(\spec{X})
\arrow[equal]{d}
\arrow{l}[above]{\pi_*^1(\fr)}[below]{\cong}
\\
\pi_*^1\bigl(\Q[h]{\CU}(\spec{X})\bigr)
\arrow[equal]{r}
&
\pi_*^1\bigl(\Q[h]{\CU}(\spec{X})\bigr)
\arrow{u}{\pi_*^1(\iota_* \circ k_*^{-1})}[swap]{\cong}
&
\pi_*^1(\spec{X})
\arrow{l}[above]{\pi_*^1(\fr)}[below]{\cong}
\end{tikzcd}
\]
The leftmost vertical composition is by definition~$\pi_*^1(\preadams)$, and the composition $\pi_*^1(\act)\circ\diag$ is left multiplication by $\sum_{n\in N}n$.
As explained in Example~\ref{eg:transfer-c}, the square in the top-right corner commutes.
The square in the bottom-left corner also commutes.
To see this, recall that $d=k_* \circ \cw \circ j$, and let $\iota$ denote the inclusion $\CV^{\ker\mu} \subseteq \CV$.
By naturality it suffices to show $\iota_* \circ \Phi \circ (\cw \sma \id) \simeq \pr_2$.
This holds because Proposition~\ref{prop:univchange} implies that $\iota_* \circ \pr_2 \circ \kappa \simeq \pr_2$, and by Remark~\ref{rem:kappa} we have that $\Phi \circ (\cw \sma \id)=\pr_2 \circ \kappa$.
All other diagrams commute by naturality.
\end{proof}


\section{Proof of the Adams isomorphism}
\label{sec:proof}

This section is devoted to the proof of part~\ref{i:adams-iso} of Main Theorem~\ref{thm:main}.
We start with an easy observation.

\begin{lemma}
\label{lem:free}
The following two statements are equivalent:
\begin{enumerate}
\item\label{i:X}
For all good and $N$-free orthogonal $G$-spectra~$\spec{X}$, $A_{\spec{X}}$ is a \piiso;
\item\label{i:EF-X}
For all good orthogonal $G$-spectra~$\spec{X}$, $A_{E\CF(N)_+\sma\spec{X}}$ is a \piiso.
\end{enumerate}
\end{lemma}

\begin{proof}
Let $\spec{X}$ be an orthogonal $G$-spectrum.
By Lemma~\ref{lem:good}\ref{i:sma-map-good}, if $\spec{X}$ is good then so is also $E\CF(N)_+\sma\spec{X}$.

Consider the projection on the second factor
\(
\MOR{\pr_2}{E\CF(N)_+\sma\spec{X}}{\spec{X}}
\).
Since
\begin{equation}
\label{eq:id-sma-pr}
\MOR{\id\sma\pr_2=\pr_1\sma\id}%
{E\CF(N)_+\sma E\CF(N)_+\sma\spec{X}}%
{E\CF(N)_+\sma\spec{X}}
\,,
\end{equation}
and $\MOR{\pr_1}{E\CF(N)_+\sma E\CF(N)_+}{E\CF(N)_+}$ is a $G$-homotopy equivalence, it follows that the map in~\eqref{eq:id-sma-pr} is a $G$-homotopy equivalence and so in particular a \piiso\ of orthogonal $G$-spectra.
Thus $E\CF(N)_+\sma\spec{X}$ is $N$-free, and \ref{i:X}~implies~\ref{i:EF-X}.

Conversely, consider the commutative square
\[
\begin{tikzcd}[column sep=huge]
\bigl(E\CF(N)_+\sma E\CF(N)_+\sma\spec{X}\bigr)/N
\arrow{r}{\adams_{E\CF(N)_+\sma\spec{X}}}
\arrow{d}[swap]{(\id\sma\pr_2)/N}
&
\Q[h]{\CU}(E\CF(N)_+\sma\spec{X})^N
\arrow{d}{\Q[h]{\CU}(\pr_2)^N}
\\
\bigl(E\CF(N)_+\sma\spec{X}\bigr)/N
\arrow{r}[swap]{\adams_{\spec{X}}}
&
\Q[h]{\CU}(\spec{X})^N
\end{tikzcd}
\]
induced by~$\pr_2$ via the naturality of the Adams map.
As observed above, $\id\sma\pr_2$ is a $G$-homotopy equivalence of orthogonal $G$-spectra, and therefore $(\id\sma\pr_2)/N$ is a $G/N$-homotopy equivalence and so a \piiso\ of orthogonal $G/N$-spectra.

Now assume that $\spec{X}$ is $N$-free, i.e., assume that $\pr_2$ is a \piiso.
If~$\spec{X}$ is good, then so is $E\CF(N)_+\sma\spec{X}$ by Lemma~\ref{lem:good}\ref{i:sma-map-good}, and hence by Theorem~\ref{thm:Q}\ref{i:repl} and~\ref{i:Omega} it follows that $\Q[h]{\CU}(\pr_2)$ is a \piiso\ of $G$-$\Omega$-spectra.
Then with Lemma~\ref{lem:Omega-fix}\ref{i:pi-iso-Omega-fix} we conclude that $\Q[h]{\CU}(\pr_2)^N$ is a \piiso.
\end{proof}

Given a good orthogonal $G$-spectrum $\spec{X}$, in order to show that $\adams_{E\CF(N)_+\sma\spec{X}}$ is a \piiso\ we proceed by induction on the cellular filtration of~$E\CF(N)$.
All the $G$-cells of~$E\CF(N)$ have the form~$D^n\times G/H$ with $H\leq G$ such that $H\cap N=1$.
The induction begin is established in the next key lemma.
This is by far the hardest step in the proof.
In the special case when $N=G$ the induction begin simplifies, as we explain in the next section; see Lemma~\ref{lem:easier}.

\begin{keylemma}
\label{lem:key}
Let $H\leq G$ be such that $H\cap N=1$.
Let $\spec{X}$ be a good orthogonal $G$-spectrum.
Then the Adams map
\[
\MOR{\adams_{G/H_+\sma\spec{X}}}%
{E\CF(N)_+\sma_N(G/H_+\sma\spec{X})}%
{\Q[h]{\CU}(G/H_+\sma\spec{X})^N}
\]
is a \piiso.
\end{keylemma}

\begin{proof}
We use the $G$-equivariant isomorphism
\begin{equation}
\label{eq:assumeinduced}
G/H_+\sma\spec{X}
\cong
G_+\sma_H\res_{H\leq G}\spec{X}
\end{equation}
from Example~\ref{eg:ind-res}.
By Theorem~\ref{thm:H-in-G}\ref{i:res-good}, $\res_{H\leq G}\spec{X}$ is good if and only if $\spec{X}$ is good.
Therefore we need to show that $\adams_{G_+\sma_H\spec{X}}$ is a \piiso\ for every good orthogonal $H$-spectrum~$\spec{X}$ of the form $\spec{X} = \res_{H\leq G}\spec{Y}$.

Consider the following diagram.
\begin{equation}
\label{eq:induction-begin}
\begin{tikzcd}[column sep=large]
\ds
q^*\bigl(E\CF(N)_+\sma G_+\sma_H\spec{X}\bigr)
\arrow{d}[swap]{\ds\one}
\\
\ds
q^*E\CF(N)_+\sma q^*G_+\sma_H\spec{X}
\arrow{r}{\pr_2}
\arrow{d}[swap]{\ds\two}{\id\sma c^!}
\arrow[draw=none, yshift=1ex]{dr}[description]{\framed{A}}
&
\ds
q^*G_+\sma_H\spec{X}
\arrow{d}{q^!}
\\
\ds
q^*E\CF(N)_+\sma\Q[h]{\CV}(\overline{N}_+\sma q^*G_+\sma_H\spec{X})
\arrow{r}{\Phi\circ(d\sma\id)}
\arrow{d}[swap]{\ds\three}{\id\sma\Q[h]{\CV}(\act)}
\arrow[draw=none, yshift=1ex]{dr}[description]{\framed{B}}
&
\ds
\Q[h]{\mu^*\CU}(\overline{N}_+\sma q^*G_+\sma_H\spec{X})
\arrow{d}{\Q[h]{\mu^*\CU}(\act)}
\\
\ds
q^*E\CF(N)_+\sma\Q[h]{\CV}(\mu^*G_+\sma_H\spec{X})
\arrow{r}[swap]{\ds\five\circ\four}{\Phi\circ(d\sma\id)}
&
\ds
\Q[h]{\mu^*\CU}(\mu^*G_+\sma_H\spec{X})
\arrow{d}[swap]{\ds\six}
\\
&
\ds
\mu^*\Q[h]{\CU}(G_+\sma_H\spec{X})
\end{tikzcd}
\end{equation}
The composition of the maps labeled $\one$ through~$\six$ is by definition the pre-Adams map~$\preadams_{G_+\sma_H\spec{X}}$ from Construction~\ref{constr:adams}.
The square labeled \framed{B} commutes by naturality.
We now explain the square~\framed{A} and show that it commutes up to $(N\semi G)$-homotopy.
This is a delicate step in the proof, based on the results from Sections \ref{sec:free} and~\ref{sec:transfer}.
The main point about~\framed{A} is that for induced spectra~$G_+\sma_H\spec{X}$ the Adams map can be described using the transfer~$q^!$ associated with the smaller universe~$\mu^*\CU$, as we proceed to explain.

Consider $\overline{N}$ as a $(N\semi G)$-$H$-set with trivial right $H$-action,
and $\spec{X}$ as an orthogonal $(N\semi G) \times H$-spectrum with trivial $(N\semi G)$-action.
Then by Theorem~\ref{thm:transfer}\ref{i:transfer-ABC} and~\ref{i:transfer-uniqueness} the transfer~$c^!$ associated with
\[
\MOR{c}{\overline{N}}{\pt}
\AND
q^*G_+\sma_H\spec{X}
\]
is $(N\semi G)$-homotopic to the transfer~$q^!$ associated with
\[
\MOR{q}{\overline{N}\times q^*G}{q^* G}
\AND
\spec{X}
\]
where $q$ is the projection onto the second factor.
In symbols,
\begin{equation}
\label{eq:c-is-q}
\MOR{c^!\simeq q^!}%
{q^*G_+\sma_H\spec{X}}%
{\Q[h]{\CV}(\overline{N}_+\sma q^*G_+\sma_H\spec{X})}
\,.
\end{equation}

Recall from step~$\four$ in Construction~\ref{constr:adams} that $k$ denotes the choice of an $(N\semi G)$-equivariant isometric isomorphism $\MOR[\cong]{k}{\CV^{\ker\mu}}{\mu^*\CU}$, and let $\iota$ denote the inclusion $\CV^{\ker \mu} \subseteq \CV$.

Notice that the universes $\mu^*\CU$ and~$\CV^{\ker \mu}$ are large enough for~$q$ in the sense of Definition~\ref{def:large}, because
\[
\overline{N}\times q^*G \TO \mu^* G/H \times q^* G, \quad (\overline{n},g) \mapsto (ngH, g)
\]
is $(N\semi G)$-$H$-equivariant and injective since $H \cap N = 1$ by assumption, and because the complete $G$-universe $\CU$ certainly contains $\IR[G/H]$.
Hence $q^!$ can also be defined using the universes $\mu^* \CU$ or $\CV^{\ker \mu}$.

\begin{remark}
\label{rem:not-large}
Notice that, unless~$N=1$, the transfer~$c^!$ cannot be defined with the universe~$\mu^*\CU$, because $\mu^*\CU$ is not large enough for~$c$ in the sense of Definition~\ref{def:large}.
\end{remark}

We proceed with the proof of Key Lemma~\ref{lem:key}.
For simplicity we introduce the abbreviations
\[
\spec{Y}=q^*G_+\sma_H\spec{X}
\AND
\spec{Z}=\overline{N}_+\sma q^*G_+\sma_H\spec{X}=\overline{N}_+\sma\spec{Y}
\,.
\]
By Theorem~\ref{thm:transfer}\ref{i:transfer-triangle} and~\ref{i:transfer-uniqueness} the following diagram commutes up to $(N\semi G)$-homotopy.
\begin{equation}
\label{eq:triangle}
\begin{tikzcd}
\ds \spec{Y}
\arrow{d}[left]{q^!}
\arrow{rd}[description]{q^!}
\arrow{rrd}{q^!}
\\
\ds\Q[h]{\CV}(\spec{Z})
&
\ds\Q[h]{\CV^{\ker \mu}}(\spec{Z})
\arrow{l}{\iota_*}
\arrow{r}{\cong}[swap]{k_*}
&
\ds\Q[h]{\mu^*\CU}(\spec{Z})
\end{tikzcd}
\end{equation}

\newcommand*{\oldE}{q^* E_G\CF(N)}
\newcommand*{\newE}{E_{N\semi G}\CF(\ker\mu)}

Consider the following diagram, whose top-left triangle is obtained by smashing the left half of \eqref{eq:triangle} with $q^*E_G\CF(N)_+$.
\begin{equation}
\label{eq:induction-begin-1}
\begin{tikzcd}
\oldE_+\sma\spec{Y}
\arrow{d}{\id\sma q^!}
\arrow{rr}{\pr_2}
\arrow{rd}{\id\sma q^!}
&
&
\spec{Y}
\arrow{d}{q^!}
\\
\oldE_+\sma\Q[h]{\CV}(\spec{Z})
\arrow{d}{j\sma\id}
\arrow[draw=none]{dr}[description]{\framed{C}}
&
\oldE_+\sma\Q[h]{\CV^{\ker \mu}}(\spec{Z})
\arrow{l}[swap]{\id\sma\iota_*}
\arrow{r}{k_* \circ \pr_2}
\arrow{d}{j\sma\id}
&
\Q[h]{\mu^*\CU}(\spec{Z})
\arrow[equal]{d}
\\
\newE_+\sma\Q[h]{\CV}(\spec{Z})
\arrow[yshift=-.7ex]{r}[swap]{\kappa}
&
\newE_+\sma\Q[h]{\CV^{\ker \mu}}(\spec{Z})
\arrow[dotted, yshift=+.7ex]{l}[swap]{\id\sma\iota_*}
\arrow{r}[swap]{k_* \circ \pr_2}
&
\Q[h]{\mu^*\CU}(\spec{Z})
\end{tikzcd}
\end{equation}
The upper half of this diagram commutes up to $(N\semi G)$-homotopy, because the same is true for~\eqref{eq:triangle}.
The lower-right square also commutes.
In \framed{C} the map $j$ was introduced in~\eqref{eq:j}, and the map $\kappa$ is the $(N\semi G)$-homotopy inverse of~$\id\sma\iota_*$ from Proposition~\ref{prop:univchange}.
Since \framed{C} obviously commutes when going from right to left, it follows that the entire diagram~\eqref{eq:induction-begin-1} commutes up to $(N\semi G)$-homotopy.

We claim that each of the two ways along the boundary of diagram~\eqref{eq:induction-begin-1} is homotopic to the corresponding way along the boundary of~\framed{A} in~\eqref{eq:induction-begin}, and so also \framed{A} commutes up to $(N\semi G)$-homotopy.
For the composition going through the upper-right corner there is nothing to prove.
For the one going through the lower-left corner we use that
$\id\sma q^! \simeq \id\sma c^!$, as observed in \eqref{eq:c-is-q},
and that $k_* \circ \pr_2 \circ \kappa \circ (j \sma \id) = \Phi \circ (d \sma \id)$;
the last equality follows from Remark~\ref{rem:kappa}, the definition of $d=k \circ \cw \circ j$ in step~$\four$ of Construction~\ref{constr:adams}, and the naturality of~$\Phi$.

Since we have now established that \eqref{eq:induction-begin} commutes up to $(N\semi G)$-homotopy, in order to finish the proof it is enough to show that the composition
\[
\six\circ\Q[h]{\mu^*\CU}(\act)\circ q^!\circ\pr_2\circ\,\one
\]
in diagram~\eqref{eq:induction-begin} induces a \piiso\ after taking $(N\semi1)$-fixed points and then $N$-orbits.

It is clear that
\[
\MOR{\bigl((\pr_2\circ\,\one)^{N\semi1}\bigr)/N=\pr_2/N}%
{\bigl(E\CF(N)_+\sma G_+\sma_H\spec{X}\bigr)/N}%
{\bigl(G_+\sma_H\spec{X}\bigr)/N}
\,;
\]
compare~\eqref{eq:blueprint}.
This map is a $G/N$-homotopy equivalence and hence a \piiso\ because of \eqref{eq:assumeinduced} and the fact that $\MOR{\pr_2}{E\CF(N)_+ \sma G/H_+}{G/H_+}$ is a $G$\=/equivariant homotopy equivalence.

It remains to analyze~$\six \circ \Q[h]{\mu^*\CU}(\act)\circ q^!$.
To this end we consider the pullback square of~$(N\semi G)$-$H$-sets
\[
\begin{tikzcd}
\overline{N}\times q^*G
\arrow{d}[swap]{\act}
\arrow{r}{q}
&
q^*G
\arrow{d}{q^* p}
\\
\mu^* G
\arrow{r}[swap]{\mu^* p}
&
\mu^* p^* G/N
=
  q^* p^* G/N
\end{tikzcd}
\]
coming from the pullback square of groups~\eqref{eq:pullback}.
Then by Theorem~\ref{thm:transfer}\ref{i:transfer-pullback} and~\ref{i:transfer-uniqueness} the diagram
\[
\begin{tikzcd}
\ds q^*G_+\sma_H\spec{X}
\arrow{r}{q^!}
\arrow{d}[swap]{q^*p \sma \id}
&
\ds\Q[h]{\mu^*\CU}\bigl(\overline{N}_+\sma q^*G_+\sma_H\spec{X}\bigr)
\arrow{d}{\Q[h]{\mu^*\CU}(\act)}
\\
\ds\mu^* p^* G/N_+\sma_H\spec{X}
\arrow{r}[swap]{(\mu^* p)^!}
&
\ds\Q[h]{\mu^*\CU}\bigl(\mu^*G_+\sma_H\spec{X}\bigr)
\end{tikzcd}
\]
commutes up to $(N\semi G)$-homotopy, and so
\(
\Q[h]{\mu^*\CU}(\act)\circ q^!\simeq(\mu^*p)^!\circ (q^*p \sma \id)
\).
Using~\eqref{eq:blueprint} we see that
\[
\MOR{\bigl((q^*p\sma \id)^{N\semi1}\bigr)/N=\id}{G/N_+\sma_H\spec{X}}{G/N_+\sma_H\spec{X}}
\,.
\]
Finally, by Theorem~\ref{thm:transfer}\ref{i:transfer-res} and~\ref{i:transfer-uniqueness}, the composition $\six\circ(\mu^*p)^!$ is $(N\semi G)$-homotopic to~$\mu^*(p^!)$, where
\[
\MOR{p^!}{p^* G/N_+\sma_H\spec{X}}{\Q[h]{\CU}(G_+\sma_H\spec{X})}
\]
is the transfer map associated with the projection $\MOR{p}{G}{p^*G/N}$, a map of $G$-$H$-sets; see Example~\ref{eg:transfer-GNH}.
By~\eqref{eq:blueprint}, the proof will be complete once we show that $(p^!)^N$ is a \piiso.

By Theorem~\ref{thm:transfer-Wirth} the following diagram commutes.
\begin{equation}
\label{eq:induction-begin-Wirth}
\begin{tikzcd}[column sep=large]
\ds p^*G/N_+\sma_H\spec{X}
\arrow{r}{p^!}
\arrow{d}[swap]{p^*W^{H\leq G/N}}
&
\ds\Q[h]{\CU}\bigl(G_+\sma_H\spec{X}\bigr)
\arrow{d}{\Q[h]{\CU}\bigl(W^{H\leq G}\bigr)}
\\
\ds p^* \map(G/N_+,\spec{X})^H
\arrow{d}[swap]{\fr_*}
&
\ds\Q[h]{\CU}\bigl(\map(G_+,\spec{X})^H\bigr)
\arrow{d}{\coasbl}
\\
\ds p^* \map\bigl(G/N_+,\Q[h]{\res\CU}(\spec{X})\bigr)^H
\arrow{r}[swap]{p^*}
&
\ds\map\bigl(G_+,\Q[h]{\res\CU}(\spec{X})\bigr)^H
\end{tikzcd}
\end{equation}
On the three terms on the left of this diagram $N$ acts trivially, and therefore taking $N$-fixed points yields
\[
G/N_+\sma_H\spec{X}
\xrightarrow{W^{H\leq G/N}}
\map(G/N_+,\spec{X})^H
\xrightarrow{\ \fr_*\ }
\map\bigl(G/N_+,\Q[h]{\res\CU}(\spec{X})\bigr)^H
.
\]
By the Wirthm\"uller isomorphism, Theorem~\ref{thm:H-in-G}\ref{i:Wirth}, $W^{H\leq G/N}$ is a \piiso.
Since $\spec{X}$ is good, Theorem~\ref{thm:Q}\ref{i:repl} implies that $\MOR{\fr}{\spec{X}}{\Q[h]{\res\CU}(\spec{X})}$ is a \piiso, and so $r_*$ above is also a \piiso\ by Theorem~\ref{thm:H-in-G}\ref{i:res-co-ind-piiso}.

Now consider the right-hand side of~\eqref{eq:induction-begin-Wirth}.
By Theorem~\ref{thm:H-in-G}\ref{i:Wirth} and~\ref{i:co-ind-good}, the Wirthm\"uller map~$W^{H\leq G}$ is a \piiso, $G_+\sma_H\spec{X}$ is good and $\map(G_+,\spec{X})^H$ is almost good.
Therefore Theorem~\ref{thm:Q}\ref{i:repl} and~\ref{i:Omega} and Addendum~\ref{add:almost} imply that $\Q[h]{\CU}(W^{H\leq G})$ is a \piiso\ of orthogonal $G$-$\Omega$-spectra, and therefore also $\Q[h]{\CU}(W^{H\leq G})^N$ is a \piiso\ by Lemma~\ref{lem:Omega-fix}\ref{i:pi-iso-Omega-fix}.
By Proposition~\ref{prop:asbl-coasbl}, $\coasbl$ is a \piiso.
Since we are assuming that $\spec{X}=\res_{H\leq G}\spec{Y}$, Theorem~\ref{thm:Q}\ref{i:res-Q} and Example~\ref{eg:ind-res} give isomorphisms
\[
\map\bigl(G_+,\Q[h]{\res\CU}(\spec{X})\bigr)^H
\cong
\map\bigl(G_+,\res\Q[h]{\CU}(\spec{Y})\bigr)^H
\cong
\map\bigl(H\backslash G_+,\Q[h]{\CU}(\spec{Y})\bigr)
,
\]
and the latter spectrum is a $G$-$\Omega$-spectrum by Lemma~\ref{lem:Omega}\ref{i:map-Omega} and Theorem~\ref{thm:Q}\ref{i:Omega}.
Then it follows from Lemma~\ref{lem:Omega-fix}\ref{i:pi-iso-Omega-fix} that also $\coasbl^N$ is a \piiso.

Finally, since
\[
\Bigl(\map\bigl(G_+,\Q[h]{\res\CU}(\spec{X})\bigr)^H\Bigr)^N
\cong
\map\bigl(G/N_+,\Q[h]{\res\CU}(\spec{X})\bigr)^H
,
\]
we see that after taking $N$-fixed points the bottom horizontal map in~\eqref{eq:induction-begin-Wirth} is just the identity of
$\map(G/N_+,\Q[h]{\res\CU}(\spec{X}))^H$.
Therefore also $(p^!)^N$ is a \piiso, and the result is proved.
\end{proof}

\begin{lemma}
\label{lem:adams-sigma}
For each good orthogonal $G$-spectrum~$\spec{X}$ and each integer~$n\geq0$ there is a commutative diagram
\begin{equation}
\label{eq:adams-sigma}
\begin{tikzcd}[column sep=large]
\ds S^n\sma\bigl(E\CF(N)_+\sma_N\spec{X}\bigr)
\arrow{r}{\id\sma\adams_{\spec{X}}}
\arrow{d}
&
\ds S^n\sma\Q[h]{\CU}(\spec{X})^N
\arrow{d}
\\
\ds E\CF(N)_+\sma_N(S^n\sma\spec{X})
\arrow{r}[swap]{\adams_{S^n\sma\spec{X}}}
&
\ds \Q[h]{\CU}(S^n\sma\spec{X})^N
\end{tikzcd}
\end{equation}
whose vertical maps are \piiso s.
Therefore $\adams_{\spec{X}}$ is a \piiso\ if and only if $\adams_{S^n\sma\spec{X}}$ is one.
\end{lemma}

\begin{proof}
Since the pre-Adams map is a natural transformation, by \eqref{eq:asbl-natural-in-F} there is a commutative diagram
\[
\begin{tikzcd}[column sep=large]
\ds S^n\sma q^*\bigl(E\CF(N)_+\sma\spec{X}\bigr)
\arrow{r}{\id\sma\preadams_{\spec{X}}}
\arrow{d}[left]{\asbl'}
&
\ds S^n\sma\mu^*\Q[h]{\CU}(\spec{X})
\arrow{d}[right]{\asbl}
\\
\ds q^*\bigl(E\CF(N)_+\sma(S^n\sma\spec{X})\bigr)
\arrow{r}[swap]{\preadams_{S^n\sma\spec{X}}}
&
\ds \mu^*\Q[h]{\CU}(S^n\sma\spec{X})
\end{tikzcd}
\]
where the vertical maps are the assembly maps of Example~\ref{eg:assembly}.

By inspection, the left-hand assembly map~$\asbl'$ is the evident $(N\semi G)$\=/isomorphism.
Since all groups act trivially on~$S^n$, after taking $(N\semi1)$-fixed points and then $N$-orbits the map~$\asbl'$ induces the left-hand vertical map in~\eqref{eq:adams-sigma}, which is thus a $G/N$-isomorphism.

Similarly, after fixed-points and orbits the map~$\asbl$ induces the right-hand vertical map in~\eqref{eq:adams-sigma}.
Since $\spec{X}$ is good, $\asbl$ is a \piiso\ by Proposition~\ref{prop:asbl-coasbl}.
Using Lemma~\ref{lem:Omega-fix} and the fact that $\pi_*^K(S^n\sma\spec{Y})\cong\pi_{*-n}^K(\spec{Y})$, it follows that also the right-hand vertical map in~\eqref{eq:adams-sigma} is a \piiso.
\end{proof}

By combining the previous lemma with Key Lemma~\ref{lem:key} we obtain the following result.

\begin{corollary}
\label{cor:begin}
Let $H\leq G$ be such that $H\cap N=1$.
Let $\spec{X}$ be a good orthogonal $G$-spectrum.
Then for every $n\geq0$ the Adams map
\[
\MOR{\adams_{S^n\sma G/H_+\sma\spec{X}}}%
{E\CF(N)_+\sma_N(S^n\sma G/H_+\sma\spec{X})}%
{\Q[h]{\CU}(S^n\sma G/H_+\sma\spec{X})^N}
\]
is a \piiso.
\end{corollary}

To finish the proof that $\adams_{E\CF(N)_+\sma\spec{X}}$ is a \piiso\ by induction on the cellular filtration of~$E\CF(N)$, we consider the Adams map
\[
\MOR{\adams_{Z\sma\spec{X}}}%
{E\CF(N)_+\sma_N(Z\sma\spec{X})}%
{\Q[h]{\CU}(Z\sma\spec{X})^N}
\]
for any pointed $G$-space~$Z$.

\begin{lemma}
\label{lem:cofiber}
For every subgroup $K$ of~$G$ containing~$N$ and every good orthogonal $G$-spectrum~$\spec{X}$, the functors
\begin{align*}
Z\mapsto\ &\pi^{K/N}_*\bigl(E\CF(N)_+\sma_N(Z\sma\spec{X})\bigr)\\
\shortintertext{and}
Z\mapsto\ &\pi^{K/N}_*\bigl(\Q[h]{\CU}(Z\sma\spec{X})^N   \bigr)
\end{align*}
send equivariant cofiber sequences of pointed $G$-spaces to long exact sequences of abelian groups.
\end{lemma}

\begin{proof}
By Lemma~\ref{lem:good}\ref{i:sma-map-good}, Lemma~\ref{lem:Omega-fix}\ref{i:Omega-fix}, and Theorem~\ref{thm:Q}\ref{i:repl}, we have that
\(
\pi^{K/N}_*(\Q[h]{\CU}(Z\sma\spec{X})^N)
\cong
\pi^{K}  _*           (Z\sma\spec{X})
\).
The lemma below implies that the functor $Z\mapsto Z\sma\spec{X}$ sends equivariant cofiber sequences of pointed $G$-spaces to levelwise cofiber sequences of orthogonal $G$-spectra, and that $Z\mapsto(E\CF(N)_+\sma Z\sma\spec{X})/N$ sends equivariant cofiber sequences of pointed $G$-spaces to levelwise cofiber sequences of orthogonal $G/N$-spectra.
Since equivariant homotopy groups send levelwise cofiber sequences of orthogonal spectra to long exact sequences of abelian groups \cite{MM}*{Theorem III.3.5.(vi) on page~46}, the statement follows.
\end{proof}

\begin{lemma}
\begin{enumerate}
\item 
For each pointed $G$-space $X$ the functor $Z\mapsto Z\sma X$ preserves equivariant cofiber sequences of pointed $G$-spaces.
\item
The $N$-orbits functor $Z\mapsto Z/N$ sends equivariant cofiber sequences of pointed $G$-spaces to equivariant cofiber sequences of pointed $G/N$-spaces.
\end{enumerate}
\end{lemma}

\begin{proof}
Both functors preserve all colimits since they are left adjoints, so in particular they preserve quotients.
Hence we only need to show that they preserve cofibrations.
For the first functor, recall that a map $\MOR{i}{Z}{Z'}$ is a cofibration if and only if the natural map ${M(i)}\TO{Z'\sma I_+}$ from the pointed mapping cylinder of~$i$ has a retraction~$r$;
see e.g.\ \cite{May-concise}*{last line of page~56} and \cite{tD-top}*{Proposition~5.1.2 on page~103}.
Since $M(i\sma\id)\cong M(i)\sma X$, $r\sma\id_X$ is the desired retraction.
For the second functor just inspect the definition directly.
\end{proof}

The next lemma says that both the source and the target of the Adams map are compatible with arbitrary coproducts.

\begin{lemma}
\label{lem:coproducts}
Let $\{\spec{X}_j\}_{j\in\CJ}$ be an arbitrary family of good orthogonal $G$-spectra.
Then the natural map
\begin{align*}
\bigvee_{j\in\CJ}\Bigl(E\CF(N)_+\sma_N\spec{X_j}\Bigr)
&\TO
E\CF(N)_+\sma_N\Bigl({\ts\bigvee_{j\in\CJ}\spec{X_j}}\Bigr)
\\
\shortintertext{is an isomorphism, and the natural map}
\bigvee_{j\in\CJ}\Q[h]{\CU}(\spec{X_j})^N
&\TO
\Q[h]{\CU}\Bigl({\ts\bigvee_{j\in\CJ}\spec{X_j}}\Bigr)^N
\end{align*}
is a \piiso\ of orthogonal $G/N$-spectra.
\end{lemma}

\begin{proof}
The first statement follows from the fact that orbits commute with coproducts, and smash products distribute over coproducts.
The second statement follows by combining Lemma~\ref{lem:Omega-fix}\ref{i:Omega-fix} and Theorem~\ref{thm:Q} with the fact that homotopy groups commute with coproducts \cite{MM}*{Theorem III.3.5.(ii) on page~46}.
\end{proof}

Using Corollary~\ref{cor:begin}, Lemma~\ref{lem:cofiber}, and Lemma~\ref{lem:coproducts}, we conclude by cellular induction that the Adams map~$A_{E\CF(N)_+\sma\spec{X}}$ is a \piiso\ for every good orthogonal $G$\=/spectrum~$\spec{X}$.
By Lemma~\ref{lem:free} this completes the proof of Main Theorem~\ref{thm:main}\ref{i:adams-iso}.


\section{An easier proof in the special case~\texorpdfstring{$N=G$}{N=G}}
\label{sec:N=G}

In this section we explain some simplifications that occur in the proof of the Adams isomorphism in the special case when $N=G$.
Notice that then $\CF(G)$ reduces to the trivial family, and therefore $E\CF(G)=EG$.
The Adams map becomes
\begin{equation*}
\MOR{\adams}{EG_+\sma_G\spec{X}}{\Q[h]{\CU}(\spec{X})^G}
.
\end{equation*}
In the construction of the Adams map in Section~\ref{sec:construction} there is no need to consider the semidirect product $G\semi G$.
In the blueprint in Construction~\ref{constr:blueprint} we can instead work with direct product $G \times G$.
In fact, the map $\MOR{\psi}{G\times G}{G\semi G}$ defined by $\psi(h,g)=(hg^{-1},g)$ is a group isomorphism, and the diagram
\[
\begin{tikzcd}[row sep=scriptsize]
{}
&
G
\\
G\times G
\arrow{rr}{\psi}[swap]{\cong}
\arrow{ur}{\pr_1}
\arrow{dr}[swap]{\pr_2}
&&
G\semi G
\arrow{ul}[swap]{\mu}
\arrow{dl}{q}
\\
&
G
\end{tikzcd}
\]
commutes.
Under this isomorphism $\overline{G}$ is the set $G$ equipped with the $G\times G$ action given by $(h,g)\overline{g} = h \overline{g} g^{-1}$,
and the pullback diagram \eqref{eq:pullback} is replaced by 
\begin{equation*}
\label{eq:pullback-simplified}
\begin{tikzcd}
G \times G
\arrow{d}[swap]{\pr_1}
\arrow{r}{\pr_2}
&
G
\arrow{d}
\\
G
\arrow{r}
&
1
\mathrlap{.}
\end{tikzcd}
\end{equation*}

So the pre-Adams map from Construction~\ref{constr:adams} can be thought of as a $(G\times G)$-map of orthogonal $(G\times G)$-spectra
\[
\MOR{\preadams}{\pr_2^*\bigl(EG_+\sma\spec{X}\bigr)}{\pr_1^*\Q[h]{\CU}(\spec{X})}
.
\]

The structure of the inductive proof of the Adams isomorphism in Section~\ref{sec:proof} remains unchanged.
However, the induction begin becomes much easier when~$N=G$.
Key Lemma~\ref{lem:key} becomes Lemma~\ref{lem:easier} below, whose proof is inspired by~\cite{Dundas}*{proof of Proposition~6.3.2.5 on pages~251--252}.

\begin{lemma}
\label{lem:easier}
Let $\spec{X}$ be a good orthogonal $G$-spectrum.
Then the Adams map
\[
\MOR{\adams_{G_+\sma\spec{X}}}%
{EG_+\sma_G(G_+\sma\spec{X})}%
{\Q[h]{\CU}(G_+\sma\spec{X})^G}
\]
is a \piiso\ of orthogonal spectra.
\end{lemma}

\begin{proof}
Like in the beginning of the proof of Key Lemma~\ref{lem:key}, we use the $G$-equivariant isomorphism
\[
G_+\sma\spec{X}
\cong
G_+\sma\res_{1\leq G}\spec{X}
\]
from Example~\ref{eg:ind-res}, and we can assume that $\spec{X}$ is an orthogonal $1$-spectrum of the form $\spec{X}=\res_{1\leq G}\spec{Y}$.

But now we consider the following diagram.
\begin{equation}
\label{eq:diagram-special}
\hspace{-.5em}
\begin{tikzcd}[column sep=.35em]
\ds \IZ[G]\tensor_\IZ\pi_*^1(\spec{X})
\arrow{rrrr}{\ell_{\sum g}\tensor\id}
\arrow{dddd}[swap]{\varepsilon\tensor\id}
&
&
&
&
\ds \IZ[G]\tensor_\IZ\pi_*^1(\spec{X})
\\
&
\pi_*^1\bigl(G_+\sma\spec{X}\bigr)
\arrow{rr}{\ell_{\sum g}}
\arrow{ul}[description]{\cong}
&
\,
\arrow[draw=none]{d}[description, pos=.4]{\framed{A}}
&
\pi_*^1\bigl(G_+\sma\spec{X}\bigr)
\arrow{d}[description]{\cong}{\ \fr_*}
\arrow{ur}[description]{\cong}
\\
&
\pi_*^1\bigl(EG_+\sma G_+\sma\spec{X}\bigr)
\arrow{rr}{\preadams_*}
\arrow{u}[description]{\cong}{{\pr_2}_*\ }
\arrow{d}
&
\,
\arrow[draw=none]{d}[description]{\framed{B}}
&
\pi_*^1\bigl(\Q[h]{\CU}(G_+\sma\spec{X})\bigr)^{\phantom{G}}
\arrow[draw=none]{r}[description, pos=.3]{\framed{D}}
&
\,
\\
&
\ds \pi_*^1\bigl((EG_+\sma G_+\sma\spec{X})/G\bigr)
\arrow{rr}[swap]{\adams_*}
\arrow{dl}[description]{\cong}
&
\,
\arrow[draw=none]{d}[description]{\framed{C}}
&
\pi_*^1\bigl(\Q[h]{\CU}(G_+\sma\spec{X})^G\bigr)
\arrow{u}
\arrow{dr}[description]{\cong}
\\
\ds \phantom{[G]}\IZ\tensor_\IZ\pi_*^1(\spec{X})
\arrow[shorten >=-1em]{rrrr}[pos=.56, below]{\id}
&
&
\,
&
&
\ds \phantom{[G]}\IZ\tensor_\IZ\pi_*^1(\spec{X})
\arrow{uuuu}[swap]{\sum g\tensor\id}
\end{tikzcd}
\end{equation}
The outer rectangle and the two unlabeled inner diagrams evidently commute.
The diagram labeled~\framed{A} commutes by Proposition~\ref{prop:norm}, and \framed{B} commutes by construction, see~\eqref{eq:blueprint-square}.
So if \framed{D} commutes, then also \framed{C} does, because $\epsilon$ is surjective, and therefore $A$ is a \piiso.

In order to explain the maps in \framed{D} and show that \framed{D} commutes, consider the following diagram, where the map $\one$ is the evident isomorphism and the three other vertical maps are given by inclusion of fixpoints.
\[
\begin{tikzcd}
\res \Q[h]{\CU}(G_+\sma\spec{X})^{\phantom{G}}
\arrow{r}{\!\!\Q[h]{\CU}(W)}
&
\res \Q[h]{\CU}(\map(G_+,\spec{X}))^{\phantom{G}}
\arrow{r}{\!\!\coasbl}
&
\res \map(G_+,\Q[h]{\res \CU}(\spec{X}))^{\phantom{G}}
\\
\Q[h]{\CU}(G_+\sma\spec{X})^G
\arrow{r}{\Q[h]{\CU}(W)^G}
\arrow[hook]{u}
&
\Q[h]{\CU}(\map(G_+,\spec{X}))^G
\arrow{r}{\coasbl^G}
\arrow[hook]{u}
&
\map(G_+,\Q[h]{\res \CU}(\spec{X}))^G
\arrow[hook]{u}[swap]{\ts\two}
\\
&
&
\Q[h]{\res \CU}(\spec{X})
\arrow{u}[swap]{\ts\one}{\cong}
\end{tikzcd}
\]
The Wirthm{\"u}ller map $W=W^{1 \leq G}$ is a \piiso\ by Theorem~\ref{thm:H-in-G}\ref{i:Wirth}.
By Theorem~\ref{thm:H-in-G}\ref{i:co-ind-good}, $G_+\sma\spec{X}$ is good and $\map(G_+,\spec{X})$ is almost good.
Therefore Theorem~\ref{thm:Q}\ref{i:repl} and~\ref{i:Omega} and Addendum~\ref{add:almost} imply that $\Q[h]{\CU}(W)$ is a \piiso\ of orthogonal $G$-$\Omega$-spectra, and so also $\Q[h]{\CU}(W)^G$ is a \piiso\ by Lemma~\ref{lem:Omega-fix}\ref{i:pi-iso-Omega-fix}.
By Proposition~\ref{prop:asbl-coasbl}, $\coasbl$ is a \piiso.
Since we are assuming that $\spec{X}=\res_{1\leq G}\spec{Y}$, Theorem~\ref{thm:Q}\ref{i:res-Q} and Example~\ref{eg:ind-res} give isomorphisms
\[
\map\bigl(G_+,\Q[h]{\res\CU}(\spec{X})\bigr)
\cong
\map\bigl(G_+,\res\Q[h]{\CU}(\spec{Y})\bigr)
\cong
\map\bigl(G_+,\Q[h]{\CU}(\spec{Y})\bigr)
,
\]
and the latter spectrum is a $G$-$\Omega$-spectrum by Lemma~\ref{lem:Omega}\ref{i:map-Omega} and Theorem~\ref{thm:Q}\ref{i:Omega}.
Then it follows from Lemma~\ref{lem:Omega-fix}\ref{i:pi-iso-Omega-fix} that also $\coasbl^G$ is a \piiso.

Using the fact that homotopy groups commute with finite products, one checks that the composition $\two \circ \one$ induces the right vertical map in~\eqref{eq:diagram-special}.
\end{proof}


\vfill


\end{document}